\documentclass[psamsfont]{amsart}
\usepackage[utf8]{inputenc}
\usepackage{graphicx}
\usepackage[dvips]{epsfig}
\usepackage{pinlabel}
\usepackage{amsmath}
\usepackage{amsfonts}
\usepackage{latexsym}
\usepackage{amssymb}
\usepackage[usenames,dvipsnames]{color}
\usepackage{amsthm}
\usepackage[all]{xypic}
\usepackage{enumitem}
\usepackage[breaklinks=true]{hyperref}

\input xy
\xyoption{all}

\author{Baptiste Chantraine}
\author{Georgios Dimitroglou Rizell}
\author{Paolo Ghiggini}
\author{Roman Golovko}

\address{Universit\'e de Nantes, France.}
\email{baptiste.chantraine@univ-nantes.fr}

\address{Uppsala University, Sweden.}
\email{georgios.dimitroglou@math.uu.se}

\address{Universit\'e de Nantes, France.}
\email{paolo.ghiggini@univ-nantes.fr}

\address{Charles University, Czech Republic.}
\email{golovko@karlin.mff.cuni.cz}

\theoremstyle{plain}
\newtheorem{Thm}{Theorem}[section]
\newtheorem{Rem}[Thm]{Remark}
\newtheorem{Prop}[Thm]{Proposition}
\newtheorem{Lem}[Thm]{Lemma}
\newtheorem{Cor}[Thm]{Corollary}
\theoremstyle{remark}
\newtheorem{defn}[Thm]{Definition}

\newtheorem{Quest}[Thm]{Question}
\newtheorem{Ex}[Thm]{Example}
\def\co{\colon\thinspace}

\DeclareMathAlphabet{\mathdj}{U}{msb}{m}{n}
\newcommand{\R}{\ensuremath{\mathdj{R}}}
\newcommand{\F}{\ensuremath{\mathdj{F}}}
\newcommand{\Z}{\ensuremath{\mathdj{Z}}}

\newcommand{\C}{\ensuremath{\mathdj{C}}}

\newcommand{\Cth}{\operatorname{Cth}}

\newcommand{\gr}{\operatorname{gr}}

\newcommand{\tr}{\operatorname{tr}}
\newcommand{\id}{\operatorname{Id}}
\newcommand{\im}{\operatorname{im}}

\newcommand{\OP}{\operatorname}

\begin{document}
\title{Floer theory for Lagrangian cobordisms}
\thispagestyle{empty}
\thanks{The first author is partially supported by the ANR project COSPIN (ANR-13-JS01-0008-01). The second author is supported by the grant KAW 2013.0321 from the Knut and Alice Wallenberg Foundation. The fourth author is supported  by the ESF Short Visit Grant, by the ERC Advanced Grant ``LDTBud'' and by the ERC Consolidator Grant 646649 ``SymplecticEinstein''.}

\maketitle
\begin{abstract}
In this article we define intersection Floer homology for exact Lagrangian cobordisms
between Legendrian submanifolds in the contactisation of a Liouville manifold, provided
that  the Chekanov-Eliashberg algebras of the negative ends of the cobordisms admit augmentations. From this theory we derive several long exact sequences relating the Morse homology of an exact Lagrangian cobordism with the bilinearised contact homologies of its ends. These are then used to investigate the topological properties of exact Lagrangian cobordisms.
\end{abstract}

\markboth{Chantraine, Dimitroglou Rizell, Ghiggini, Golovko}{Floer theory for
Lagrangian cobordisms}

\section{Introduction}

Lagrangian cobordism is a natural relation between Legendrian submanifolds, and it is a crucial ingredient in the definition of the functorial properties of invariants of Legendrian submanifolds in the spirit of symplectic field theory of Eliashberg, Givental and Hofer \cite{Eliashberg_&_SFT}.
In the present paper we study rigidity phenomena in the topology of  exact Lagrangian cobordisms in the symplectisation of the contactisation of a Liouville manifold.  In \cite{Eliash_Mur_Caps}, Eliashberg and Murphy showed that  exact Lagrangian cobordisms are flexible when their negative ends are loose (in the sense of Murphy \cite{Murphy_loose}). On the contrary, we will show that they become rigid if their negative ends admit augmentations (or more generally finite-dimensional representations) of their Chekanov-Eliashberg algebras.

In order to  study the topology of such cobordisms, we introduce a version of Lagrangian Floer homology 
for pairs of exact Lagrangian cobordisms. This construction finds its inspiration in the work of Ekholm in \cite{Ekholm_FloerlagCOnt}, which gives a symplectic field theory point of view on wrapped Floer homology of Abouzaid and Seidel from \cite{WrappedFuk}.

The definition of this new Floer theory requires the use of augmentations of the Chekanov-Eliashberg algebras of the negative ends as bounding cochains in order to algebraically cancel certain degenerations of the holomorphic curves at the negative ends of the cobordisms. Bounding cochains have been introduced, in the closed case, by Fukaya, Oh, Ohta and Ono in \cite{fooo}, while augmentations, which play a similar role in the context of Legendrian contact homology, have been introduced by Chekanov in \cite{Chekanov_DGA_Legendrian}.

For a pair of exact Lagrangian cobordisms obtained by a suitable small Hamiltonian push-off, our construction gives rise to various long exact sequences relating the singular homology of the cobordism with the Legendrian contact homology of its ends. We then use these long exact sequences to give restrictions on the topology of exact Lagrangian cobordisms under various hypotheses on the topology of the Legendrian ends.
Analogous long exact sequences have previously been found by Sabloff and Traynor in \cite{Sabloff_Traynor} in the setting of generating family homology  under the additional assumption that the cobordism admits a compatible generating family, and by the fourth author in \cite{GolovkoLagCob} in the case when the negative end of the cobordism admits an exact Lagrangian filling. The latter results have been put in a much more general framework in recent work by Cieliebak and Oancea \cite{SymplecticEilenbergSteenrod}.

We will assume that the reader is familiar with Legendrian contact homology, defined in
  \cite{Chekanov_DGA_Legendrian}, \cite{Ekholm_Contact_Homology} and \cite{LCHgeneral}. See also Etnyre's excellent survey \cite{Etnyre_Legendrian_Transver}
for a quick and relatively painless introduction to the topic.

The notion of Lagrangian cobordism between Legendrian submanifolds studied in this article is (in general) different from the notion of Lagrangian cobordisms  between Lagrangian submanifolds introduced by Arnol'd in \cite{Arnold_Lagrange_Cobordism1} and recently popularised by Biran and Cornea
in \cite{BiranCornea_LCI} and \cite{BirCo2}.
Despite the differences, for Lagrangian cobordisms between Legendrian submanifolds with no Reeb chords, some of the results we obtain resemble some of the results
obtained by Biran and Cornea \cite{BiranCornea_LCI, BirCo2} and Su\'arez \cite{suarezthesis}.

\subsection{Main results.}
\label{sec:main-result}

Let  $(P, \theta)$ be a Liouville manifold and $(Y, \alpha):= (P \times \R, dz+ \theta)$ its contactisation. We consider a pair of exact Lagrangian embeddings $\Sigma_0,\Sigma_1\hookrightarrow X$, where $(X, \omega) = (\R\times Y, d(e^t \alpha))$ is the symplectisation of $(Y, \alpha)$. We assume that the positive and negative ends of $\Sigma_i$ $i=0,1$ are cylindrical over  Legendrian submanifolds $\Lambda_i^-$ and $\Lambda_i^+$ respectively, and thus $\Sigma_i$ is a Lagrangian cobordisms from $\Lambda_i^-$ to $\Lambda_i^+$; see Figure \ref{fig:paircob} for a schematic representation. We assume that $\Sigma_0$ and $\Sigma_1$ intersect transversely and that their Legendrian ends are chord-generic in the sense of \cite{LCHgeneral}.
\begin{figure}[htp]
\vspace{3mm}
\labellist \pinlabel $t$ at 5 150
\pinlabel $\color{red}{\Lambda^+_1}$ at 38 154
\pinlabel $\Lambda^+_0$ at 120 154
\pinlabel $\Sigma_0$ at 136 98
\pinlabel $\color{red}{\Sigma_1}$ at 30 80
\pinlabel $\color{red}{\Lambda^-_1}$ at 38 4
\pinlabel $\Lambda^-_0$ at 116 4
\endlabellist
\centering
\includegraphics{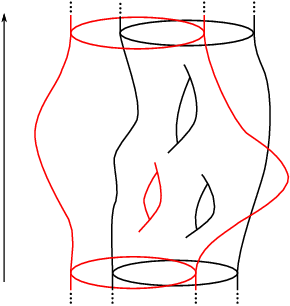}
\vspace{3mm}
\caption{Two Lagrangian cobordisms inside a symplectisation $\R \times Y$, where the vertical axis corresponds to the $\R$-coordinate.}
\label{fig:paircob}
\end{figure}
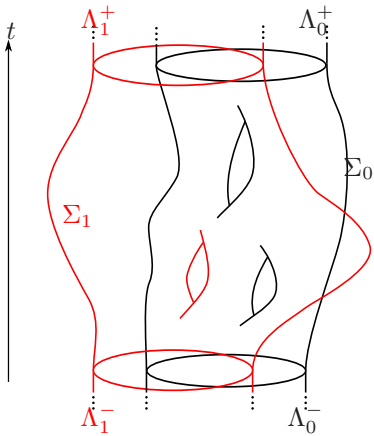

We denote by $\mathcal{R}(\Lambda_i^\pm)$ the set of Reeb chords of $\Lambda_i^\pm$  for $i=0,1$, and by $\mathcal{R}(\Lambda_1^\pm,\Lambda_0^\pm)$ the set of Reeb chords from $\Lambda_1^\pm$ to $\Lambda_0^\pm$. Let $R$ be a ring of characteristic $2$ or, if all $\Sigma_i$'s and $\Lambda_i^\pm$'s are relatively Pin, any  ring. (See Section \ref{sec:remarks-about-grad}.)  We denote by $C(\Lambda_0^\pm,\Lambda_1^\pm)$ the free $R$-module generated by $\mathcal{R}(\Lambda_1^\pm,\Lambda_0^\pm)$. Note that we are not assuming that $R$ is commutative, but we assume it is unital.

We assume that the Chekanov-Eliashberg algebra $\mathcal{A}(\Lambda^-_i;R)$ of $\Lambda^-_i$ admits an augmentation $\varepsilon^-_i$ over $R$ for $i=0,1$. It follows from the results of Ekholm, Honda and K\'alm\'an in \cite{Ekhoka} that $\mathcal{A}(\Lambda^+_i;R)$ also admits an augmentation $\varepsilon^+_i= \varepsilon^-_i \circ \Phi_{\Sigma_i}$, where  $\Phi_{\Sigma_i} \co \mathcal{A}(\Lambda^+_i;R) \to \mathcal{A}(\Lambda^-_i;R)$ is the unital DGA morphism induced by the cobordism $\Sigma_i$.  Thus the bilinearised
contact cohomologies $LCH_{\varepsilon_0^\pm, \varepsilon_1^\pm}(\Lambda_0^\pm, \Lambda_1^\pm)$ are defined. See Chekanov \cite{Chekanov_DGA_Legendrian} and Bourgeois and Chantraine \cite{augcat} for the notions of linearisation and bilinearisation of a differential graded algebra.

We denote by $CF(\Sigma_0,\Sigma_1)$ the free $R$-module spanned by the intersection points $\Sigma_0\cap\Sigma_1$. Note that, in general, it is not possible to define a Floer differential on $CF(\Sigma_0,\Sigma_1)$ because of breakings at the negative ends. % In Section~\ref{sec:action-energy} we define the action of an intersection point and use it to filter  $CF(\Sigma_0,\Sigma_1)$: we denote by $CF_\pm(\Sigma_0,\Sigma_1)$ the submodule of $CF(\Sigma_0,\Sigma_1)$ generated by intersection points of positive (respectively, negative) action.
We define a chain complex $(\Cth(\Sigma_0,\Sigma_1),\mathfrak{d}_{\varepsilon_0^-,\varepsilon_1^-})$ associated to a pair of Lagrangian cobordisms which we call the \emph{Cthulhu complex} (see Section \ref{sec:Cthulhu-complex}). Its underlying $R$-module is
$$\Cth(\Sigma_0,\Sigma_1) = C(\Lambda_0^+, \Lambda_1^+)  \oplus CF_+(\Sigma_0, \Sigma_1) \oplus C(\Lambda_0^-, \Lambda_1^-) \oplus CF_-(\Sigma_0, \Sigma_1).$$
The Cthulhu complex is acyclic because  of its invariance properties with respect to a large class of Hamiltonian deformations which, in the contactisation of a Liouville manifold, allow us to displace any pair of Lagrangian cobordisms.

When the negative ends are empty, this complex recovers the wrapped Floer cohomology complex in the form described by Ekholm in \cite{Ekholm_FloerlagCOnt}. When the positive ends are empty and there are no homotopically trivial Reeb chords of both $\Lambda_i^-$'s, this complex is similar to the Floer complex sketched in the work of Akaho in \cite[Section 8]{Akaho}. However, the later condition cannot be satisfied in the symplectisation of a contactisation of a Liouville manifold by Corollary \ref{cor:nonemptypos}; see also \cite{Caps}.

\subsubsection{Remarks about grading and orientation.}
\label{sec:remarks-about-grad} In order to define a graded theory,
we need that $2c_1(P)=0$ and that all Lagrangian cobordisms have
vanishing Maslov class. This implies that all Lagrangian cobordisms
admit Maslov potentials; a particular choice of such a potential
leads to the notion of a \emph{graded} Lagrangian cobordisms, for
which $Cth(\Sigma_0,\Sigma_1)$ has a well-defined grading in $\Z$.
See \cite{Seidel_Graded} for the closed case, which is similar.
In general the grading must be
taken in a (possibly trivial) cyclic group. Most of the result here
are stated for graded Lagrangian cobordisms, but our methods
apply to the ungraded cases as well. The only difference is that the
long exact sequences in Section \ref{sec:les} become exact triangles.

For the results to hold when the coefficient ring $R$ is of characteristic different from two, which is crucial for several of the applications in Section \ref{sec:cobobs}, we need to define coherent orientations for the relevant moduli spaces of pseudoholomorphic curves. This can be done in the case when the Lagrangian cobordisms are relatively Pin (following Ekholm, Etnyre and Sullivan in \cite{Ekholm_&_Orientation_Homology} and Seidel in \cite[Section 11]{Seidel_Fukaya}). For a precise treatment of signs, we refer to the recent work of Karlsson \cite{CoherentOrientationsLagrangianCobordisms}.

\subsection{Long exact sequences for LCH induced by a Lagrangian cobordism.}
\label{sec:les}

If $\Sigma_1$ is a Hamiltonian deformation of $\Sigma_0$ for some suitable and sufficiently small Hamiltonian, there is a well defined Floer differential on $CF(\Sigma_0, \Sigma_1)$ and the Floer homology group $HF(\Sigma_0, \Sigma_1)$ can be identified with the Morse homology group of $\Sigma_0$. Similarly, the bilinearised Legendrian contact homology groups $LCH_{\varepsilon^\pm_0,\varepsilon_1^\pm} (\Lambda^\pm_0,\Lambda^\pm_1)$ can be identified with the bilinearised contact homology groups $LCH_{\varepsilon^\pm_0,\varepsilon_1^\pm}(\Lambda^\pm_0)$ (as defined in \cite{augcat}) following \cite{Duality_EkholmetAl}. Moreover, in the same situation, the Cthulhu complex can be interpreted as a double cone, and thus provides long exact sequences which can be reinterpreted, by the identifications discussed above, as exact sequences relating the singular homology of a Lagrangian cobordism and the Legendrian contact homology of its ends. These results are proved in  Section \ref{sec:small-pert-lagr}.

In the rest of this introduction, $\Lambda^+$ and $\Lambda^-$ will always denote closed Legendrian submanifolds of dimension $n$ in the contactisation of a Liouville manifold, and every Lagrangian cobordism between them, as well as any Lagrangian filling of them, will always live in the corresponding symplectisation. We will denote by $\overline{\Sigma}$ the natural compactification of $\Sigma$ obtained by adjoining its Legendrian ends $\Lambda_\pm$. Note that $\overline{\Sigma}$ is diffeomorphic to $\Sigma \cap [-T, +T] \times Y$ for some $T \gg 0$ sufficiently large. We will also use the notation $\partial_\pm\overline{\Sigma}:=\Lambda_\pm \subset \overline{\Sigma}$, which implies that $\partial \overline{\Sigma}=\partial_+\overline{\Sigma} \sqcup \partial_-\overline{\Sigma}$.

\subsubsection{A generalisation of the long exact sequence of a pair.}

The first exact sequence we produce from a Lagrangian cobordism (see Section \ref{sec: proof of exact sequences}) is given by the following:

\begin{Thm}\label{thm:lespair}
  Let $\Sigma$ be a graded exact Lagrangian cobordism from $\Lambda^-$ to $\Lambda^+$ and let $\varepsilon^-_0$ and $\varepsilon^-_1$ be two augmentations of $\mathcal{A}(\Lambda^-)$ inducing augmentations $\varepsilon_0^+$, $\varepsilon_1^+$ of $\mathcal{A}(\Lambda^+)$. There is a long exact sequence
\begin{equation}
\label{leqtr}
\xymatrixrowsep{0.15in}
\xymatrixcolsep{0.15in}
\xymatrix{
       \cdots\ar[r]& LCH^{k-1}_{\varepsilon^+_0,\varepsilon_1^+}(\Lambda^+) \ar[d] & & & \\
& H_{n+1-k}(\overline{\Sigma},\partial_- \overline{\Sigma};R) \ar[r] & LCH^{k}_{\varepsilon^-_0,\varepsilon^-_1}(\Lambda^-) \ar[r] & LCH^{k}_{\varepsilon^+_0,\varepsilon_1^+}(\Lambda^+)\ar[r] &\cdots,}
\end{equation}
where the map $\Phi^{\varepsilon^-_0,\varepsilon^-_1}_\Sigma \colon LCH^{k}_{\varepsilon^-_0,\varepsilon^-_1}(\Lambda^-) \to LCH^{k}_{\varepsilon^+_0,\varepsilon_1^+}(\Lambda^+)$ is the adjoint of the bilinearised DGA morphism $\Phi_\Sigma$ induced by $\Sigma$ (see \cite{Ekhoka}).
\end{Thm}

When the negative end $\Lambda^-=\emptyset$ is empty, i.e.~when $\Sigma$ is an exact Lagrangian filling of $\Lambda^+$, and $\varepsilon^+_i$, $i=0,1$ both are augmentations induced by this filling, the resulting long exact sequence simply becomes the isomorphism
\[ LCH^{k-1}_{\varepsilon^+_0,\varepsilon_1^+}(\Lambda^+) \xrightarrow \cong H_{n+1-k}(\overline{\Sigma};R) \]
appearing in the work of Ekholm in \cite{Ekholm_FloerlagCOnt}. This isomorphism was first observed by Seidel, and is sometimes called \emph{Seidel's isomorphism}. (See the map $G^{\varepsilon^-_0,\varepsilon^-_1}_\Sigma$ in Section  \ref{sec:seidel} for another incarnation.) Its proof was completed by the second author in \cite{LiftingPseudoholomorphic}; also see \cite{Floer_Conc} for an analogous isomorphism induced by a \emph{pair} of fillings.

\subsubsection{A generalisation of the duality long exact sequence and fundamental class}
\label{sec:dualityles}
A Legendrian submanifold $\Lambda$ is horizontally displaceable if there exists a Hamiltonian isotopy $\phi_t$ of $(P,d\theta)$ which displaces the Lagrangian projection $\Pi_{\OP{Lag}}(\Lambda) \subset P$ from itself. In Section \ref{sec: proof of exact sequences} we obtain the following:

\begin{Thm}\label{thm:lesduality}
Let $\Sigma$ be an exact graded Lagrangian cobordism from $\Lambda^-$ to $\Lambda^+$ and let $\varepsilon^-_0$ and $\varepsilon^-_1$ be two augmentations of $\mathcal{A}(\Lambda^-)$ inducing augmentations $\varepsilon_0^+$, $\varepsilon_1^+$ of $\mathcal{A}(\Lambda^+)$. Assume that $\Lambda^-$ is horizontally displaceable; then there is a long exact sequence
  \begin{equation}
      \label{ledual}
\xymatrixrowsep{0.15in}
\xymatrixcolsep{0.15in}
       \xymatrix{
\cdots \ar[r] & LCH^{k}_{\varepsilon^+_0,\varepsilon_1^+}(\Lambda^+) \ar[r] &  LCH_{n-k-1}^{\varepsilon^-_0,\varepsilon^-_1}(\Lambda^-)\ar[r] & H_{n-k-1}(\Sigma;R) \ar[d]
& \\
& & & LCH^{k+1}_{\varepsilon^+_0,\varepsilon_1^+}(\Lambda^+)\ar[r] & \cdots,
}
\end{equation}
where the map $G^{\varepsilon^-_0,\varepsilon^-_1}_\Sigma \colon H_{n-k-1}(\Sigma;R)\to LCH^{k+1}_{\varepsilon^+_0,\varepsilon_1^+}(\Lambda^+)$ is defined in Section \ref{sec:seidel}.
\end{Thm}

If $\Sigma=\R \times \Lambda$, then $H_\bullet(\Sigma)=H_\bullet(\Lambda)$, and hence the above long exact sequence recovers the duality long exact sequence for Legendrian contact homology, which was proved by Sabloff in \cite{Sabloff_Duality} for Legendrian knots and later generalised to arbitrary Legendrian submanifolds in \cite{Duality_EkholmetAl} by Ekholm, Etnyre and Sabloff. In the bilinearised setting, the duality long exact sequence was introduced by Bourgeois and the first author in \cite{augcat}. In Section \ref{sec:fund-class-twist} we use Exact Sequence \eqref{ledual} to prove that the fundamental class in LCH defined by Sabloff in \cite{Sabloff_Duality} and Ekholm, Etnyre and Sabloff in \cite{Duality_EkholmetAl} is functorial with respect to the maps induced by exact Lagrangian cobordisms.

\subsubsection{A generalisation of the Mayer-Vietoris long exact sequence}

The last exact sequence that we will extract from the Cthulhu complex generalises the Mayer-Vietoris exact sequence (see Section \ref{sec:push-inducing-mayer}).

\begin{Thm}\label{thm:lesmayer-vietoris}
Let $\Sigma$ be an exact graded Lagrangian cobordism from $\Lambda^-$ to $\Lambda^+$ and let $\varepsilon^-_0$ and $\varepsilon^-_1$ be two augmentations of $\mathcal{A}(\Lambda^-)$ inducing augmentations $\varepsilon_0^+$, $\varepsilon_1^+$ of $\mathcal{A}(\Lambda^+)$. Then there is a long exact sequence
\begin{equation}
\label{eq:mayer-vietoris}
\xymatrixrowsep{0.15in}
\xymatrixcolsep{0.15in}
\xymatrix{
 & \cdots \ar[r] & LCH^{k-1}_{\varepsilon^+_0,\varepsilon_1^+}(\Lambda^+) \ar[dl] &    \\
& H_{n-k}(\partial_- \overline{\Sigma};R) \ar[r] & LCH^{k}_{\varepsilon^-_0,\varepsilon^-_1}(\Lambda^-) \oplus  H_{n-k}(\overline{\Sigma};R) \ar[r] & LCH^{k}_{\varepsilon^+_0,\varepsilon_1^+}(\Lambda^+)\ar[dl],\\
& & \cdots &  }
\end{equation}
where the component
\[ H_{n-k}(\partial_- \overline{\Sigma};R) \to H_{n-k}(\overline{\Sigma};R) \]
of the left map is induced by the topological inclusion of the negative end.

If $\varepsilon^-_0=\varepsilon^-_1=\varepsilon$, it moreover follows that the image of the fundamental class under the component $H_n(\partial_- \overline{\Sigma};R) \to LCH^{0}_{\varepsilon,\varepsilon}(\Lambda^-)$ of the above morphism vanishes. Moreover, under the additional assumption that $\Lambda^-$ is horizontally displaceable, the image of a generator under $H_0(\partial_- \overline{\Sigma};R) \to LCH^{n}_{\varepsilon,\varepsilon}(\Lambda^-)$ is equal to the fundamental class in Legendrian contact homology.
\end{Thm}
In particular we get that the fundamental class in $H_{n}(\partial_- \overline{\Sigma};R)$ either is non-zero in $H_n(\overline{\Sigma})$, or is the image of a class in $LCH^{-1}_{\varepsilon_0^+,\varepsilon_1^+}(\Lambda^+)$. In both cases, $\Lambda^+\not=\emptyset$.  Thus we obtain a new proof of the following result.
\begin{Cor}[\cite{Caps}]
\label{cor:nonemptypos}
If $\Lambda \subset P \times \R$ admits an augmentation, then there is no exact Lagrangian cobordism from $\Lambda$ to $\emptyset$, i.e.~there is no exact Lagrangian ``cap'' of $\Lambda$.
\end{Cor}

\begin{Rem} Assume that $\Lambda_-$ admits an exact Lagrangian filling $L$ inside the symplectisation, and that $\varepsilon^-$ is the augmentation induced by this filling. It follows that $\varepsilon^+$ is the augmentation induced by the filling $L \odot \Sigma$ of $\Lambda_+$ obtained as the concatenation of $L$ and $\Sigma$.  Using  Seidel's isomorphisms
$$LCH^{k}_{\varepsilon^-,\varepsilon^-}(\Lambda^-) \cong H_{n-k}(L;R), \quad
LCH^{k}_{\varepsilon^+,\varepsilon^+}(\Lambda^+) \cong H_{n-k}(L\odot \Sigma;R)$$
to replace the relevant terms in the long exact sequences \eqref{leqtr} and \eqref{eq:mayer-vietoris}, we obtain the long exact sequence for the pair $(L \odot \Sigma, L)$ and the Mayer-Vietoris long exact sequence for the decomposition $L\odot \Sigma = L \cup \Sigma$, respectively.  This fact was already observed and used by the fourth author in \cite{GolovkoLagCob}.
\end{Rem}

\subsection{Topological restrictions on Lagrangian cobordisms}
\label{sec:cobobs}
Using the long exact sequences from the previous subsection and their refinements to coefficients twisted by the fundamental group, as defined in Section \ref{sec:l2-legendr-cont}, we find strong topological restrictions on exact Lagrangian cobordisms between certain classes of Legendrian submanifolds.
\subsubsection{The homology of an exact Lagrangian cobordism from a Legendrian submanifold to itself}
One of the consequences of Theorem~\ref{thm:lesmayer-vietoris} is
the following theorem, proved in Section \ref{sec:topol-endoc}. A similar statement
has been proven by the second and the fourth author in \cite[Theorem 1.6]{Rigidityofendo} under the more restrictive assumption that $\Lambda$ bounds an exact Lagrangian filling.

\begin{Thm} \label{homrigidityold}
Let $\Sigma$ be an exact Lagrangian cobordism from $\Lambda$ to $\Lambda$ and $\F$ a field (of characteristic two if $\Lambda$ is not relatively Pin).
If the Chekanov-Eliashberg algebra ${\mathcal A}(\Lambda; \F)$ admits an augmentation, then:
\begin{itemize}
\item[(i)] There is an equality $\dim_{\F} H_\bullet(\Sigma;\F) = \dim_{\F} H_\bullet(\Lambda;\F)$;
\item[(ii)] The map
%\begin{align*}
$(i^-_{\ast}, i^+_{\ast}): H_{\bullet}(\Lambda; \F)\to
H_{\bullet}(\Sigma; \F)\oplus H_{\bullet}(\Sigma; \F)$
%\end{align*}
is injective; and
\item[(iii)] The map
$i^+_* \oplus i^-_* \colon H_\bullet (\Lambda \sqcup \Lambda) \to H_\bullet(\Sigma)$
is surjective.
\end{itemize}
Here $i^{+}$ is the inclusion of $\Lambda$ as the positive end of $\Sigma$, while $i^{-}$ is the inclusion of $\Lambda$ as the negative end of $\Sigma$.
\end{Thm}
\begin{Rem}
The above equalities hold for the $\Z$-graded singular homology groups without assuming that the cobordism $\Sigma$ is graded.

\end{Rem}

An immediate corollary of Theorem \ref{thm:lespair} is the following
result, which had already appeared in \cite[Theorem 1.7]{Rigidityofendo}
under the stronger assumption that the negative end is fillable.

\begin{Thm}\label{thm:homologycylinder}
  If $\Lambda$ is a homology sphere which admits an augmentation over $\mathbb{Z}$, then any exact Lagrangian cobordism $\Sigma$ from $\Lambda$ to itself is a homology cylinder (i.e. $H_\bullet(\Sigma,\Lambda)=0$).
\end{Thm}

Inspired by the work of Capovilla-Searle and Traynor \cite{NonorLagcobbetLegknots},
in Section \ref{sec:nonorientable} we prove the following restriction on the characteristic classes of an exact Lagrangian cobordism from a Legendrian submanifold to itself. Given a manifold $M$, we denote by $w_i(M)$ the $i$-th Stiefel-Whitney class of $TM$.
\begin{Thm}\label{thm:w_ivanish}
  Let $\Sigma$ be an exact Lagrangian cobordism from $\Lambda$ to itself, and $\F = \Z / 2 \Z$. Assume that ${\mathcal A}(\Lambda; \F)$ admits an augmentation. If, for some $i\in\mathbb{N}$, $w_i(\Lambda)=0$, then $w_i(\Sigma)=0$.

If $\Lambda$ is relatively Pin, the same holds for the Pontryagin classes and for the Maslov class.
\end{Thm}
In particular we partially answer Question~6.1 of the same article.
\begin{Cor}\label{cor:nonorien}
  If $\Lambda$ is an orientable Legendrian submanifold admitting an augmentation, then any exact Lagrangian cobordism from $\Lambda$ to itself is orientable.
\end{Cor}

\subsubsection{Restrictions on the fundamental group of certain exact Lagrangian fillings and cobordisms}

In order to incorporate the fundamental group in our constructions, following ideas of Sullivan in \cite{KFloer} and Damian in \cite{Damian_Lifted}, we define a  ``twisted'' version of the Cthulhu complex $Cth(\Sigma_0, \Sigma_1)$  with coefficients in the group ring $R[\pi_1(\Sigma_0)]$ in Section \ref{sec:l2-legendr-cont}.

We also establish long exact sequences analogous to those in Section \ref{sec:les} involving homology groups over twisted coefficients in $R[\pi_1(\Sigma)]$. In the setting of Legendrian contact homology, these techniques were introduced by Ekholm and Smith in \cite{EkholmSmith} and further developed by Eriksson-\"Ostman in \cite{Albin}.

Using generalisations of the long exact sequence from Theorem~\ref{thm:lespair} and the functoriality of the fundamental class  from Proposition~\ref{prp:fundclasstwisted} (see Section \ref{sec:proofpi_1carclass}) we prove the following theorem:

\begin{Thm}\label{thm:pi_1carclass}
Let $\Sigma$ be a graded exact Lagrangian cobordism from $\Lambda^-$ to $\Lambda^+$. Assume that $\mathcal{A}(\Lambda^-;R)$ admits an augmentation and that $\Lambda^+$ has no Reeb chords in degree zero. If $\Lambda^-$ and $\Lambda^+$ both are simply connected, then $\Sigma$ is simply connected as well.
\end{Thm}

\begin{Rem}
\label{rem:uniqueaug}
The seemingly unnatural condition that $\Lambda_+$ has no Reeb chords in degree zero is used to ensure that the Chekanov-Eliashberg algebra $\mathcal{A}(\Lambda^+; A)$ has \emph{at most} one augmentation in $A$ for every unital $R$-algebra $A$.

This condition is clearly not invariant under Legendrian isotopy, but the conclusion of Theorem~\ref{thm:pi_1carclass} can be extended to every Legendrian submanifold which is Legendrian isotopic to $\Lambda^+$ because Legendrian isotopies induce Lagrangian cylinders by \cite[4.2.5]{Eliashberg_&_Lagrangian_Intersection_Finite} (also, see \cite{chantraine_conc}). 
\end{Rem}
We now present another result which imposes constraints on the fundamental group of an exact Lagrangian cobordism from a Legendrian submanifold to itself (see Section \ref{sec:proofl2rigidity}). Its proof
uses an $L^2$-completion of the Floer homology groups with twisted coefficients and the $L^2$-Betti numbers of the universal cover (using results of Cheeger and Gromov in \cite{CheeGro}).

\begin{Thm}\label{thm:l2rigidity}
Let $\Lambda$ be a simply connected Legendrian submanifold which is relatively Pin, and let $\Sigma$ be an exact Lagrangian cobordism from $\Lambda$ to itself. If $\mathcal{A}(\Lambda;\C)$ admits an augmentation, then $\Sigma$ is simply connected as well.
\end{Thm}

Combining Theorem \ref{thm:homologycylinder} with
Theorem \ref{thm:l2rigidity}, we get the following result.
\begin{Cor}\label{cor:trivhspheres}
Let $\Sigma$ be an $n$-dimensional Legendrian homotopy sphere and assume that
${\mathcal A}(\Lambda; \Z)$ admits an augmentation. Then any exact Lagrangian cobordism $\Sigma$ from $\Lambda$ to itself is an $h$-cobordism. In particular:
\begin{enumerate}
\item If $n \neq 3,4$, then $\Sigma$ is diffeomorphic to a cylinder;
\item If $n=3$, then $\Sigma$ is homeomorphic to a cylinder; and
\item If $n=4$ and $\Lambda$ is diffeomorphic to $S^4$, then $\Sigma$ is diffeomorphic to a cylinder.
\end{enumerate}
\end{Cor}

When $n=1$, a stronger result is known. Namely, in \cite[Section 4]{Floer_Conc} we proved that any exact Lagrangian cobordism $\Sigma$ from the standard Legendrian unknot $\Lambda_0$ to itself is  compactly supported Hamiltonian isotopic to the trace of a Legendrian isotopy of $\Lambda_0$ which is induced by the complexification of a rotation by $k\pi$, $k \in \Z$. This classification makes use of the uniqueness of the exact Lagrangian filling of $\Lambda_0$ up to compactly supported Hamiltonian isotopy, which was proved in \cite{Eliashberg_&_Local_Lagrangian_knots} by Eliashberg and Polterovich. In contrast, the methods we develop in this article give restrictions only on the smooth type of the cobordisms and little information is known about their symplectic knottedness in higher dimension.

\subsubsection{Obstructions to the existence of a Lagrangian concordance}
A {\em Lagrangian concordance} from $\Lambda^-$ to $\Lambda^+$ is a symplectic cobordism from $\Lambda^-$ to $\Lambda^+$ which is diffeomorphic to a product $\R \times \Lambda$. In particular this implies that $\Lambda^-$ and $\Lambda^+$ are diffeomorphic as smooth manifolds. Note that a Lagrangian concordance is automatically exact.

If $\Sigma$ is a Lagrangian concordance, then $H_{\bullet}(\overline{\Sigma},\partial_- \overline{\Sigma};R)=0$, and thus Theorem \ref{thm:lespair} implies the following corollary.

\begin{Cor}\label{cor:concobstruction}
Let $\Sigma$ be an exact Lagrangian concordance from $\Lambda^-$ to $\Lambda^+$.  If, for $i=0,1$, $\varepsilon^-_i$ is an augmentation of $\mathcal{A}(\Lambda^-; R)$ and $\varepsilon^+_i$ is the pull-back of $\varepsilon^-_i$ under the DGA morphism induced by $\Sigma$, then the map $$\Phi_\Sigma^{\varepsilon_0^-,\varepsilon_1^-} \colon LCH^\bullet_{\varepsilon^-_0,\varepsilon_1^-}(\Lambda^-)\rightarrow LCH^\bullet_{\varepsilon^+_0,\varepsilon_1^+}(\Lambda^+)$$
is an isomorphism. Consequently, there is an inclusion
\[ \{LCH^\bullet_{\varepsilon^-_0,\varepsilon_1^-}(\Lambda^-) \}/\text{isom.} \:\:  \hookrightarrow  \{LCH^\bullet_{\varepsilon^+_0,\varepsilon_1^+}(\Lambda^+) \}/\text{isom.} \]
of the sets consisting of isomorphism classes of bilinearised Legendrian contact cohomologies, for all possible pairs of augmentations.
\end{Cor}
This corollary can be used to obstruct the existence of Lagrangian concordances. For example, it can be applied to the computation of the linearised Legendrian contact homologies given by Chekanov in \cite[Theorem 5.8]{Chekanov_DGA_Legendrian} to prove that there is no exact Lagrangian concordance from either of the two Chekanov-Eliashberg knots to the other. We also use Corollary~\ref{cor:concobstruction} to deduce new examples of non-symmetric concordances in the spirit of the example given by the first author in \cite{Chantraine_Non_symmetry}. We refer to Section \ref{sec:noninv} for a simply connected example in high dimensions.

We recall that a Legendrian isotopy induces a Lagrangian concordance. Since Legendrian isotopies are invertible, two isotopic Legendrian submanifolds thus admit Lagrangian concordances going in either direction. On the other hand, we have now many examples of non-symmetric Lagrangian concordances, and hence the following natural question can be asked.
\begin{Quest}
Assume that there exists Lagrangian concordances from $\Lambda_0$ to $\Lambda_1$ as well as from $\Lambda_1$ to $\Lambda_0$. Does this imply that the Legendrian submanifolds $\Lambda_0$ and $\Lambda_1$ are Legendrian isotopic? Are such Lagrangian concordances moreover Hamiltonian isotopic to one induced by a Legendrian isotopy (as constructed by \cite[4.2.5]{Eliashberg_&_Lagrangian_Intersection_Finite})?
\end{Quest}
In view of Corollary~\ref{cor:concobstruction}, this question will not be easily answered by Legendrian contact homology.

\subsection{Remarks about the hypotheses.}
\label{sec:remarks-about-hypoth}

\subsubsection{Restrictions on the ambient manifolds.}
\label{sec:restriction-y}

The reasons for restricting our attention to Lagrangian cobordisms in the symplectisation of the contactisation of a Liouville manifold are two-fold. First, the analytic framework to have a well defined complex $(\Cth(\Sigma_0,\Sigma_1),\mathfrak{d}_{\varepsilon_0^-,\varepsilon_1^-})$ is vastly simplified from the fact that the Reeb flow has no periodic Reeb orbits. However, using recent work of Pardon in \cite{Pardon_CH} (or the polyfold technology being developed by  Hofer, Wysocki and Zehnder), we expect that it should be possible to extend the construction of the complex $(\Cth(\Sigma_0,\Sigma_1),\mathfrak{d}_{\varepsilon_0^-,\varepsilon_1^-})$ to more general symplectic cobordisms.
Second, our applications use exact sequences arising from the acyclicity of the complex $(\Cth(\Sigma_0,\Sigma_1),\mathfrak{d}_{\varepsilon_0^-,\varepsilon_1^-})$, which is a consequence of the fact that that any Lagrangian cobordism can be displaced in the symplectisation of a contactisation. Floer theory for Lagrangian cobordisms in more general symplectic cobordisms will be investigated in a future article.

\subsubsection{Restrictions on the Lagrangian submanifolds.}
\label{sec:restriction-}

Now we describe some examples showing that many of the hypotheses we made in Section~\ref{sec:cobobs} are in fact essential, and not merely artefacts of the techniques used.
First, an exact Lagrangian cobordism having a negative end whose Chekanov-Eliashberg algebra admits no augmentation can be a quite flexible object: in fact Eliashberg and Murphy proved in
\cite{Eliash_Mur_Caps} that exact Lagrangian cobordisms with a loose negative end satisfy an $h$-principle, and therefore one cannot hope for a result in the spirit of Theorem \ref{thm:homologycylinder} to hold in complete generality. Indeed, we refer to the work of the second and fourth authors in \cite{Rigidityofendo} for examples of exact Lagrangian cobordisms from a loose Legendrian sphere to itself having arbitrarily large Betti numbers.

Second, the condition that $\Lambda$ is a homology sphere in the statement of Theorem \ref{thm:homologycylinder} was shown to be essential already in \cite[Section 2.3]{Rigidityofendo}.

Finally, the importance of the condition on the Reeb chords of the positive end in Theorem \ref{thm:pi_1carclass} is emphasised by the following example, which will
be detailed in Section~\ref{sec:some-expl-lagr}.
\begin{Prop}\label{prop:example}
There exists a non-simply connected exact Lagrangian cobordism from the two-dimensional standard Legendrian sphere to a Legendrian sphere inside the symplectisation of standard contact $(\R^5,\xi_{\OP{std}})$.
\end{Prop}
As a converse to Theorem \ref{thm:pi_1carclass}, the existence of a non-simply connected exact Lagrangian cobordism can be used to show the existence of degree zero Reeb chords on the positive end of the cobordism.

\section*{Acknowledgements}
\label{sec:aknowledgments}
While this research was conducted, the authors benefited from the hospitality of various institutions in addition to their own; in particular they are indebted to CIRGET and CRM in Montr\'{e}al, the Centro De Giorgi in Pisa, and the institute Mittag-Leffler in Stockholm. We warmly thank Stefan Friedl for suggesting that we should look at $L^2$-homology theory to gain information on fundamental groups.

\section{Preliminary definitions}
\label{sec:preliminaries}

\subsection{Exact Lagrangian cobordisms}
\label{sec:exact-lagr-cobord}

Let $(P, \theta)$ be a Liouville manifold.  Its contactisation $(Y, \alpha)$ is $Y= P \times \mathbb{R}$ with the contact form $\alpha = dz+\theta$. We consider exact Lagrangian cobordisms in the contactisation of a Liouville manifolds $(P,\theta)$. The following convention will be used throughout the article: given a smooth manifold $M$, a submanifold $N \subset M$ and a differential form $\eta$ on $M$, we will denote by $\eta|_N$ the pull-back of $\eta$ under the inclusion of $N$ in $M$.

\begin{defn}\label{defn: exact lagrangian cobordism}
 Let $\Lambda^{-}$ and $\Lambda^{+}$ be two
closed Legendrian submanifolds of $(Y,\alpha)$. An \emph{exact Lagrangian cobordism from $\Lambda^-$ to $\Lambda^+$} in $(\R \times Y, d(e^t \alpha))$ is a properly embedded submanifold $\Sigma \subset \R\times Y$ without boundary satisfying the following conditions:
\begin{enumerate}
\item for some $T\gg 0$,
\begin{itemize}
  \item[(a)] $ \Sigma \cap ( (-\infty,-T) \times Y)= (-\infty, -T)\times \Lambda^{-}$,
  \item[(b)] $ \Sigma \cap ((T,+\infty) \times Y) = (T,+\infty)\times \Lambda^{+}$, and
  \item[(c)] $ \Sigma \cap ([-T,T] \times Y)$ is compact;
  \end{itemize}
\item There exists a smooth function $f_\Sigma:\Sigma\to \R$ for which
\begin{itemize}
\item[(a)] $e^{t}\alpha|_{\Sigma} = df_\Sigma$,
\item[(b)] $f_\Sigma|_{(-\infty,-T)\times \Lambda^{-}}$ is constant, and
\item[(c)] $f_\Sigma|_{(T,\infty)\times \Lambda^{+}}$ is constant.
\end{itemize}
\end{enumerate}
We will call $(T,+\infty)\times \Lambda_+ \subset \Sigma$ and $(-\infty,-T) \times \Lambda_- \subset \Sigma$ the \emph{positive end} and the \emph{negative end} of $\Sigma$, respectively. We will call a cobordism from a submanifold to itself an \textit{endocobordism}.
\end{defn}

Condition (2b) is used to rule out certain bad breakings of pseudoholomorphic curves. Condition (2c) is used to ensure that the concatenation of two exact Lagrangian cobordisms (as in Definition \ref{def:concat}) still is an exact Lagrangian cobordism. If one does not care about concatenations, then this condition can be dropped.
\begin{Ex}
If $\Lambda$ is a closed Legendrian submanifold of $(Y, \xi)$, then $\R \times \Lambda$
is an exact Lagrangian cobordism inside $(\R \times Y, d(e^t\alpha))$ from $\Lambda$ to itself. Cobordisms of this
type are called {\em (trivial) Lagrangian cylinders}.
\end{Ex}
In the case when there exists an exact Lagrangian cobordism from $\Lambda^-$ to $\Lambda^+$ we say that \emph{$\Lambda^-$ is exact Lagrangian cobordant to $\Lambda^+$}. If $\Sigma$ is an exact Lagrangian cobordism from the empty set to $\Lambda$, we call $\Sigma$ an \textit{exact Lagrangian filling} of $\Lambda$. In the latter case we say that $\Lambda$ is {\em exactly fillable}.

\begin{defn}\label{def:concat}
Given exact Lagrangian cobordisms $\Sigma_a$ from $\Lambda^-$ to $\Lambda$ and $\Sigma_b$ from $\Lambda$ to $\Lambda^+$, their \emph{concatenation}
$\Sigma_a \odot \Sigma_b$
is defined as follows.

 First, translate $\Sigma_a$ and $\Sigma_b$ so that
\begin{align*}
& \Sigma_a \cap ((- 1, + \infty) \times Y) = (-1,+\infty) \times \Lambda,\\
& \Sigma_b \cap ((- \infty, 1) \times Y) = (- \infty, 1) \times \Lambda.
\end{align*}
Then we define
\[ \Sigma_a \odot \Sigma_b := (\Sigma_a \cap ((- \infty, 0] \times Y)) \cup (\Sigma_b
\cap ([0, + \infty) \times Y)).\]
\end{defn}
Conditions (2b) and (2c) of Definition \ref{defn: exact lagrangian cobordism} imply that $\Sigma_a \odot \Sigma_b$ is an exact Lagrangian cobordism from $\Lambda^-$ to $\Lambda^+$.

The following Lemma follows from the fact that different choices of translation lead to Hamiltonian isotopic exact cobordisms.

\begin{Lem}\label{noia}
The compactly supported Hamiltonian isotopy class of $\Sigma_a \odot \Sigma_b$ is independent of the above choices of translations.
\end{Lem}
\subsection{Almost complex structures}
\label{sec:holomorphic-curves}
In this subsection we describe the almost complex structures that we will use to set up the theory.

\subsubsection{Cylindrical almost complex structures}
\label{sec:cylindrical}

Let $(Y,\alpha)$ be a contact manifold with the choice of a contact form. We denote by $\mathcal{J}^{\OP{cyl}}(Y)$ the set of \emph{cylindrical} almost complex structures on the symplectisation $(\mathbb{R}\times Y, d(e^t\alpha))$, i.e.~almost complex structures $J$ satisfying the following conditions:
\begin{itemize}
\item $J$ is invariant under the natural action of $\mathbb{R}$ on $\mathbb{R}\times Y$;
\item $J\partial_t=R_\alpha$;
\item $J(\xi)=\xi$, where $\xi :=\ker \alpha \subset TY$; and
\item $J$ is compatible with $d \alpha|_{\xi}$, i.e.~$d\alpha_{\xi}(\cdot,J \cdot)$ is a metric on $\xi$.
\end{itemize}

\subsubsection{Almost complex structures on Liouville manifolds.}
\label{sec:almost-compl-struct}

Let $(P,\theta)$ be a Liouville manifold.  Recall that there is a subset $P_\infty \subset P$ that is exact symplectomorphic to half a symplectisation
\[ ([0,+\infty) \times V,d(e^{\tau}\alpha_V)),\]
and where $P \setminus P_\infty \subset P$ is pre-compact. We say that an almost complex structure $J_P$ compatible with $d\theta$ is \textit{admissible} if the almost complex structure $J_P$ on $P$ coincides with a cylindrical almost complex structure in $\mathcal{J}^{\OP{cyl}}(V)$ outside of a compact subset of $P_\infty$. We denote by $\mathcal{J}^{\OP{adm}}(P)$ the set of these almost complex structures.

\subsubsection{Cylindrical lifts}
\label{sec:cyllift}
Given an almost complex structure $J_P \in \mathcal{J}^{\OP{adm}}(P)$ there is a unique cylindrical almost complex structure $\widetilde{J}_P$ on $(\R \times Y, d(e^t\alpha))$ which makes the projection
\[ \pi \colon \R \times Y = \R \times P \times \R \to P \]
a $(\widetilde{J}_P,J_P)$-holomorphic map. We will call this almost complex structure the \emph{cylindrical lift of $J_P$}.

We denote by $\mathcal{J}^{\OP{cyl}}_\pi(Y) \subset \mathcal{J}^{\OP{cyl}}(Y)$ the set of cylindrical lifts of almost complex structures in $\mathcal{J}^{\OP{adm}}(P)$.

\subsubsection{Compatible almost complex structures with cylindrical ends.}
\label{sec:admissible}
Let $J^+$ and $J^-$ be almost complex structures in $\mathcal{J}_{\pi}^{\OP{cyl}}(Y)$, and let $T\in\mathbb{R}^+$. We require that $J^+$ and $J^-$ are lifts of
almost complex structures in $\mathcal{J}^{\OP{adm}}(P)$ which coincide outside of a
compact subset of $P$. We denote by $\mathcal{J}_{J^-,J^+,T}^{\OP{adm}}(X)$ the set of almost complex structures $J$ on $X=\mathbb{R}\times Y$ that tame $d(e^t\alpha)$ and satisfy the following:
\begin{enumerate}[label={(\alph*)}]
\item $J$ is equal to the cylindrical almost complex structures $J^-$ and $J^+$ on subsets of the form
%\begin{gather*}
$$(-\infty,-T] \times P \times \R, \quad
[T,+\infty) \times P \times \R,$$
%\end{gather*}
respectively; and
\item for some compact $K \subset Y$, the almost complex structure $J$ coincides, outside of $\R \times K$, with a cylindrical lift living in $\mathcal{J}^{\OP{cyl}}_\pi(Y)$.
\end{enumerate}
Condition (b) is needed in order to deal with compactness issues.

We will always assume that $T>0$ is such that the cobordisms are cylindrical outside $[-T, T] \times Y$, and simply write $\mathcal{J}^{\OP{adm}}_{J^-,J^+}(X)$. The union of all $\mathcal{J}^{\OP{adm}}_{J^-,J^+}(X)$ over all $J^-,J^+\in \mathcal{J}^{\OP{cyl}}_{\pi}(Y)$ is denoted $\mathcal{J}^{\OP{adm}}(X)$; almost complex structure in this set will be called \textit{admissible}.

\section{A monster compendium of holomorphic discs}
\label{sec: monster compendium}

Through the paper we will consider various moduli spaces of holomorphic curves with boundary on a pair of Lagrangian cobordisms. Asymptotics will be either intersection points or Reeb chords. Such moduli spaces are similar to many already appearing in the literature (see \cite{RSFT}, \cite{Ekholm_FloerlagCOnt}, \cite{EffectLegendrian}) and compactness results for such moduli spaces follows from the results in \cite{Bourgeois_&_Compactness} and \cite{Abbas_book}. In this section we describe these spaces and express their dimension in terms of the Conley-Zender indices of their asymptotics (using the dimension formula from \cite[Proposition 11.13]{Seidel_Fukaya}). We will follow the SFT convention: in all moduli spaces we take a quotient by reparametrisations of the source.

We fix two lifted almost complex structures $J^\pm \in {\mathcal J}^{\OP{cyl}}_{\pi}(Y)$ on $\R \times Y$ and a path $J_\bullet= \{J_t\}_{t \in [0,1]}$ of almost complex structures in ${\mathcal J}^{\OP{adm}}_{J^+,J^-}(X)$ which is constant near $t=0,1$. We denote by ${\mathcal R}(\Lambda_i^{\pm}, \Lambda_j^\pm)$, $\{i,j \} = \{ 0,1 \}$ the set of Reeb chords from $\Lambda_i^{\pm}$ to $\Lambda_j^\pm$. Those chords are called {\em mixed}, while chords starting and ending on the same Legendrian are called {\em pure}.
\subsection{Pure moduli spaces.}
\label{sec:pure-lch-moduli-spaces}

Denote $\Sigma = \Sigma_i$, $J=J_i$ and $\Lambda^\pm = \Lambda_i^\pm$ for one of $i=0$ or $1$. Let $\gamma^+, \delta^+_i, \ldots, \delta^+_d$ be Reeb chords of $\Lambda^+$ and $\gamma_-, \delta^-_1,\ldots,\delta^-_d$  Reeb chords of $\Lambda^-$. Throughout the paper we will write a $d$-tuple of pure Reeb chords as a word; this notation is reminiscent of the multiplicative structure of the Chekanov-Eliashberg algebra. Therefore, we set
$\boldsymbol{\delta}^\pm= \delta_1^\pm \ldots \delta_d^\pm$.

We will consider three types of pure LCH moduli spaces:
$$\widetilde{M}_{\mathbb{R}\times\Lambda^+}(\gamma^+;\boldsymbol{\delta}^+; J^+), \quad \widetilde{M}_{\mathbb{R}\times\Lambda^-} (\gamma^-;\boldsymbol{\delta}^-; J^-), \quad \text{and} \quad \mathcal{M}_{\Sigma}(\gamma^+;\boldsymbol{\delta}^-; J).$$
The first two moduli spaces will be called pure {\em cylindrical} LCH moduli spaces,  and are used to define the Legendrian contact homology differential of $\Lambda^\pm$.  The third moduli space will be called pure {\em cobordism} LCH moduli space and is used in \cite{Ekhoka} to define maps between Legendrian contact homology algebras.

The pure cylindrical LCH moduli spaces are moduli spaces of $J^\pm$-holomorphic maps from a $(d+1)$-punctured disc ($d \ge 1$) to $\R \times Y$ with boundary on $\R \times \Lambda^\pm$, one puncture positively asymptotic to $\gamma^\pm$ and $d$ punctures negatively asymptotic to $\boldsymbol{\delta}^\pm$. The pure cobordism LCH moduli spaces are moduli spaces of $J$-holomorphic maps from a $(d+1)$-punctured disc ($d \ge 1$) to $\R \times Y$ with boundary on $\Sigma$, one positive puncture asymptotic to $\gamma^+$ and $d$ negative punctures asymptotic to $\boldsymbol{\delta}^-$.
Finally in the cylindrical moduli spaces we take a quotient by the natural $\R$-action, while such an operation is not possible (nor desirable) for the cobordism moduli space.

\subsection{Mixed moduli spaces}\label{subsec: mixed moduli spaces}
\subsubsection{General definition}
Let $\underline{L}=(L_0, L_1)$, where $\{ L_0, L_1 \}$ denotes one of  the sets $\{ \R \times \Lambda^\pm_0,  \R \times \Lambda^\pm_1 \}$ or $\{ \Sigma_0, \Sigma_1 \}$. In the former case define $J_{\underline{L}}=J^\pm$ and in the latter $J_{\underline{L}}=J_\bullet$. Let
$\boldsymbol{\delta}$ be a word in the Reeb cords of the negative end of $L_0$ and $\boldsymbol{\zeta}$ a word in the Reeb cords of the negative end of $L_1$. Finally, compatibly with $\underline{L}$, let $x^\pm$ be either an intersection point or a mixed Reeb chord, requiring that, when $x^+$ is a Reeb chord, it should go from the positive ends of $L_1$ to the positive end of $L_0$.

We denote by ${\mathcal M}_{\underline{L}}(x^+; \boldsymbol{\delta}, x^-, \boldsymbol{\zeta}; J_{\underline{L}})$ the moduli space consisting of $J_{\underline{L}}$-holomorphic maps from a boundary punctures strip (with coordinates $(t,s) \in [0,1] \times \R$) to $\R \times Y$ with boundary on $L_0$ for $t=0$ and on  $L_1$ for $t=1$, asymptotic to $\gamma^+$ for $s \to + \infty$ and to $\gamma^-$ for $s \to - \infty$, and with boundary punctures negatively asymptotic to $\boldsymbol{\delta}$ and $\boldsymbol{\zeta}$.
If $\{ L_0, L_1 \} = \{ \R \times \Lambda^\pm_0, \R \times \Lambda^\pm_1 \}$, we denote by
$\widetilde{\mathcal M}_{\underline{L}}(x^+; \boldsymbol{\delta}, x^-, \boldsymbol{\zeta}; J_{\underline{L}})$ the quotient of ${\mathcal M}_{\underline{L}}(x^+; \boldsymbol{\delta}, x^-, \boldsymbol{\zeta}; J_{\underline{L}})$ by the natural $\R$-action.

\begin{Rem}
When no confusion can arise, we will often drop the almost complex structure, and sometimes
also the boundary conditions, from the notation of the moduli spaces. In the latter case, the boundary conditions can be inferred from the asymptotics.
\end{Rem}
\subsubsection{Mixed LCH moduli spaces.}\label{sec:mixed-lch-moduli-spaces}
Let $\boldsymbol{\delta}^\pm:=\delta^\pm_1\ldots\delta^\pm_{i-1}$ be Reeb chords of $\Lambda^\pm_1$ and $\boldsymbol{\zeta}^\pm=\zeta^\pm_{i+1}\ldots\zeta^\pm_d$ be Reeb chords of $\Lambda_0^\pm$.
We will consider three types of mixed LCH moduli spaces:
$$\widetilde{\mathcal{M}}_{\mathbb{R}\times\Lambda_0^\pm,\mathbb{R}\times\Lambda_1^\pm} (\gamma^+; \boldsymbol{\delta}^\pm, \gamma^-, \boldsymbol{\zeta}^\pm; J^\pm) \quad \text{and} \quad \mathcal{M}_{\Sigma_0,\Sigma_1} (\gamma^+; \boldsymbol{\delta}^-, \gamma^-, \boldsymbol{\zeta}^-; J_\bullet),$$
where
\begin{itemize}
\item[] $\gamma^\pm \in \mathcal{R}(\Lambda^+_1,\Lambda^+_0)$ for $\widetilde{\mathcal{M}}_{\mathbb{R}\times\Lambda_0^+,\mathbb{R}\times\Lambda_1^+} (\gamma^+; \boldsymbol{\delta}^+, \gamma^-, \boldsymbol{\zeta}^+; J^+)$,
\item[] $\gamma^\pm \in \mathcal{R}(\Lambda^-_1,\Lambda^-_0)$ for $\widetilde{\mathcal{M}}_{\mathbb{R}\times\Lambda_0^-,\mathbb{R}\times\Lambda_1^-} (\gamma^+; \boldsymbol{\delta}^-, \gamma^-, \boldsymbol{\zeta}^-; J^-)$,
and
\item[] $\gamma^+ \in \mathcal{R}(\Lambda^+_1,\Lambda^+_0)$, $\gamma^- \in \mathcal{R}(\Lambda^-_1,\Lambda^-_0)$ for $\mathcal{M}_{\Sigma_0,\Sigma_1} (\gamma^+; \boldsymbol{\delta}^-, \gamma^-, \boldsymbol{\zeta}^-; J_\bullet)$.
\end{itemize}

The first two moduli spaces will be called mixed {\em cylindrical} LCH moduli spaces and are used to define the bilinearised Legendrian contact homology differential of $(\Lambda^\pm_0, \Lambda^\pm_1)$ (see \cite{augcat}).  The third moduli space will be called mixed {\em cobordism} LCH moduli space and is used to define maps between bilinearised Legendrian contact homology groups.

 An illustration of  a curve in the mixed cobordism LCH moduli spaces is shown in Figure \ref{fig:LCHmap}.

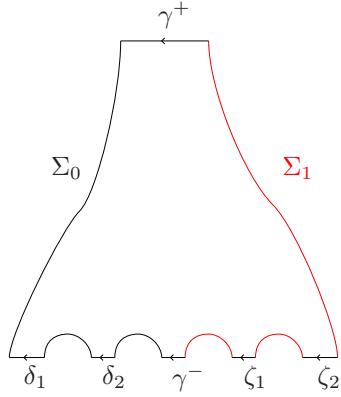
\begin{figure}[!h]
\centering
\vspace{0.5cm}
\labellist
\pinlabel $\gamma^+$ [bl] at 196 447
\pinlabel $\delta_1$ [tl] at 13 0
\pinlabel $\delta_2$ [tl] at 120 0
\pinlabel $\gamma^-$ [tl] at 215 2
\pinlabel $\zeta_1$ [tl] at 313 0
\pinlabel $\zeta_2$ [tl] at 414 0
\pinlabel $\Sigma_0$ [bl] at 50 240
\pinlabel ${\color{red}\Sigma_1}$ [bl] at 366 240
\endlabellist
\includegraphics[height=4.3cm]{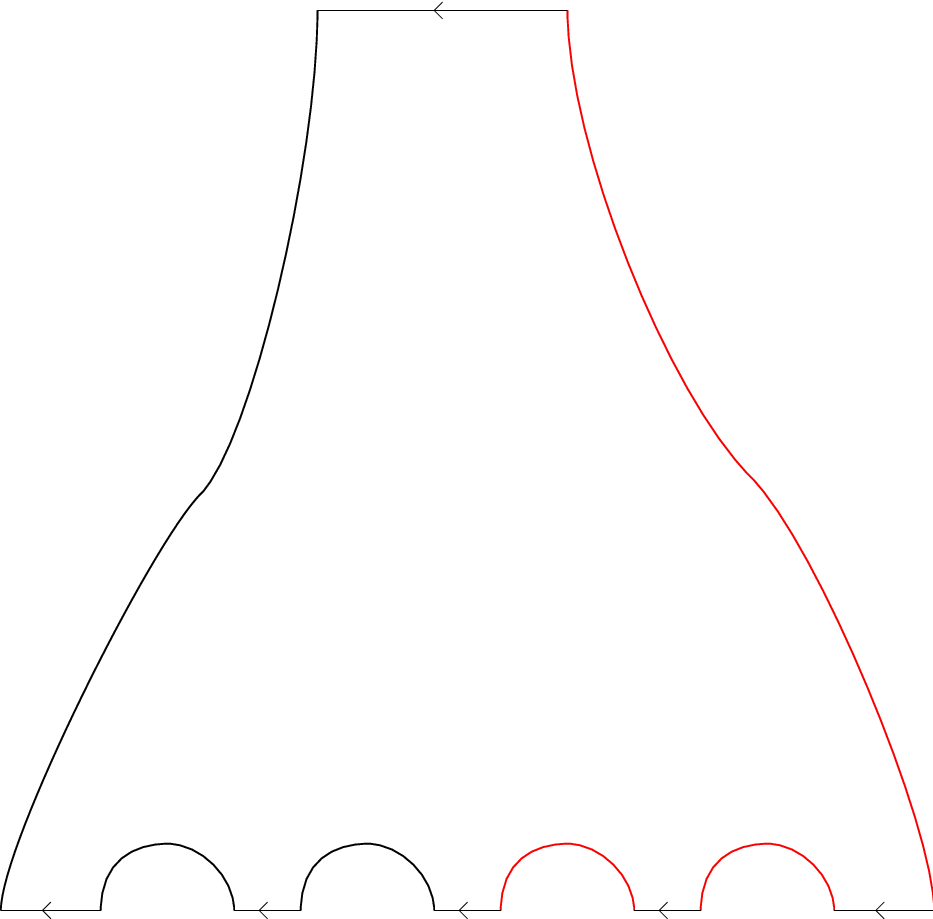}
  \caption{A mixed cobordism LCH curve.}
  \label{fig:LCHmap}
\end{figure}

\subsubsection{Floer moduli space.}
\label{sec:floer-moduli-space}
Let $p,q\in \Sigma_0\cap\Sigma_1$ be intersection points, $\boldsymbol{\delta}^-=
\delta^-_1\ldots\delta^-_{i-1}$ a word of Reeb chords on $\Lambda^-_0$, and $\boldsymbol{\zeta}^-= \zeta^-_{i+1} \ldots \zeta^-_d$ a word of Reeb chords on $\Lambda^-_1$.
Elements of the moduli spaces
$$\mathcal{M}_{\Sigma_0,\Sigma_1}(p;\boldsymbol{\delta}^-,q,\boldsymbol{\zeta}^-; J_\bullet)$$
will be called \textit{(punctured) Floer strips}.
See Figure \ref{fig:floerdiff}.

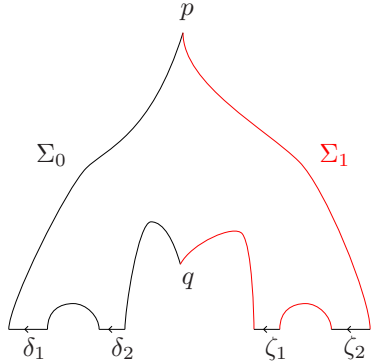
\begin{figure}[!h]
\centering
\vspace{0.5cm}
\labellist
\pinlabel $p$ [bl] at 207 382
\pinlabel $\delta_1$ [tl] at 11 1
\pinlabel $\delta_2$ [tl] at 121 1
\pinlabel $q$ [tl] at 209 82
\pinlabel $\zeta_1$ [tl] at 312 2
\pinlabel $\zeta_2$ [tl] at 410 3
\pinlabel $\Sigma_0$ [br] at 80 200
\pinlabel ${\color{red}\Sigma_1}$ [bl] at 380 200
\endlabellist
  \includegraphics[height=4cm]{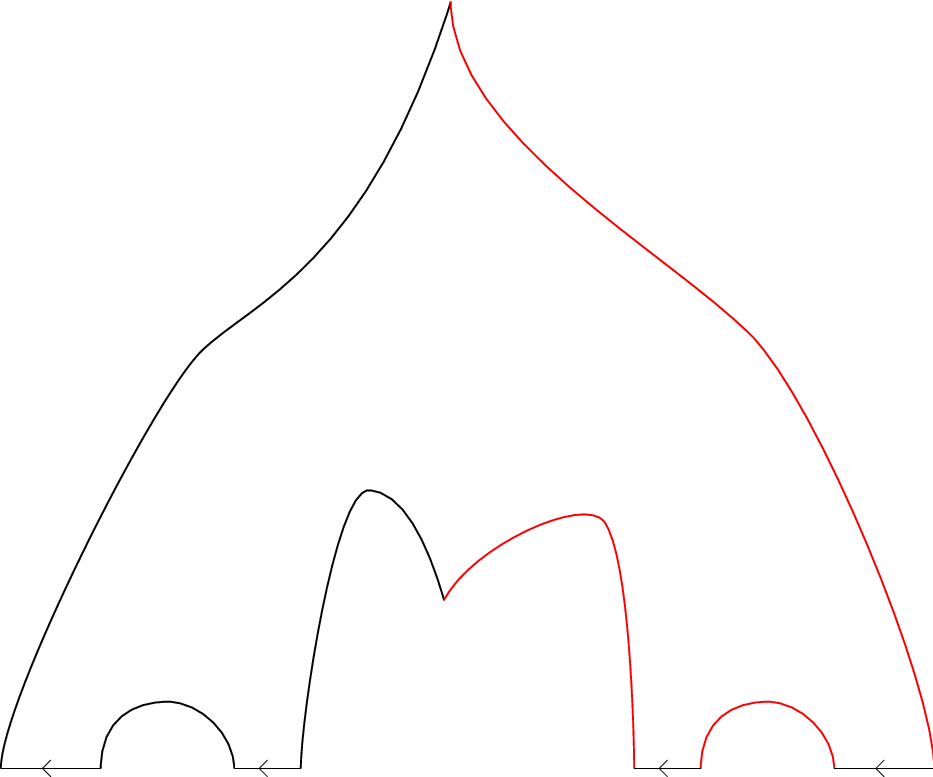}
  \caption{A punctured Floer strip.}
  \label{fig:floerdiff}
\end{figure}

\subsubsection{LCH to Floer moduli space.}
\label{sec:lch-floer-moduli}
Let $\gamma^-\in \mathcal{R}(\Lambda^-_1,\Lambda^-_0)$ be a mixed chord, $p\in \Sigma_0\cap\Sigma_1$ an intersection point, $\boldsymbol{\delta}^-=\delta^-_1\ldots\delta^-_{i-1}$ a word of Reeb chords on $\Lambda^-_0$, and $\boldsymbol{\zeta}^-=\zeta^-_{i+1} \ldots \zeta^-_d$ a word of Reeb chords on $\Lambda^-_1$. Curves in the moduli space $$\mathcal{M}_{\Sigma_0,\Sigma_1}(p;\boldsymbol{\delta}^-,\gamma^-,\boldsymbol{\zeta}^-; J_\bullet)$$ will be called \textit{pseudoholomorphic cthulhus}.
See Figure \ref{fig:cultist1}.

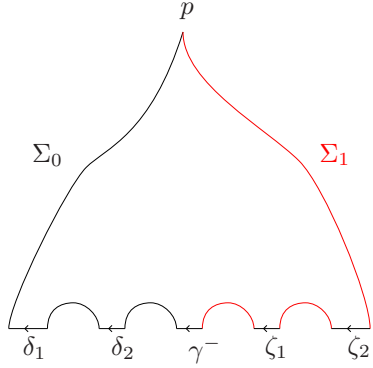
\begin{figure}[!h]
\centering
\vspace{0.5cm}
\labellist
\pinlabel $p$ [bl] at 207 382
\pinlabel $\delta_1$ [tl] at 11 1
\pinlabel $\delta_2$ [tl] at 121 1
\pinlabel $\gamma^-$ [tl] at 217 1
\pinlabel $\zeta_1$ [tl] at 312 2
\pinlabel $\zeta_2$ [tl] at 415 3
\pinlabel $\Sigma_0$ [br] at 75 200
\pinlabel ${\color{red}\Sigma_1}$ [bl] at 380 200
\endlabellist
  \includegraphics[height=4cm]{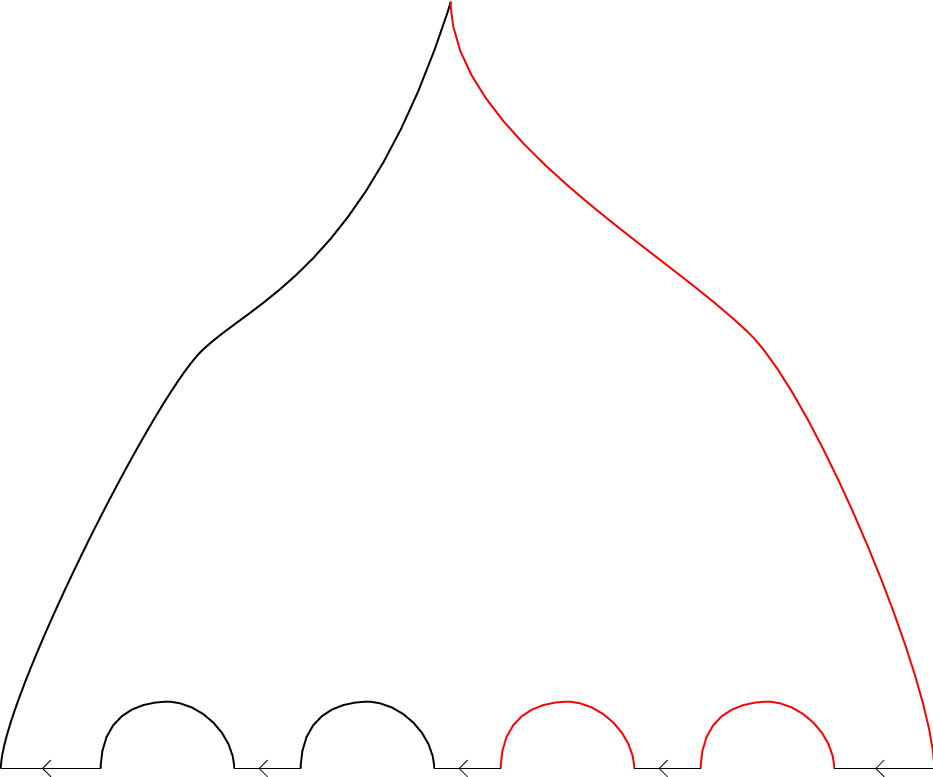}
  \caption{A pseudoholomorphic cthulhu.}
  \label{fig:cultist1}
\end{figure}

\subsubsection{Floer to LCH moduli space.}
\label{sec:floer-lch-moduli}

Let $\gamma^+\in \mathcal{R}(\Lambda^+_1,\Lambda^+_0)$ be a mixed chord,
$p\in \Sigma_0\cap\Sigma_1$ an intersection point,  $\boldsymbol{\delta}^-=\delta^-_1\ldots\delta^-_{i-1}$  Reeb chords of $\Lambda^-_0$ and $\boldsymbol{\zeta}^-=\zeta^-_{i+1} \ldots \zeta^-_d$  Reeb chords of $\Lambda_1^-$. Curves in the moduli space $$\mathcal{M}_{\Sigma_0,\Sigma_1}(\gamma^+; \boldsymbol{\delta}^-, p, \boldsymbol{\zeta}^-; J_\bullet)$$ will be called \textit{pseudoholomorphic cultists}. See Figure \ref{fig:cultist2}.

\begin{figure}[!h]
\centering
\vspace{0.5cm}
\labellist
\pinlabel $\gamma^+$ [bl] at 199 445
\pinlabel $\delta_1$ [tl] at 14 1
\pinlabel $\delta_2$ [tl] at 121 1
\pinlabel $q$ [tl] at 205 92
\pinlabel $\zeta_1$ [tl] at 312 2
\pinlabel $\zeta_2$ [tl] at 418 3
\pinlabel $\Sigma_0$ [br] at 90 215
\pinlabel ${\color{red}\Sigma_1}$ [bl] at 380 215
\endlabellist
  \includegraphics[height=4cm]{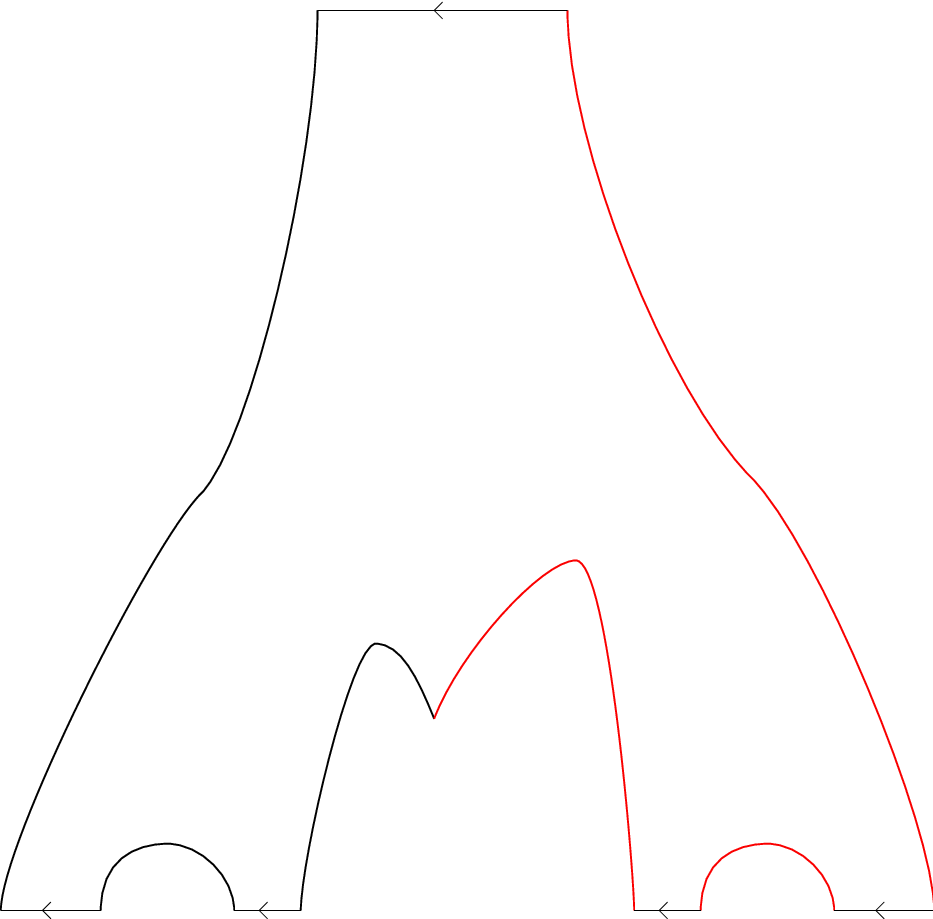}
  \caption{A pseudoholomorphic cultist}
  \label{fig:cultist2}
\end{figure}
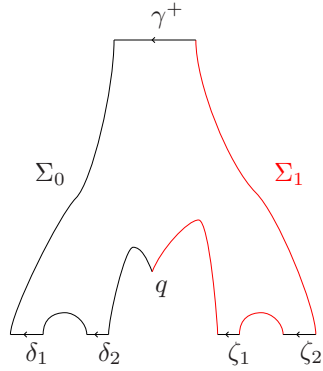

\subsubsection{Bananas moduli space.}
\label{sec:banana-moduli-space}

Let $\gamma_{1,0}\in \mathcal{R}(\Lambda^-_1,\Lambda^-_0)$ and $\gamma_{0,1} \in \mathcal{R}(\Lambda^-_0,\Lambda^-_1)$ be mixed Reeb chords (going in opposite directions), let $\boldsymbol{\delta}^- = \delta^-_1 \ldots \delta^-_{i-1}$ be Reeb chords of $\Lambda^-_0$ and $\boldsymbol{\zeta}^- =\zeta^-_{i+1} \ldots \zeta^-_d$ Reeb chords of $\Lambda_1^-$. Curves in the moduli space
$$\widetilde{\mathcal{M}}_{\R \times \Lambda^-_0, \R \times \Lambda^-_1}(\gamma_{1,0}; \boldsymbol{\delta}^-, \gamma_{0,1}, \boldsymbol{\zeta}^-; J^-)$$
will be called \textit{pseudoholomorphic bananas}. See Figure \ref{fig:banana}.

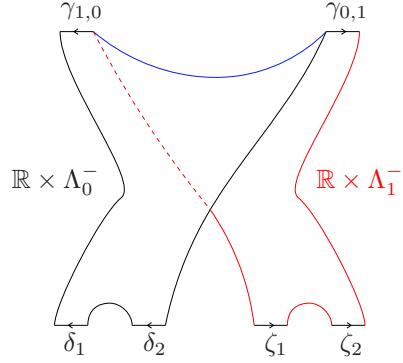
\begin{figure}[!h]
\centering
\vspace{0.5cm}
\labellist
\pinlabel $\gamma_{1,0}$ [bl] at 1 435
\pinlabel $\delta_1$ [tl] at 5 1
\pinlabel $\delta_2$ [tl] at 120 1
\pinlabel $\gamma_{0,1}$ [bl] at 384 435
\pinlabel $\zeta_1$ [tl] at 295 1
\pinlabel $\zeta_2$ [tl] at 404 1
\pinlabel $\R \times \Lambda^-_0$ [br] at 73 185
\pinlabel ${\color{red} \R \times \Lambda^-_1}$ [bl] at 370 185
\endlabellist
  \includegraphics[height=4cm]{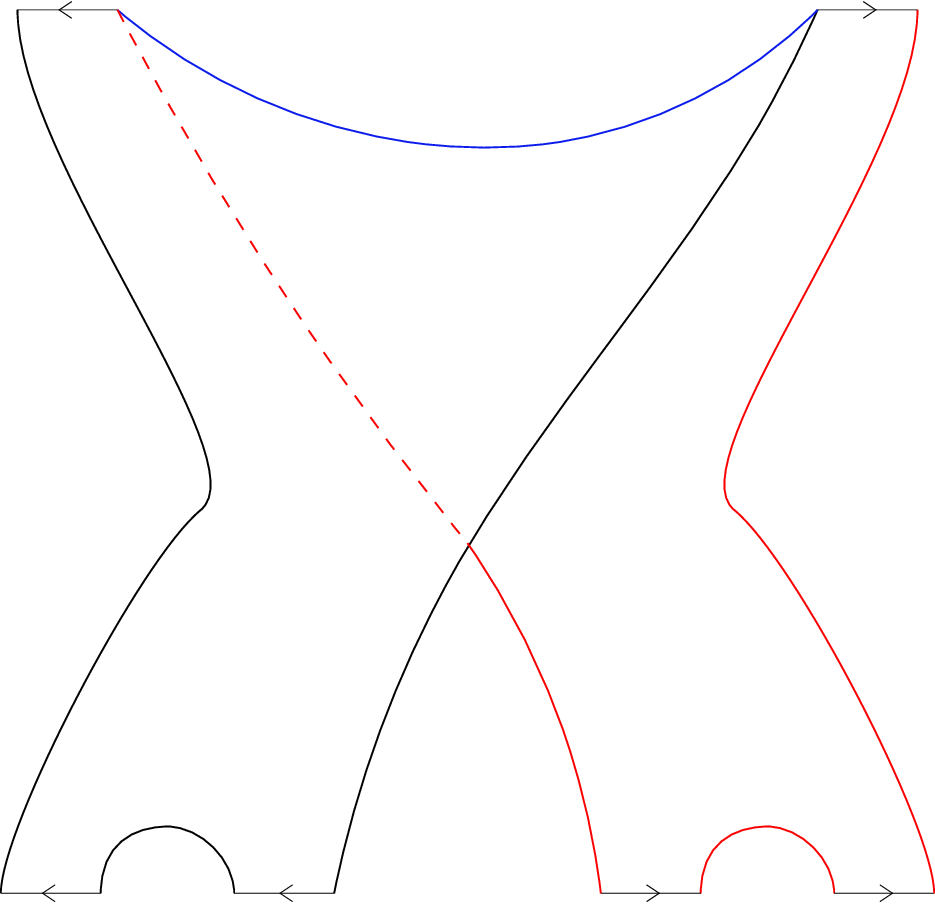}
  \caption{A pseudoholomorphic banana.}
  \label{fig:banana}
\end{figure}

Note that we can also define bananas moduli spaces
$$\widetilde{\mathcal M}_{\R \times \Lambda^+_0, \R \times \Lambda^+_1}(\gamma_{1,0}; \boldsymbol{\delta}^+, \gamma_{0,1}, \boldsymbol{\zeta}^+; J^+), \quad {\mathcal M}_{\Sigma_0, \Sigma_1}(\gamma_{1,0}; \boldsymbol{\delta}^-, \gamma_{0,1}, \boldsymbol{\zeta}^-; J_\bullet).$$

\subsection{Structure of the moduli spaces}
We recall that the cobordisms $\Sigma_i$ have dimension $n+1$. If $x$ is an intersection point or a Reeb chord (either pure or mixed), we denote by $\gr(x)$ its Conley-Zehnder index
as defined in \cite{LCHgeneral}.
If the cobordisms are graded, then $\gr$ is well defined; otherwise it depends on some (non-canonical) choice and only the dimension formulas for the moduli spaces will be well defined for maps in the same relative homotopy class. The proof of the following proposition is a patchwork of results from the literature, see \cite[Section~2.2]{FredholmTheory}, \cite[Theorem A.1]{CieliebakSwitching}, \cite[Lemma 2.5]{Duality_EkholmetAl}.
\begin{Prop}\label{prop: structure of moduli spaces}
For generic almost complex structures $J^\pm$ and $J_\bullet$, the moduli spaces described in Section~\ref{sec: monster compendium} are transversely cut out and therefore are smooth manifolds.  Their dimensions are
\begin{align}
  \label{eq:22}
\dim \mathcal{M}(\gamma^+; \boldsymbol{\delta}^-, \gamma^-, \boldsymbol{\zeta}^-) & = \gr(\gamma^+)-\gr(\gamma^-)-\gr(\boldsymbol{\delta})-\gr(\boldsymbol{\zeta}),\\
\label{eq:23}\dim \mathcal{M}(\gamma^+; \boldsymbol{\delta}^-, q, \boldsymbol{\zeta}^-) & = \gr(\gamma^+)- \gr(q)- \gr(\boldsymbol{\delta})- \gr(\boldsymbol{\zeta})+1,\\
\label{eq:24} \dim \mathcal{M}(p; \boldsymbol{\delta}^-, q ,\boldsymbol{\zeta}^-) & =\gr(p)-\gr(q)-\gr(\boldsymbol{\delta})-\gr(\boldsymbol{\zeta})-1,\\
\label{eq:25}\dim \mathcal{M}(p; \boldsymbol{\delta}^-, \gamma^-, \boldsymbol{\zeta}^-) & =\gr(p)- \gr(\gamma^-)- \gr(\boldsymbol{\delta})- \gr(\boldsymbol{\zeta})-2,\\
\label{eq:26} \dim \mathcal{M}(\gamma_{1,0}; \boldsymbol{\delta}^-, \gamma_{0,1}, \boldsymbol{\zeta}^-) & =\gr(\gamma_{1,0})+ \gr(\gamma_{0,1})- \gr(\boldsymbol{\delta})- \gr(\boldsymbol{\zeta})- n+2.
\end{align}
When the natural $\R$-action is well defined and free, the moduli spaces after taking the quotient are still transversely cut out and their dimension is one less. Moreover the zero-dimensional moduli spaces (after quotient, when possible) are compact, and the one-dimensional moduli spaces can be compactified by adding two-levels pseudoholomorphic buildings where each level belong to a zero-dimensional moduli space. Finally, if both cobordisms are relatively Pin, the moduli spaces can be coherently oriented.
\end{Prop}
\subsection{Energy}
\label{sec:energy-compactness}
In this section, we recall the notion of Hofer energy for pseudoholomorphic curves in the symplectisation of a contact manifold as introduced in \cite{Hofer93} and \cite{Bourgeois_&_Compactness}. See also \cite{Abbas_Chord_Asy} for the relative case.

\subsubsection{Hofer energy}
\label{sec:energy-1}

Assume that $\Sigma_0$ and $\Sigma_1$ are two exact Lagrangian cobordisms  in the symplectisation of $(Y,\alpha)$. Let $f_i \colon \Sigma_i \to \R$ be primitives of  $e^t\alpha\vert_{\Sigma_i}$ which are constant at the cylindrical ends. Without loss of generality we will assume that both constants are $0$ on the negative ends, while the constants on the positive end of $\Sigma_i$ will be denoted by $\mathfrak{c}_i$, $i=0,1$. Here we rely on Definition \ref{defn: exact lagrangian cobordism} of an exact cobordism.

Take any $T > 0$ and $T>\epsilon > 0$ for which
\begin{eqnarray*}
& & \Sigma_i \cap ((- \infty, -T + \varepsilon) \times Y) = (- \infty, -T + \varepsilon) \times \Lambda_i^-\\
& & \Sigma_i \cap (( T - \varepsilon, + \infty) \times Y) =  ( T - \varepsilon, + \infty) \times \Lambda_i^+,
\end{eqnarray*}
for $i=0,1$. Now, we let $\phi:\mathbb{R}\rightarrow [e^{-T},e^T]$ be a smooth function satisfying:
\begin{itemize}
\item $\phi(\pm t)=e^{\pm T}$ for $t>T$;
\item $\phi(t)=e^t$ for $t\in [-T+\epsilon,T-\epsilon]$;
\item $\phi'(t)\geq 0$.
\end{itemize}
In the case when both $\Sigma_0$ and $\Sigma_1$ are trivial cylinders over Legendrian submanifolds, we will also allow the case $T=\epsilon=0$, and $\phi \equiv 1$.

By construction we have $\phi\cdot\alpha\vert_{\Sigma_i} = e^t \alpha|_{\Sigma_i}$ for the reason that $\alpha|_{\Sigma_i}=0$ in the subset where $\phi$ is not equal to $e^t$. A primitive of $e^t \alpha|_{\Sigma_i}$ (which exists by exactness) is hence also a primitive of $\phi\cdot\alpha\vert_{\Sigma_i}$.

Let $\mathcal{C}^-$ be the set of compactly supported smooth functions
\[w_- \colon (-\infty,-T+\epsilon) \to [0,+\infty)\]
satisfying $\displaystyle\int_{-\infty}^{-T+\epsilon} w_-(s)ds=e^{-T}$, and let $\mathcal{C}^+$ be the set of compactly supported smooth functions
\[w_+ \colon (T-\epsilon,+\infty) \to [0,+\infty)\]
satisfying $\displaystyle\int_{T-\epsilon}^{+\infty}w_+(s)ds=e^T$.

\begin{defn}
Let $S$ be a punctured disc and  let $u=(a,v)\colon S \rightarrow \mathbb{R}\times Y$ be a smooth map.
\begin{itemize}
\item The \emph{$d(\phi\alpha)$-energy} of $u$ is given by
\begin{align*}
E_{d(\phi\alpha)}(u)=\displaystyle\int_Su^*(d(\phi\alpha)).
\end{align*}
\item The \emph{$\alpha$-energy of $u$} is given by
\begin{align*}
E_\alpha(u)=\sup\limits_{(w_-,w_+)\in \mathcal{C}^-\times\mathcal{C}^+}\left(\displaystyle\int_S(w_-\circ a) da\wedge v^*\alpha + \displaystyle\int_S(w_+\circ a) da\wedge v^*\alpha\right).
\end{align*}
\item The \emph{total energy}, or the \emph{Hofer energy}, of $u$ is given by
\begin{align*}
E(u)=E_{\alpha}(u)+E_{d(\phi\alpha)}(u).
\end{align*}
\end{itemize}
If $E(u)<\infty$, we say that $u$ is a \textit{finite energy pseudoholomorphic disc}.
\end{defn}

Non-constant holomorphic curves have positive total energy, as stated in the following simple lemma. We leave the proof to the reader.
\begin{Lem}
If $u$ is non-constant punctured pseudoholomorphic disc with boundary on a pair of exact Lagrangian cobordisms, and if the almost complex structure is cylindrical outside of $[-T+\epsilon,T-\epsilon] \times Y$, then $E(u) >0$, $E_\alpha(u) \ge 0$, and $E_{d(\phi\alpha)}(u) \ge 0$. Moreover, $E_{d(\phi\alpha)}(u) =0$ implies that the image of  $u$ is contained in a trivial cylinder over a Reeb orbit.
\end{Lem}

\subsubsection{Action and energy}
\label{sec:action-energy}
Consider a pair of exact Lagrangian cobordisms $\Sigma_i$ from $\Lambda^-_i$ to $\Lambda^+_i$ with potential functions $f_i \colon \Sigma_i \to \R$, $i=0,1$.
For a Reeb chord $c$ of $\Lambda^\pm_0 \cup \Lambda^\pm_1$ we define
\[ \ell(c) := \int_c \alpha.\]
Recall that the definition of the $E_{d(\phi\alpha)}$-energy depends on the choice of a constant $T \ge 0$, where equality is possible  only when both cobordisms are trivial cylinders. The action of a  Reeb chord $\gamma$ of $\Lambda^\pm_1 \cup \Lambda^\pm_0$ is defined by
\begin{align*}
\mathfrak{a}(\gamma) := e^T\ell(\alpha) + (\mathfrak{c}_i- \mathfrak{c}_j) & \quad  \text{if  $\gamma$ is a chord of $\Lambda_0^+ \cup \Lambda_1^+$, and} \\
\mathfrak{a}(\gamma)  := e^{-T}\ell(\alpha) \phantom{+ (\mathfrak{c}_1- \mathfrak{c}_0)}  & \quad  \text{if $\gamma$ is a chord of $\Lambda_0^- \cup \Lambda_1^-$.}
 \end{align*}
In particular,  the action of a pure Reeb chord $\gamma$ of $\Lambda_i^\pm$ is $\mathfrak{a}(\gamma)=e^{\pm T}\ell(\gamma)$. Given a word $\boldsymbol{\gamma}= \gamma_1 \ldots \gamma_d$, we denote
  $\mathfrak{a}(\boldsymbol{\gamma})= \sum \limits_{i=1}^d \mathfrak{a}(\gamma_i)$.

The action of an intersection point $p \in \Sigma_0 \cap \Sigma_1$ is defined by
\[ \mathfrak{a}(p):=f_1(p)-f_0(p).\]

Stoke's theorem gives the following proposition (see \cite{LegendrianAmbient} for details), whose proof heavily relies on the fact that each cobordism $\Sigma_i$, $i=0,1$, is exact.
\begin{Prop}
\label{prop:energy}
Let $\gamma^\pm \in {\mathcal R}(\Lambda_1^\pm, \Lambda_0^\pm)$ be mixed Reeb chords, $\boldsymbol{\delta}^-= \delta_1^- \ldots \delta_{i-1}^-$ and $\boldsymbol{\zeta}^- = \zeta_{i+1} \ldots \zeta_d$ words of pure Reeb chords on $\Lambda_1^-$ and $\Lambda_0^-$, respectively, and $p, q \in \Sigma_0 \cap \Sigma_1$ intersection points. 
\begin{itemize}
\item If $u\in\mathcal{M}(\gamma^+;\boldsymbol{\delta}^-,\gamma^-, \boldsymbol{\zeta}^-)$, then
  \begin{align}
    E_{d(\phi\alpha)}(u)&=\mathfrak{a}(\gamma^+) - \mathfrak{a}(\gamma^-)-\big(\mathfrak{a}(\boldsymbol{\delta}^-) +\mathfrak{a}(\boldsymbol{\zeta}^-) \big).\\\label{eq:7}
E_\alpha(u)&\leq 2\mathfrak{a}(\gamma^+)\nonumber
    \end{align}
\item If $u\in\mathcal{M}(\gamma^+;\boldsymbol{\delta}^-,p,\boldsymbol{\zeta}^-)$, then
  \begin{align}
 E_{d(\phi\alpha)}(u)&=\mathfrak{a}(\gamma^+) - \mathfrak{a}(p) - \big(\mathfrak{a}(\boldsymbol{\delta}^-) +\mathfrak{a}(\boldsymbol{\zeta}^-)\big).\\\label{eq:8}
E_\alpha(u)&\leq 2\mathfrak{a}(\gamma^+)\nonumber
      \end{align}
\item If $u\in\mathcal{M}(p;\boldsymbol{\delta}^-,\gamma^-,\boldsymbol{\zeta}^-)$, then
  \begin{align}
    E_{d(\phi\alpha)}(u)&= \mathfrak{a}(p) -\mathfrak{a}(\gamma^-)-\big(\mathfrak{a}(\boldsymbol{\delta}^-) + \mathfrak{a}(\boldsymbol{\zeta}^-)\big).\\\label{eq:9}
E_\alpha(u)&\leq \mathfrak{a}(p)\nonumber
\end{align}
\item If $u\in\mathcal{M}(p;\boldsymbol{\delta}^-,q,\boldsymbol{\zeta}^-)$, then
  \begin{align} E_{d(\phi\alpha)}(u)&= \mathfrak{a}(p) - \mathfrak{a}(q) -\big(\mathfrak{a}(\boldsymbol{\delta}^-)+\mathfrak{a}(\boldsymbol{\zeta}^-)\big).\\\label{eq:10}
E_\alpha(u)&\leq \mathfrak{a}(p)\nonumber
    \end{align}
 \item If $u\in\mathcal{M}(\gamma_{1,0};\boldsymbol{\delta}^-,\gamma_{0,1},\boldsymbol{\zeta}^-)$, then
  \begin{align}
E_{d(\phi\alpha)}(u)&=\big(\mathfrak{a}(\gamma_{1,0})+\mathfrak{a}(\gamma_{0,1})\big)-\big(\mathfrak{a}(\boldsymbol{\delta}^-)+\mathfrak{a}(\boldsymbol{\zeta}^-)\big).\\\label{eq:12}
E_\alpha(u)&\leq 2\mathfrak{a}(\gamma_{1,0})+2\mathfrak{a}(\gamma_{0,1}) \nonumber
  \end{align}
\end{itemize}
\end{Prop}

\section{The Cthulhu complex}
\label{sec:Cthulhu-complex}
Let $\Sigma_0$ and $\Sigma_1$ be two exact Lagrangian cobordisms inside the symplectisation $(\R \times P \times \R,d(e^t\alpha))$ of a contactisation. We assume that:

\begin{itemize}
\item $\Sigma_0$ and $\Sigma_1$ intersect transversely (in particular
  $\Lambda_0^\pm\cap \Lambda_1^\pm=\emptyset$), and
\item The links $\Lambda_0^\pm \sqcup \Lambda_1^\pm$ are chord-generic.
\end{itemize}

The \emph{Cthulhu complex} of the pair $(\Sigma_0,\Sigma_1)$ is the complex whose underlying graded $R$-module is
\begin{equation}\label{eqn: cthulhu complex}
\Cth_\bullet (\Sigma_0, \Sigma_1) := C^{\bullet}(\Lambda_0^+, \Lambda_1^+)[2] \oplus C^{\bullet}(\Sigma_0, \Sigma_1) \oplus C^{\bullet}(\Lambda^-_0, \Lambda^-_1)[1]
\end{equation}
for a unital ring $R$. Here $C^{\bullet}(\Lambda_0^+,\Lambda_1^+)$ is the free graded module generated by the Reeb chords
from $\Lambda_1^\pm$ to $\Lambda_0^\pm$ and $C^{\bullet}(\Sigma_0,\Sigma_1)$ is the free graded module generated by the
intersection points $\Sigma_0 \cap \Sigma_1$. 
\subsection{The Cthulhu differential}
In this section we fix generic almost complex structures $J^\pm \in {\mathcal J}^{cyl}_\pi(Y)$
and a generic path $J_\bullet$ of admissible almost complex structures in ${\mathcal J}^{adm}_{J^-, J^+}(X)$. We fix also two augmentations $\varepsilon^-_0$ and $\varepsilon^-_1$ of the Chekanov-Eliashberg algebras of $\Lambda^-_0$ and $\Lambda^-_1$ respectively. %, both defined using $J^-$.

We define the \emph{Cthulhu differential} $\mathfrak{d}_{\varepsilon^-_0,\varepsilon^-_1}$, which is a differential of degree $1$ on $Cth_\bullet (\Sigma_0, \Sigma_1)$. With respect to the direct sum decomposition in Equation~\eqref{eqn: cthulhu complex}, it has the form
\begin{equation}\label{eq:20}
\mathfrak{d}_{\varepsilon^-_0,\varepsilon^-_1}=\begin{pmatrix}
 d_{++} & d_{+0} & d_{+-} \\
0 & d_{00} &  d_{0-}  \\
0 & d_{-0} & d_{--}
\end{pmatrix}.
\end{equation}
Loosely speaking, every non-zero entry in this matrix is defined by counting rigid punctured pseudoholomorphic strips in the moduli spaces  described in Section \ref{subsec: mixed moduli spaces}, where the counts are ``weighted'' by the augmentations. 

Below we give a careful description of each entry of the matrix~\eqref{eq:20}, where the  mentioned degrees are the degrees as maps between the summands in~\eqref{eqn: cthulhu complex} \emph{without} the shifts in grading appearing in the definition of $\Cth_{\bullet}(\Sigma_0,\Sigma_1)$.
When we want to emphasise to which pair of Lagrangian cobordisms the maps belong, we put them as superscript, e.g.~$\mathfrak{d}_{\varepsilon^+_0,\varepsilon^-_1}^{\Sigma_0,\Sigma_1}$, $d_{+,-}^{\Sigma_0,\Sigma_1}$, etc. We will denote by $\#$ the count (either with sign or modulo two, as appropriate) of the {\em zero dimensional part} of the corresponding moduli space.

\subsubsection{The bilinearised LCH differential}
\label{sec:bilin-diff}
We define $\varepsilon^+_i := \varepsilon^-_i \circ \Phi_{\Sigma_i,J}$, $i=0,1$, which are augmentations of the Chekanov-Eliashberg algebras of $\Lambda^+_0$ and $\Lambda^+_1$, respectively.

The map $d_{\pm\pm}$ is the bilinearised Legendrian cohomology differential for $(\Lambda_0^\pm,\Lambda_1^\pm)$ induced by the pair $(\varepsilon_0^\pm,\varepsilon_1^\pm)$ of augmentations as defined in \cite{augcat}.
 In other words, it is defined as
\begin{eqnarray}
\label{eq:2}
\lefteqn{d_{\pm\pm}(\gamma_2^\pm):=d_{\varepsilon_0^\pm,\varepsilon_1^\pm}(\gamma_2^\pm)=}\\
\nonumber &=&\sum_{\gamma_1^\pm}\sum_{\boldsymbol{\delta}^\pm,\boldsymbol{\zeta}^\pm}\#\widetilde{\mathcal{M}}_{\R \times \Lambda^\pm_0,\R \times \Lambda^\pm_1}(\gamma_1^\pm;\boldsymbol{\delta}^\pm,\gamma_2^\pm,\boldsymbol{\zeta}^\pm;J^\pm)\cdot\varepsilon^\pm_0(\boldsymbol{\delta}^\pm)\varepsilon_1^\pm(\boldsymbol{\zeta}^\pm)\cdot
  \gamma_1^\pm.
\end{eqnarray}

It follows from Equation \eqref{eq:22} that $d_{\pm \pm}$ has degree $1$.
\subsubsection{The Floer ``differential''}
\label{sec:floer-differential}

The map $d_{00}$ is a modification of the differential in Lagrangian Floer homology. For an intersection point $q$, it is defined as
\begin{align}
  \label{eq:1}
d_{00}(q):=\sum_{p}\sum_{\boldsymbol{\delta}^-,\boldsymbol{\zeta}^-}\#\mathcal{M}_{\Sigma_0,\Sigma_1}(p;\boldsymbol{\delta}^-,q,\boldsymbol{\zeta}^-;J_\bullet)\cdot\varepsilon^-_0(\boldsymbol{\delta}^-)\varepsilon_1^-(\boldsymbol{\zeta}^-)\cdot p.
  \end{align}

From Equation \eqref{eq:24} we deduce that this map is of degree $1$.

\subsubsection{The Cultist maps}
\label{sec:cultists-maps}

The maps $d_{+0}$ and $d_{0-}$ are defined as
\begin{align}
  \label{eq:3}
 & d_{0-}(\gamma^-):=\sum_{p}\sum_{\boldsymbol{\delta}^-,\boldsymbol{\zeta}^-}\#\mathcal{M}_{\Sigma_0,\Sigma_1}(p;\boldsymbol{\delta}^-,\gamma^-,\boldsymbol{\zeta}^-; J_\bullet) \cdot\varepsilon^-_0(\boldsymbol{\delta}^-)\varepsilon_1^-(\boldsymbol{\zeta}^-)\cdot p,
\\&d_{+0}(q):=\sum_{\gamma^+}\sum_{\boldsymbol{\delta}^-,\boldsymbol{\zeta}^-}\#\mathcal{M}_{\Sigma_0,\Sigma_1}(\gamma^+;\boldsymbol{\delta}^-,q,\boldsymbol{\zeta}^-; J_\bullet)\cdot\varepsilon^-_0(\boldsymbol{\delta}^-)\varepsilon_1^-(\boldsymbol{\zeta}^-)\cdot \gamma^+.
\end{align}

Equations \eqref{eq:23} implies
that $d_{0-}$ has degree $2$, while \eqref{eq:25} implies that  $d_{+0}$ has degree $-1$.
A version of the map $d_{+0}$ appears already in \cite{Ekholm_FloerlagCOnt} in the case when the negative ends of the cobordisms are empty.
\subsubsection{The LCH map}
\label{sec:lch-map}
The map $d_{+-}$ is analogous to the map in bilinearised Legendrian
contact homology induced by an exact Lagrangian cobordism. It is defined
 as:
\begin{align}
  \label{eq:4}
  & d_{+-}(\gamma^-):=\sum_{\gamma^+}\sum_{\boldsymbol{\delta}^-,\boldsymbol{\zeta}^-}\#\mathcal{M}_{\Sigma_0,\Sigma_1}(\gamma^+;\boldsymbol{\delta}^-,\gamma^-,\boldsymbol{\zeta}^-; J_\bullet)\cdot\varepsilon^-_0(\boldsymbol{\delta}^-)\varepsilon_1^-(\boldsymbol{\zeta}^-)\cdot \gamma^+.
 \end{align}

 It follows from Equation \eqref{eq:22} that this map
has degree $0$.

\subsubsection{The Nessie map}
\label{sec:banana-map}
Let $C_\bullet(\Lambda_1^\pm,\Lambda_0^\pm)$ be the dual of $C^\bullet(\Lambda_1^\pm,\Lambda_0^\pm)$ and $\delta_{\pm\pm}$ the adjoints of the differentials $d_{\pm \pm}^{\Sigma_1,\Sigma_0}$.

The count of ``banana'' pseudoholomorphic strips gives rise to a map

\[ b \colon C_{n-1-\bullet}(\Lambda_1^\pm,\Lambda_0^\pm)\rightarrow C^\bullet(\Lambda_0^\pm,\Lambda_1^\pm) \]
which is defined as
\begin{equation} \label{eq: banana map}
b(\gamma_{01}) = \sum_{\gamma_{10}} \sum_{\boldsymbol{\delta},\boldsymbol{\zeta}} \# \widetilde{\mathcal{M}}_{\R \times \Lambda_0^-, \R \times \Lambda_1^-}(\gamma_{10};\boldsymbol{\delta},\gamma_{01},\boldsymbol{\zeta}; J^-)\cdot \varepsilon_0(\boldsymbol{\delta})\varepsilon_1(\boldsymbol{\zeta})\cdot\gamma_{10}.
\end{equation}
The degree of the map $b$ follows from Equation \eqref{eq:26}.

Further, let $CF_\bullet(\Sigma_0,\Sigma_1)$ be the dual of $CF^\bullet(\Sigma_0,\Sigma_1)$.
For a fixed choice of Maslov potentials for the two cobordisms, there is a canonical isomorphism
\begin{gather*}
CF^\bullet(\Sigma_0,\Sigma_1) \cong CF_{n+1-\bullet}(\Sigma_1,\Sigma_0).
\end{gather*}
Using $\delta_{-0}:CF_\bullet(\Sigma_1,\Sigma_0)\rightarrow C_{\bullet-2}(\Lambda_1^-,\Lambda_0^-)$ to denote the adjoint of the map $d^{\Sigma_1,\Sigma_0}_{0-}$ in the entry of $\mathfrak{d}^{\Sigma_1,\Sigma_0}_{\varepsilon_1,\varepsilon_0}$ (defined using the opposite path $\check{J}_\bullet$ such that $\check{J}_t=J_{1-t}$), we define
\[ d_{-0}:=b\circ \delta_{-0} \colon CF^\bullet(\Sigma_0,\Sigma_1) = CF_{n+1-\bullet}(\Sigma_1,\Sigma_0)\rightarrow C^\bullet(\Lambda_0^-,\Lambda_1^-),\]
which has degree 0.

\begin{figure}[ht!]
\vspace{3mm}
\labellist
\pinlabel $q$ at 80 108
\pinlabel $\gamma$ at 31 68
\endlabellist
  \centering
  \includegraphics[height=3cm]{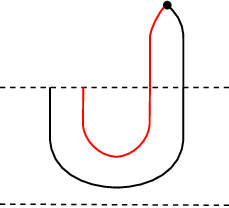}
  \caption{A building contributing to $\langle d_{-0}q,\gamma\rangle$.}
  \label{fig:nessie}
\end{figure}

More explicitly, the map $d_{-0}$ counts pseudoholomorphic buildings with two levels, one of which is a punctured strip with one end at an intersection point $p \in \Sigma_0 \cap \Sigma_1$ and  another end at a chord $\gamma_{01} \in {\mathcal R}(\Lambda^-_0, \Lambda^-_1)$ (the neck of the nessie), and the other one is a banana with one end end at $\gamma_{01} \in {\mathcal R}(\Lambda^-_0, \Lambda^-_1)$ and another end at $\gamma_{10} \in {\mathcal R}(\Lambda^-_1, \Lambda^-_0)$ (the body of the nessie). See Figure \ref{fig:nessie}. Observe that, compared to the pseudoholomorphic cthulhus used in the definition of the map $d_{0-}$, the neck of a nessie has reversed boundary conditions.
 The configurations above have previously been considered by Akaho in \cite{Akaho_conc}. (See also  \cite{Akaho} from the same author.)

\subsection{The proof of $\mathfrak{d}_{\varepsilon^-_0,\varepsilon^-_1}\,^2=0$}
\label{sec:Cthulu-differential}
We are now ready to present and prove the following central result.
\begin{Thm}\label{thm:dsquare}

Let $\Sigma_i \subset \R \times Y$, $i=0,1$, be a pair of exact Lagrangian cobordisms from $\Lambda^-_i$ to $\Lambda^+_i$ as above. 
If $\varepsilon_i^\pm$, $i=0,1$ is an augmentation of the Chekanov-Eliashberg algebra of $\Lambda_i^-$, then
\begin{itemize}
\item $\mathfrak{d}_{\varepsilon^-_0,\varepsilon^-_1}$ is well-defined over
a ring of characteristic two, and
\item $\mathfrak{d}_{\varepsilon^-_0,\varepsilon^-_1}\,^2=0$.
\end{itemize}
If $\Sigma_i$, $i=0,1$, are both endowed with relative Pin structures,  the differential $\mathfrak{d}_{\varepsilon_0^-, \varepsilon_1^-}$ is defined for any ring.
\end{Thm}

In our setting pseudoholomorphic strips are punctured, and their count is weighted by the augmentations. For this reason, we insist on two important points that need to be taken into account when performing these counts:
\begin{enumerate}[label={(\Roman*)}]
\item The punctured pseudoholomorphic strips used in the definition of $d_{++}$ have boundary on $\R \times \Lambda_i^+$ and are allowed to have punctures which are negatively asymptotic to pure chords of $\Lambda^+_i$. When adjoining the rigid punctured strips in the definition of $d_{++}$ and $d_{+*}$, where $*=0,-$, i.e.~the glued configurations corresponding to the compositions $d_{++} d_{+*}$, we do not necessarily obtain a broken pseudoholomorphic strip. In order to obtain a broken strip, we will adjoin the pseudoholomorphic punctured discs that appear in the count defining the right-hand side of
\[ \varepsilon^+_i = \varepsilon^-_i \circ \Phi_{\Sigma_i}, \:\: i=0,1.\]
From the latter equality it follows that the composition $d_{++} d_{+*}$ indeed is obtained by counting buildings of this form. \label{I}

\item Not all broken punctured pseudoholomorphic strips correspond to two glued pseudoholomorphic \emph{strips}.
Namely
there are so-called $\partial$-breakings that consist of a punctured strip together with a punctured disc with boundary on $\R \times \Lambda^\pm_i$, $i=0,1$. The count of these discs defines the differentials $\partial_\pm$ of the Chekanov-Eliashberg algebras of $\Lambda^\pm_0 \cup \Lambda^\pm_1$, and the equalities
\[ \varepsilon^\pm_i \circ \partial_\pm = 0, \:\: i=0,1,\]
holds by definition. It thus follows that the total count of  broken strips of this kind vanishes if it is weighted by the augmentations. \label{II}
\end{enumerate}

\begin{proof}[Proof of Theorem~\ref{thm:dsquare}]
The theorem is a consequence of Proposition~\ref{prop: structure of moduli spaces}. In order to verify $\mathfrak{d}_{\varepsilon^-_0,\varepsilon^-_1}\,^2=0$
we make a term-by-term analysis of the matrix
\begin{gather*}
\mathfrak{d}_{\varepsilon^-_0,\varepsilon^-_1}\,^2 =  \\
\begin{pmatrix}
 d_{++}\,^2& d_{++}d_{+0}+d_{+0}d_{00}+d_{+-}d_{-0} & d_{++}d_{+-}+d_{+0}d_{0-}+d_{+-}d_{--} \\
0 & d_{00}d_{00}+d_{0-}d_{-0} &  d_{00}d_{0-}+d_{0-}d_{--}  \\
0 & d_{-0}d_{00}+d_{--}d_{-0} & d_{-0}d_{0-}+d_{--}d_{--}
\end{pmatrix}.
\end{gather*}
and show that all entries count boundary configurations of one-dimensional moduli spaces.

\begin{itemize}
\item $d_{++}^2=0$. The term $d_{++}$ is the standard bilinearised Legendrian contact cohomology differential \cite{augcat} restricted to mixed chords from $\Lambda^+_1$ to $\Lambda^+_0$.

\item $d_{++}d_{+0}+d_{+0}d_{00}+d_{+-}d_{-0}=0$. We must study the boundary of the moduli spaces $\mathcal{M}_{\Sigma_0,\Sigma_1}(\gamma;\boldsymbol{\delta},p,\boldsymbol{\zeta})$ of dimension $1$. Their possible breakings are schematically depicted in Figure \ref{fig:breaking3}.

We claim that, when counting these boundary points weighted by the augmentations $\varepsilon^-_i$, $i=0,1$, we get the contribution
\[ \langle (d_{++}d_{+0}+d_{+0}d_{00}+d_{+-}d_{-0})(p),\gamma \rangle.\]
Indeed, inspecting the boundary components of different types
we obtain the terms $\langle d_{++}d_{+0}(p),\gamma \rangle$ (here we use $\varepsilon^+_i=\varepsilon^-_i \circ \Phi_{\Sigma_i}$; see \ref{I}), $\langle d_{+0}d_{00}(p),\gamma \rangle$, and $\langle d_{+-}d_{-0})(p),\gamma \rangle$, together with the $\partial$-breakings which contribute to zero (here we use $\varepsilon^\pm_i\circ \partial^\pm=0$; see \ref{II}).

\begin{figure}[ht!]
\centering
\vspace{0.5cm}
\labellist
\pinlabel $\Lambda^+_0\cup{\color{red}\Lambda^+_1}$ at -40 145
\pinlabel $\Sigma_0\cup{\color{red}\Sigma_1}$ at -40 90
\pinlabel $\Lambda^-_0\cup{\color{red}\Lambda^-_1}$ at -40 30
\pinlabel $0$ at 16 160
\pinlabel $0$ at 17 103
\pinlabel $0$ at 44 81
\pinlabel $1$ at 97 160
\pinlabel $0$ at 98 103
\pinlabel $0$ at 168 160
\pinlabel $0$ at 188 80
\pinlabel $1$ at 192 32
\pinlabel $0$ at 257 160
\pinlabel $0$ at 257 104
\pinlabel $1$ at 282 14
\pinlabel $0$ at 306 70
\endlabellist
\includegraphics[scale=0.65]{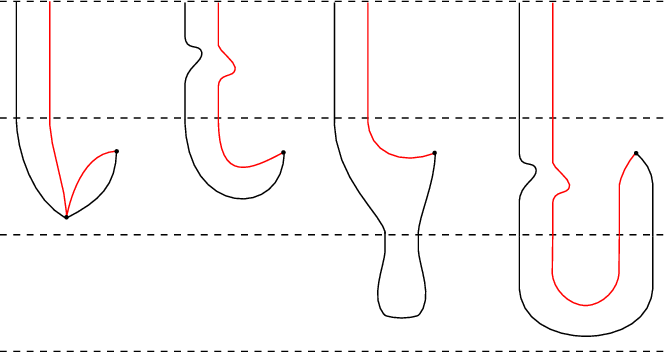}
\caption{Breakings involved in $d_{++}d_{+0}+d_{+0}d_{00}+d_{+-}d_{-0}=0$. The number on each component denotes its Fredholm index.}
\label{fig:breaking3}
\end{figure}
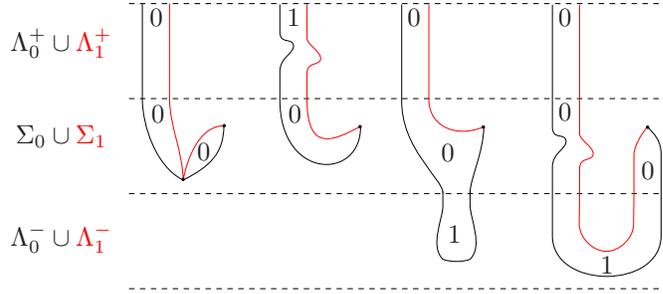

\item $d_{++}d_{+-}+d_{+0}d_{0-}+d_{+-}d_{--}=0$. This follows as above.
\item $d_{00}d_{00}+d_{0-}d_{-0}=0$. Again this follows as above.

\item $d_{00}d_{0-}+d_{0-}d_{--}=0$. This also follows  as
above.
\item $d_{-0}d_{00}+d_{--}d_{-0}=0$. Analysing the breakings of holomorphic bananas, we get that the map $b$ satisfies $b\circ \delta_{--}=d_{--}\circ b$ (see Figure \ref{fig:breakingbanana}), where $\delta_{--}$ again denotes the adjoint of $d_{--}^{\Sigma_1,\Sigma_0}$. Hence, we obtain
\[d_{-0}d_{00}+d_{--}d_{-0}=b\delta_{-0}d_{00}+d_{--}b\delta_{-0}=b\big(\delta_{-0}d_{00}+\delta_{--}\delta_{-0}\big),\]
where $\delta_{-0}$ is the adjoint of $d^{\Sigma_1,\Sigma_0}_{0-}$. Since $d_{00}$ is the adjoint of $d^{\Sigma_1,\Sigma_0}_{00}$, the factor $\delta_{-0}d_{00}+\delta_{--}\delta_{-0}$ is the adjoint of $d^{\Sigma_1,\Sigma_0}_{00}d^{\Sigma_1,\Sigma_0}_{0-}+d^{\Sigma_1,\Sigma_0}_{0-}d^{\Sigma_1,\Sigma_0}_{--}$, which vanishes by the previous case.

\begin{figure}[ht!]
\centering
\vspace{0.5cm}
\labellist
\pinlabel $\Lambda^-_0\cup{\color{red}\Lambda^-_1}$ at -40 90
\pinlabel $\Lambda^-_0\cup{\color{red}\Lambda^-_1}$ at -40 30
\pinlabel $1$ at 16 102
\pinlabel $0$ at 65 102
\pinlabel $1$ at 40 16
\pinlabel $1$ at 133 71
\pinlabel $1$ at 134 33
\pinlabel $0$ at 193 102
\pinlabel $1$ at 243 104
\pinlabel $0$ at 219 16
\endlabellist
\includegraphics[scale=0.65]{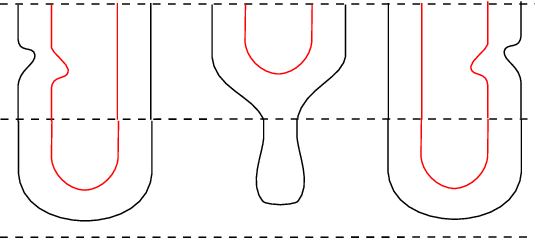}
\caption{Breakings involved in $b\circ \delta_{--}=d_{--}\circ b$.}
\label{fig:breakingbanana}
\end{figure}
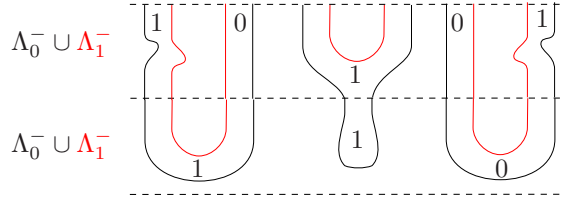
\item $d_{-0}d_{0-}+d_{--}d_{--}=0$. For action reasons we must have $d_{-0}d_{0-}=0$. Moreover $d_{--}d_{--}=0$ for the same reason why $d_{++}d_{++}=0$.
\end{itemize}
\end{proof}

\section{The transfer and co-transfer map for concatenations of cobordisms}
\label{sec:conc-cobord}

Recall that two exact Lagrangian cobordisms with matching ends can be concatenated; see Section \ref{sec:exact-lagr-cobord}. In this section we will provide formulas which relate the Floer homologies of the different pieces of such a concatenation. This will be done by introducing a relative version of Viterbo's transfer map, originally defined in \cite{Viterbo_Functors} for symplectic (co)homology. (Recall that Viterbo's transfer map concerns concatenations of \emph{symplectic cobordisms.}) For the Hamiltonian formulation of wrapped Floer homology, the transfer map was constructed and treated in \cite{WrappedFuk}.

In the following we will consider exact Lagrangian cobordisms
$$V_0,V_1,W_0,W_1 \subset \R \times P \times \R,$$
where  $V_i$ is a cobordism from $\Lambda_i^-$ to $\Lambda_i$ and $W_i$ is a cobordism from $\Lambda_i$ to $\Lambda_i^+$
for $i=0,1$. The concatenations
\[ V_i \odot W_i \subset \R \times P \times \R, \:\:i=0,1,\]
are exact Lagrangian cobordisms from $\Lambda^-_i$ to $\Lambda^+_i$.

Under the further assumption that the negative ends of $V_i$, $i=0,1$, are empty, a transfer map
\[ \Phi_{W_0,W_1} \co \Cth_\bullet(V_0,V_1) \to \Cth_\bullet(V_0 \odot W_0, V_1 \odot W_1)\]
was constructed in \cite[Section 4.2.2]{Ekholm_FloerlagCOnt}; recall that the analytic set-up of the latter article is the same as the one used here. Our construction of the transfer map will be a straight-forward generalisation of this construction to the case when the negative ends are non-empty.

We will also construct a  \emph{co-transfer map}
\[\Phi^{V_0,V_1} \co \Cth_\bullet(V_0 \odot W_0, V_1 \odot W_1) \to \Cth_\bullet(W_0,W_1)\]
which should be thought of as a quotient projection associated to a transfer map.

\subsection{Concatenations and stretching of the neck}
\label{sec:conc}

Assume that $J^a_\bullet$ and $J^b_\bullet$ are paths of admissible almost complex structures on $\R \times Y$ which are cylindrical in the subsets $\{ t \ge -1\}$ and $\{t \le 1\}$, respectively, and which moreover agree in the subset $\{ -1 \le t \le 1\}$. We also assume that $V_i$ and $W_i$ are cylindrical in the subsets $\{ t \ge -1\}$ and $\{t \le 1\}$, respectively, where they coincide. For each $N \ge 0$ let $\tau_N \colon \R \times Y \to \R \times Y$ be the translation $\tau_N(t,y)=(t+N, y)$. We define
\begin{eqnarray*}
V \odot_N W &:=& (V \cap \{ t \le 0\}) \cup (\tau_N(W) \cap \{ t \ge 0\}),\\
(J^a_\bullet \odot_N J^b_\bullet)(t,p,z)&:=& \begin{cases} J^a_\bullet(t,p,z) & t \le 0\\
J^b_\bullet(t-N,p,z) & t \ge 0,
\end{cases}
\end{eqnarray*}
where we recall that $\tau_T$ is the translation of the $t$-coordinate by $T \in \R$. We also write $J^a_\bullet \odot J^b_\bullet := J^a_\bullet \odot_0 J^b_\bullet$.

The Hamiltonian isotopy class of $V \odot_N W \subset \R \times Y$ is independent of $N \ge 0$ by Lemma~\ref{noia}. We have thus produced a family of boundary value problems for $(J^a_\bullet \odot_N J^b_\bullet)$-holomorphic curves in $\R \times Y$ having boundary on $V \odot_N W$, $N \ge 0$. This family of boundary value problems is conformally equivalent to the family which ``stretches the neck'' along the contact-type hypersurface $\{0\} \times Y \subset \R \times Y$ with boundary condition $V \odot W$; see \cite[Section 3.4]{Bourgeois_&_Compactness} as well as \cite[Section 1.3]{Eliashberg_&_SFT} for more details. This fact will be important below.

There is a compactness theorem
in the case of a neck-stretching sequence of almost complex structure; see \cite[Section 10]{Bourgeois_&_Compactness} for the precise formulation. The key fact is that a sequence of $J^a_\bullet \odot_N J^b_\bullet$-holomorphic punctured strips, with $N \to +\infty$, has a subsequence converging to a building consisting of \emph{several} levels whose components satisfy non-cylindrical boundary conditions. In the case under consideration, the limit buildings consist of:
\begin{itemize}
\item An upper level containing punctured $J^b_\bullet$-holomorphic strips (or discs) with boundary on $W_0 \cup W_1$; and
\item A lower level containing punctured $J^a_\bullet$-holomorphic strips (or discs) with boundary on $V_0 \cup V_1$.
\end{itemize}
A priori there could also be levels consisting of pseudoholomorphic discs or strips for a cylindrical almost complex structure satisfying cylindrical boundary conditions on $\R \times
\Lambda_i^-$, $\R \times \Lambda_i$ or $\R \times \Lambda_i^+$. However, since we are only interested in rigid configurations, and since the latter solutions will have positive dimension (unless they are trivial strips), they can be omitted from our breaking analysis (under the assumption that transversality is achieved for every level).

There is a gluing result in this setting (see \cite[Lemma 3.14]{RationalSFT}), giving a bijection between buildings of the above type where all components are of Fredholm index zero, and punctured $J^a_\bullet \odot_N J^b_\bullet$-holomorphic strips for each $N \gg 0$ sufficiently large. Figure \ref{fig:strech_hol_buil} schematically depicts two such buildings.

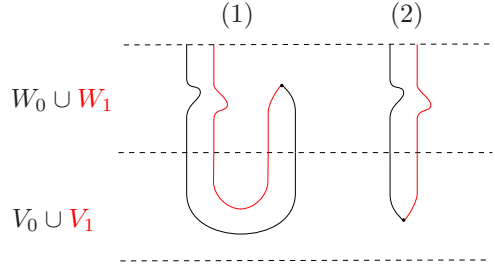
\begin{figure}[ht!]
\vspace{0.5cm}
\labellist
\pinlabel $W_0\cup {\color{red} W_1}$ at -33 95
\pinlabel $V_0\cup {\color{red} V_1}$ at -38 24
\pinlabel (1) at 70 145
\pinlabel (2) at 170 145
\endlabellist
  \centering
  \includegraphics[scale=0.65]{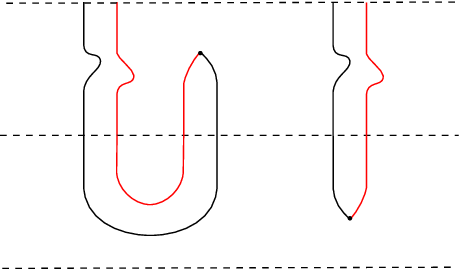}
  \caption{A holomorphic buildings appearing after stretching the neck along $\Lambda$}
  \label{fig:strech_hol_buil}
\end{figure}

\subsection{The complex after a neck stretching procedure}
\label{sec:splitting}
The goal of this section is to find a description of the complex
\[(\Cth_\bullet(V_0 \odot_N W_0, V_1 \odot_N W_1),\mathfrak{d}_{\varepsilon_0^-\varepsilon_1^-}^{V\odot W})\]
when the almost complex structure is given by $J^a_\bullet \odot_N J^b_\bullet$ for $N \gg 0$ sufficiently large.

We first consider the Cthulhu complex for $(V_0, V_1)$ defined using $J^a_\bullet$:
\begin{gather*}
(\Cth_\bullet(V_0,V_1)=C^{\bullet-{2}}(\Lambda_0,\Lambda_1) \oplus C^{\bullet}(V_0,V_1) \oplus C^{\bullet -1} (\Lambda^-_0, \Lambda^-_1),\mathfrak{d}_{\varepsilon_0^-\varepsilon_1^-}^V),\\
\mathfrak{d}_{\varepsilon_0^-\varepsilon_1^-}^V=\begin{pmatrix}
d_{++}^{V_0,V_1} & d^{V_0,V_1}_{+0} & d^{V_0,V_1}_{+-} \\
0 & d^{V_0,V_1}_{00} & d^{V_0,V_1}_{0-} \\
0 & d^{V_0,V_1}_{-0} & d_{--}^{V_0,V_1}
\end{pmatrix}.
\end{gather*}

Consider the entries $d^{V_1,V_0}_{+0}$, $d^{V_1,V_0}_{+-}$, and $d_{\pm\pm}^{V_1,V_0}$ in the differential of the complex $\Cth_\bullet(V_1,V_0)$  (i.e. with the inverse order for the cobordisms), where 
the opposite path $\check{J}^a_\bullet$ is used.
We will need their adjoints
\begin{eqnarray*}
\delta^{V_0,V_1}_{0+} &:=& (d^{V_1,V_0}_{+0})^* \colon C_\bullet(\Lambda_1,\Lambda_0) \to C_{\bullet-2}(V_1,V_0),\\
\delta^{V_0,V_1}_{-+} &:=& (d^{V_1,V_0}_{+-})^* \colon C_\bullet(\Lambda_1,\Lambda_0) \to C_{\bullet-1}(\Lambda^-_1,\Lambda^-_0),\\
\delta_{++}^{V_0,V_1} &:=& (d_{++}^{V_1,V_0})^* \colon C_\bullet (\Lambda_1,\Lambda_0) \to C_{\bullet -1}(\Lambda_1,\Lambda_0),\\
\delta_{--}^{V_0,V_1} &:=& (d_{--}^{V_1,V_0})^* \colon C_\bullet (\Lambda_1^-,\Lambda_0^-) \to C_{\bullet -1}(\Lambda_1^-,\Lambda_0^-),\\
\delta_{-+}^{V_0,V_1} &:=& (d_{+-}^{V_1,V_0})^* \colon C_\bullet (\Lambda_1,\Lambda_0) \to C_{\bullet -1}(\Lambda_1^-,\Lambda_0^-),\end{eqnarray*}
where the canonical basis of Reeb chords and double points has been used in order to identify the modules and their duals. Observe that, exploiting the same notation, we also get $\delta^{V_0,V_1}_{00}=d^{V_0,V_1}_{00}$.

We write
\[ \varepsilon_i' := \varepsilon_i \circ \Phi_{V_i,J^a_\bullet}, \:\: i=0,1, \]
for the pull-backs of the augmentations under the DGA morphisms induced by the respective cobordisms. These augmentations now give rise to a complex
\begin{gather*}
(\Cth_\bullet(W_0,W_1)=C^{\bullet-2}(\Lambda_0^+,\Lambda_1^+) \oplus C^{\bullet} (W_0,W_1) \oplus C^{\bullet-1} (\Lambda_0, \Lambda_1), \mathfrak{d}_{\varepsilon'_0\varepsilon'_1}^W),\\
\mathfrak{d}_{\varepsilon'_0\varepsilon'_1}^W=\begin{pmatrix}
d^{W_0,W_1}_{++} & d^{W_0,W_1}_{+0} & d^{W_0,W_1}_{+-} \\
0 & d^{W_0,W_1}_{00} & d^{W_0,W_1}_{0-} \\
0 & d^{W_0,W_1}_{-0} & d^{W_0,W_1}_{--}
\end{pmatrix},
\end{gather*}

where the differential is defined using the almost complex structure $J^b_\bullet$.
Note that $d^{W_0,W_1}_{--}=d^{V_0,V_1}_{++}$.

In addition we will also need the map
\[ b^{V_0,V_1} \colon C_\bullet(\Lambda_1,\Lambda_0) \to C^{n-2-\bullet}(\Lambda_0,\Lambda_1) \]
which is defined similarly to
\[b^{\Lambda_0,\Lambda_1} \colon C_\bullet(\Lambda_1,\Lambda_0) \to C^{n-1-\bullet}(\Lambda_0,\Lambda_1)\] as defined in Section \ref{sec:banana-map}, but which instead counts \emph{rigid} ``bananas'' having boundary on $V_0 \cup V_1$, and two punctures with positive asymptotics to Reeb chords.

Recall that the compactness theorem for a neck-stretching sequence together with pseudoholomorphic gluing shows the following. For $N \gg 0$ sufficiently large, the rigid $(J^a_\bullet \odot_N J^b_\bullet)$-holomorphic curves in $\R \times Y$ having boundary on $(V_0 \odot_N W_0) \cup (V_1 \odot_N W_1)$ are in bijective correspondence with pseudoholomorphic buildings of the form described in Section \ref{sec:conc}, in which every involved component is rigid.

Analysing the possible such pseudoholomorphic buildings, we obtain the following. When $N \gg 0$ is sufficiently large, the differential of the complex for the concatenated cobordisms, defined using the almost complex structure $J^a_\bullet \odot_N J^b_\bullet$ as in the previous paragraph, is given by
\begin{gather*}
(\Cth_\bullet(V_0 \odot_N W_0, V_1 \odot_N W_1)\\= C^{\bullet-2}(\Lambda_0^+,\Lambda_1^+) \oplus C^{\bullet}(W_0,W_1) \oplus C^{\bullet}(V_0,V_1) \oplus C^{\bullet-1}(\Lambda^-_0, \Lambda^-_1), \mathfrak{d}^{V\odot W}_{\varepsilon^-_0\varepsilon^-_1}),\\
\end{gather*}
with differential
\begin{gather*}
\mathfrak{d}^{V\odot W}_{\varepsilon^-_0\varepsilon^-_1} = \\
\begin{pmatrix}
d^{W_0,W_1}_{++} & d^{W_0,W_1}_{+0}+d^{W_0,W_1}_{+-} b^{V_0,V_1}\delta^{W_0,W_1}_{-0} & d^{W_0,W_1}_{+-}d^{V_0,V_1}_{+0} & d^{W_0,W_1}_{+-}d^{V_0,V_1}_{+-} \\
0 & d^{W_0,W_1}_{00}+d^{W_0,W_1}_{0-} b^{V_0,V_1}\delta^{W_0,W_1}_{-0} & d^{W_0,W_1}_{0-}d^{V_0,V_1}_{+0} & d^{W_0,W_1}_{0-}d^{V_0,V_1}_{+-}  \\
0 & d^{V_0,V_1}_{0+} \delta^{W_0,W_1}_{-0} &  d^{V_0,V_1}_{00} & d^{V_0,V_1}_{0-} \\
0 & b^{\Lambda_0^-,\Lambda_1^-}\delta^{V_0,V_1}_{-+} \delta^{W_0,W_1}_{-0} & b^{\Lambda_0^-,\Lambda_1^-}\delta^{V_0,V_1}_{-0} & d_{--}^{V_0,V_1}
\end{pmatrix},
\end{gather*}
in terms of pseudoholomorphic strips on $V_0\cup V_1$ and $W_0 \cup W_1$ for each $N \gg 0$ sufficiently large. (For instance the term $d^{W_0,W_1}_{+-} b^{V_0,V_1}\delta^{W_0,W_1}_{-0}$ corresponds to the breaking (1) in Figure \ref{fig:strech_hol_buil} and the term $d^{W_0,W_1}_{+-}d^{V_0,V_1}_{+0}$ corresponds to the breaking (2) in the same figure.)

We have here relied on the exactness assumptions in Definition \ref{defn: exact lagrangian cobordism} and action considerations from Section \ref{sec:action-energy} in order to rule out certain configurations.

\subsection{Definition of the transfer and co-transfer maps}
\label{sec:transfer}
The transfer and co-transfer maps on the chain level are defined for a very ``stretched'' almost complex structure on a concatenated cobordism (i.e.~when the parameter $N \gg 0$ in Section \ref{sec:conc} is sufficiently large), so that the complexes take the form as described in Section \ref{sec:splitting}.
\begin{defn}
The \emph{transfer map} is defined by
\begin{gather*}
\Phi_{W_0,W_1} \co \Cth_\bullet (V_0,V_1) \to \Cth_\bullet(V_0 \odot_N W_0,V_1 \odot_N W_1),\\
\Phi_{W_0,W_1} = \begin{pmatrix}  d^{W_0,W_1}_{+-} & 0 & 0\\
d^{W_0,W_1}_{0-}  & 0 & 0 \\
0 & \mathrm{id} & 0 \\
0 & 0 & \mathrm{id} \\
\end{pmatrix},
\end{gather*}
while the \emph{co-transfer map} is defined by
\begin{gather*}
\Phi^{V_0,V_1} \co \Cth_\bullet(V_0 \odot_N W_0,V_1 \odot_N W_1) \to \Cth_\bullet (W_0,W_1),\\
\Phi^{V_0,V_1} = \begin{pmatrix} \mathrm{id} & 0 & 0 & 0\\
0 & \mathrm{id} & 0 & 0 \\
0 &  b^{V_0,V_1}\delta^{W_0,W_1}_{-0} & d^{V_0,V_1}_{+0} & d^{V_0,V_1}_{+-}
\end{pmatrix}.
\end{gather*}
\end{defn}

\begin{Lem}
\label{lem:exptrans}
If $\mathfrak{d}^V_{\varepsilon_0^-\varepsilon_1^-}$, $\mathfrak{d}^W_{\varepsilon'_0\varepsilon'_1}$ and $\mathfrak{d}^{V\odot W}_{\varepsilon_0^-\varepsilon_1^-}$ are defined using $J^a_\bullet$, $J^b_\bullet$ and $J^a_\bullet\odot_N J^b_\bullet$ (for $N \gg 0$ as before) respectively, the transfer and co-transfer maps  are chain maps.
\end{Lem}
\begin{proof}
This a standard but tedious study of the degenerations of $1$-parameter families of curves through the neck stretching procedure.
\end{proof}

\begin{Rem}
\label{rem:transfer}
When there are no Reeb chords from $\Lambda_1$ to $\Lambda_0$ (recall that these are the Legendrian submanifolds along which the concatenations are performed), the transfer and co-transfer maps take a particularly simple form. Since $C(\Lambda_0,\Lambda_1)=0$, the transfer map $\Phi_{W_0,W_1}$ is simply the inclusion of a subcomplex, while the co-transfer map $\Phi^{V_0,V_1}$ becomes the corresponding quotient projection. In fact, as it will be shown below in Section \ref{sec:wrap}, this situation can always be achieved after the application of a Hamiltonian isotopy that ``wraps'' the positive and negative ends of $V_1$ and $W_1$, respectively.
\end{Rem}

The following lemma is standard. It follows from the fact that, in the cylindrical situation, regular $0$-dimensional moduli spaces are trivial, together with a stretching-the-neck argument.
\begin{Lem}\label{lem:composition} The transfer and co-transfer maps satisfy the following properties:
\begin{itemize}
\item If $W_i=\R \times \Lambda_i$, $i=0,1$, and $J^b_\bullet=J^b$ is a cylindrical almost complex structure, we have
\[ \Phi_{W_0,W_1}= \mathrm{id}. \]
\item If $V_i=\R \times \Lambda_i$, $i=0,1$, and $J^a_\bullet=J^a$ is a cylindrical almost complex structure, we have
\[  \Phi^{V_0,V_1} = \mathrm{id}. \]
\item If $W_i=U_i \odot_M U_i'$, $i=0,1$, and $J^b_\bullet=J^c_\bullet \odot_M J^d_\bullet$, we have
\[\Phi_{W_0,W_1}=\Phi_{U_0',U_1'} \circ \Phi_{U_0,U_1}\]
when $M \gg 0$ is sufficiently large.
\item  If $V_i=U_i \odot_M U_i'$, $i=0,1$, and $J^a_\bullet=J^c_\bullet \odot_M J^d_\bullet$ we have
\[\Phi^{V_0,V_1}=\Phi^{U_0',U_1'} \circ \Phi^{U_0,U_1}\]
when $M \gg 0$ is sufficiently large.
\end{itemize}
\end{Lem}

In the case when $W_1=\R \times \Lambda$, we write
$\Phi_{W_0}:=\Phi_{W_0,W_1}$
and, similarly, when $V_1=\R \times \Lambda \subset \R \times Y$, we write $\Phi^{V_0}:=\Phi^{V_0,V_1}$.

\section{Proof of acyclicity}
\label{sec:acyclicity}
In this section we establish the invariance result for our Floer theory. In fact, in our context, the invariance is simply the fact the complex $\Cth(\Sigma_0,\Sigma_1)$ is acyclic (actually null-homotopic). The naive reason for this is that one can use the Reeb flow in order to displace any exact Lagrangian cobordism from any other one.

\subsection{Wrapping the ends}
\label{sec:wrap}
Let $\Sigma_i$, $i=0,1$, be exact Lagrangian cobordisms from $\Lambda^-_i$ to $\Lambda^+_i$, $i=0,1$. We assume without loss of generality that $\Sigma_i$, $i=0,1$, are both cylindrical in the subset $\{ |t| \ge T \}$ for some $T>0$.

Fix a smooth non-decreasing cut-off function $\rho \co \R \to [0,1]$ satisfying $\rho(t)=0$ for $t \le 1$ and $\rho(t)=1$ for $t \ge 2$ and, for $N \ge 0$, consider the functions
\begin{align*}
&\rho_{N, +}(t)  = \rho(t-T-N), \\
&\rho_{N, -}(t)  = \rho(-t-T-N), \\
&\rho_N (t)  = \rho_{N, +}(t) + \rho_{N, -}(t)
\end{align*}
and the flows
\begin{align*}
&\phi_{N,+}^s(t,p,z)  = (t,p,z+s \rho_{N,+}(t)), \\
&\phi_{N,-}^s(t,p,z)  = (t,p,z+s \rho_{N,-}(t)), \\
&\phi_N^s(t,p,z)  = (t,p,z+s \rho_N(t)),
\end{align*}
which are the Hamiltonian flows of the functions $h_{N, \pm}(t,p,z)= e^t \sigma_{N, \pm}(t)$ and
$h_N(t,p,z)= e^t \sigma_N(t)$ respectively where $\sigma_{N, \pm}$ and $\sigma_N$ are such that $\sigma_{N, \pm}'+\sigma_{N, \pm}=\rho_{N, \pm}$ and $\sigma_N'+\sigma_N=\rho_N$. Note that $\phi_N^s = \phi_{N,+}^s \circ \phi_{N, -}^s = \phi_{N, -}^s \circ \phi_{N,+}^s $. We denote also $ \phi_\pm^s= \phi_{0,\pm}^s$.

After applying the isotopy $\phi^{-S}_N$ to $\Sigma_0$ for $S \gg 0$ sufficiently large, we get additional double points $\phi^{-S}_N(\Sigma_0) \cap \Sigma_1$, all which correspond to the Reeb chords starting on the ends of $\Sigma_1$ and ending on the corresponding ends on $\Sigma_0$. More precisely, the following is true (the proof is standard and left to the reader):
\begin{Lem}
\label{lem:wrap}
When
\[S \ge S_0:= 2\max_{c \in \mathcal{R}(\Lambda^-_1,\Lambda^-_0) \cup \mathcal{R}(\Lambda^+_1,\Lambda^+_0)} \ell(c)\]
there are canonical bijections
\[ w_\pm \colon \phi^{-S}_{N, \pm}(\R \times \Lambda^\pm_0) \cap (\R \times \Lambda^\pm_1)\to \mathcal{R}(\Lambda^\pm_1,\Lambda^\pm_0)\]
induced by the Lagrangian projection, i.e.~by identifying elements on both sides with a double point in $\Pi_{\OP{Lag}}(\Lambda^\pm_0 \cup \Lambda^\pm_1) \subset P$. These bijections moreover satisfy
\[ gr(w_-(p))= \gr(p) \:\: \text{and} \:\: gr(w_+(p))= gr(p)-1.\]
In particular, there is a canonical identification
\[\Cth_\bullet(\Sigma_0,\Sigma_1) \cong \Cth_\bullet(\phi_N^{-S}(\Sigma_0),\Sigma_1) \]
on the level of \emph{graded modules}. After taking $N \gg 0$ sufficiently large, we may assume that the action of a generator in
\[w_-(C(\Lambda^-_0,\Lambda^-_1)) \subset \Cth_\bullet(\phi_N^{-S}(\Sigma_0),\Sigma_1),\]
is arbitrarily small, the action of a generator in
\[ w_+(C(\Lambda^+_0,\Lambda^+_1)) \subset \Cth_\bullet(\phi_N^{-S}(\Sigma_0),\Sigma_1)\] is arbitrarily large, while the action of a generator in
\[C(\Sigma_0,\Sigma_1) \subset \Cth_\bullet(\phi_N^{-S}(\Sigma_0),\Sigma_1)\]
coincides with its original action.
\end{Lem}

Next we will show that the identification given by the above lemma, in fact can be made to hold at the level of \emph{complexes} as well.

\begin{Prop}
\label{prp:wrap}
For each $N \gg 0$ sufficiently large and $S \ge S_0$ as defined above, there is a canonical identification of complexes
\[ (\Cth_\bullet(\Sigma_0,\Sigma_1),\mathfrak{d}_{\varepsilon_0^-\varepsilon_1^-})= (\Cth_\bullet(\phi_N^{-S}(\Sigma_0),\Sigma_1),\mathfrak{d}_{\varepsilon_0^-\varepsilon_1^-}).\]
\end{Prop}
\begin{Rem}
Recall that $J_\bullet$ is a cylindrical lift outside of a compact set of $\R \times Y$ by assumption, so outside that compact set it is invariant under the Reeb flow. Since the negative ends of $\Sigma_0$ and $\phi_N^S(\Sigma_0)$ differ by the time-$S$ Reeb flow, it follows that these Legendrian submanifolds have canonically isomorphic Chekanov-Eliashberg algebras and, in particular, we can identify their augmentations.

For a general choice of almost complex structure with cylindrical ends, there should again exist an analogous isomorphism, albeit non-canonical.
\end{Rem}
\begin{proof}[Proof of Proposition~\ref{prp:wrap}]
Consider the transfer and co-transfer maps
\begin{align*}
\Phi_{\phi^{-S}_+(\R \times \Lambda^+_0)} \co & \Cth_\bullet(\Sigma_0,\Sigma_1) \to \Cth_\bullet(\phi^{-S}_{N,+}(\Sigma_0),\Sigma_1),\\
\Phi^{\phi^{-S}_-(\R \times \Lambda^-_0)} \co & \Cth_\bullet(\phi_N^{-S}(\Sigma_0),\Sigma_1) \to \Cth_\bullet(\phi^{-S}_{N,+}(\Sigma_0),\Sigma_1),
\end{align*}
defined by counting $J^\pm$-holomorphic strips having boundary on $\phi^{-S}_{N,+}(\R \times \Lambda^+_0) \cup (\R \times \Lambda^+_1)$ and $\phi^{-S}_{N,-}(\R \times \Lambda^-_0) \cup (\R \times \Lambda^-_1)$, respectively. In order to identify the domains and codomains of the above maps we have used the fact that,
\[ J_\bullet \odot_N J^+ = J^- \odot_N J_\bullet=J_\bullet, \:\: N \ge 0,\]
which holds by the assumptions made on $J_\bullet$, as well as the facts that
\begin{eqnarray*}
\phi^{-S}_{N,+}(\Sigma_0) &=& \Sigma_0 \odot_N \phi^{-S}_+(\R \times \Lambda^+_0), \\
\phi_N^{-S}(\Sigma_0) &=& \phi^{-S}_-(\R \times \Lambda^-_0) \odot_N \phi^{-S}_{N,+}(\Sigma_0),\\
\end{eqnarray*}
hold for every $N \ge 0$ by construction.

Recall that the transfer and co-transfer maps are chain maps by Lemma \ref{lem:exptrans}, assuming that $N \gg 0$ has been chosen sufficiently large. The proposition will follow from the claim that the maps $\Phi_{\phi^{-S}_+(\R \times \Lambda^+_0)}$ and $\Phi^{\phi^{-S}_-(\R \times \Lambda^-_0)}$ both are isomorphisms which, moreover, induce the respective canonical identifications of graded modules described in Lemma \ref{lem:wrap}. The latter facts follows by the explicit disc count performed in the proof of \cite[Theorem 2.15]{Floer_Conc}; also, see \cite[Proposition 5.11]{LiftingPseudoholomorphic} for a similar argument. Roughly speaking, it is shown there that every $J^\pm$-holomorphic disc of index zero in the definition of the above (co-)transfer map is a transversely cut-out strip having one positive puncture and one negative puncture, and whose image under the canonical projection to $P$ is constant. Conversely, there is an explicitly defined such strip for every double point in $P$ corresponding to a Reeb chord. In particular, under an appropriate choice of basis, the matrices of both maps $\Phi_{\phi^{-S}_+(\R \times \Lambda^+_0)}$ and $\Phi^{\phi^{-S}_-(\R \times \Lambda^-_0)}$ are equal to the identity matrices.
\end{proof}

\subsection{Invariance under compactly supported Hamiltonian isotopies}

The following proposition is the core of the invariance result that we need in order to deduce the acyclicity of the Cthulhu complex.

\begin{Prop}\label{prp:invariance-comp}
Let $(\Sigma^s_0,\Sigma_1)$, $s \in [0,1]$, be a compactly supported one-parameter family of pairs of exact Lagrangian cobordisms from $\Lambda^-_i$ to $\Lambda^+_i$, $i=0,1$.
There is an induced homotopy equivalence
\[\Psi_{\{\Sigma^s_0\}} \co (\Cth_\bullet(\Sigma^0_0,\Sigma_1),\mathfrak{d}_{\varepsilon_0^-\varepsilon_1^-}) \to (\Cth_\bullet(\Sigma^1_0,\Sigma_1),\mathfrak{d}_{\varepsilon_0^-\varepsilon_1^-}).\]
This map moreover restricts to give an isomorphism of the complex
$(C_\bullet(\Lambda^+_0,\Lambda^+_1),d_{++})$, regarded as a subcomplex of $(\Cth_\bullet(\Sigma^*_0,\Sigma_1),\mathfrak{d}_{\varepsilon_0^-\varepsilon_1^-})$, for $i=0,1$.
\end{Prop}
\begin{Rem}
Under the additional assumption that $\Lambda^-_i=\emptyset$, $i=0,1$, this result was established in \cite[Section 4.2.1]{Ekholm_FloerlagCOnt}.
\end{Rem}

\begin{proof}
The first part follows from standard bifurcation analysis.

In order to deduce the last claim of the proposition, it will be necessary to use the following additional property of the identifications of complexes described in the proof of Proposition \ref{prp:wrap}. Consider the subset $C \subset \phi_N^{-S}(\Sigma^s_0) \cap \Sigma_1$ of intersection points corresponding to the Reeb chords from $\Lambda^+_1$ to $\Lambda^+_0$; these intersection points are contained in a subset of the form $(N,+\infty) \times P \times \R$, which may be assumed to be fixed in the one-parameter family $\phi_N^{-S}(\Sigma^s_0)$ of cobordisms. Even though these intersection points are fixed, their actions will in general vary with $s \in [0,1]$. We use $M$ to denote the minimum of the action of a intersection point in $C \subset \phi_N^{-S}(\Sigma^s_0) \cap \Sigma_1$ taken over all $s\in [0,1]$ . For $N \gg 0$ sufficiently large, Lemma \ref{lem:wrap} shows that any intersection point $(\phi_N^{-S}(\Sigma^s_0) \cap \Sigma_1) \setminus C$ has action strictly less than $M$. In conclusion, we must have
$K(C) \subset C$ for the map $K$ defined above, thus implying the claim.
\end{proof}

\subsection{Displacing the cobordisms}
\label{sec:proof-acyclicity}
We are finally ready to prove the main result of this section. The idea is to displace the cobordisms so that the Cthulhu complex vanishes, and to invoke the invariance properties from Propositions~\ref{prp:wrap} and~\ref{prp:invariance-comp}.

\begin{Thm} 
  For any pair of exact Lagrangian cobordisms $\Sigma_0,\Sigma_1 \subset \R \times Y$ from $\Lambda^-_i$ to $\Lambda^+_i$ and choices $\varepsilon_i^-$ of augmentations of the Chekanov-Eliashberg algebras of $\Lambda^-_i$, $i=0,1$, the complex
\[ (\Cth_\bullet(\Sigma_0,\Sigma_1),\mathfrak{d}_{\varepsilon_0^-\varepsilon_1^-})\]
is homotopic to the trivial complex,
\end{Thm}
\begin{proof}
For $S \gg 0$ sufficiently large, there is an exact Lagrangian cobordism $\Sigma_0' \subset \R \times Y$, isotopic to $ \phi_N^{-S} (\Sigma_0)$ by a compactly supported Hamiltonian isotopy, such that $\Cth(\Sigma_0',\Sigma_1)=0$. To that end, observe that $\Cth(\phi_N^{-S} (\Sigma_0),\Sigma_1)$ has no Reeb chord generators for $S \gg 0$ sufficiently large by Lemma \ref{lem:wrap}. The Hamiltonian isotopy is constructed as follows. Let $\phi^s$ be the Reeb flow $\phi^s(t,p,z)=
(t, p, z+s)$. Then, denoting $\Sigma_0' =  \phi^{-s} \circ \phi_N^{-(S-s)}(\Sigma_0)$, we have that $\Sigma_0'$ is isotopic to $\phi_N^{-S} (\Sigma_0)$ by a compactly supported Hamiltonian isotopy and $\Sigma_0' \cap \Sigma_1 = \emptyset$. Therefore, $\Cth(\Sigma_0',\Sigma_1)=0$.

The invariance result for compactly supported Hamiltonian isotopies that was established by Proposition \ref{prp:invariance-comp} thus implies that the complex $\Cth_\bullet(\phi_N^{-S}(\Sigma_0),\Sigma_1)$ is null-homotopic. The fact that the same is true for the complex $\Cth_\bullet(\Sigma_0,\Sigma_1)$ is now an immediate consequence of Proposition \ref{prp:wrap}, for $N \gg 0$ that have been chosen sufficiently large.
\end{proof}

\section{Long exact sequences from the Cthulhu complex}
\label{sec:long-exact-sequence}
Under additional hypotheses, we deduce several exact sequences, including those stated in Section \ref{sec:les}, from the acyclicity of the Cthulhu complex. They will be induced by Lemma \ref{lem:posnegLES} applied to a pair of exact Lagrangian cobordisms $(\Sigma, \Sigma_{\epsilon h})$, where $\Sigma_{\epsilon h}$ is obtained from $\Sigma$ by a Hamiltonian perturbation for a suitable cylindrical Hamiltonian $h$.

\subsection{The general construction}
As usual, we fix augmentations $\varepsilon_i^-$, $i=0,1$, for the negative ends of the Lagrangian cobordisms $\Sigma_i$ and a generic path $J_\bullet$ of admissible almost complex structures.   We denote  by $CF_{-\infty}(\Sigma_0,\Sigma_1)$ the complex $C(\Sigma_0,\Sigma_1)\oplus C(\Lambda^-_0,\Lambda^-_1)$ with differential
\[ d_{- \infty}^{\varepsilon_0^-, \varepsilon_1^-}=\begin{pmatrix}
  d_{00} & d_{0-}\\
d_{-0} & d_{--}
\end{pmatrix}. \]
 Its homology is denoted by $HF_{-\infty}(\Sigma_0,\Sigma_1)$. Note that, in general, $HF_{-\infty}(\Sigma_0,\Sigma_1)$ depends on the choice of augmentations.

The complex $\Cth(\Sigma_0,\Sigma_1)$ is the cone of the chain map
$$d_{+0}+d_{+-}: CF_{-\infty}(\Sigma_0,\Sigma_1)\rightarrow C(\Lambda_0^+,\Lambda_1^+).$$
Thus, the  acyclicity of  $(\Cth(\Sigma_0,\Sigma_1),\mathfrak{d}_{\varepsilon_0^-\varepsilon_1^-})$ implies that this map is a quasi-isomorphism, and hence that
\begin{equation}
  \label{eq:19}
  LCH^k_{\varepsilon_0^+,\varepsilon_1^+}(\Lambda_0^+,\Lambda^+_1)\cong HF^{k+1}_{-\infty}(\Sigma_0,\Sigma_1).
\end{equation}

\begin{Lem} \label{lem:posnegLES}
Let $\Sigma_i$, $i=0,1$, be a graded exact Lagrangian cobordisms from the Legendrian submanifold $\Lambda^-_i$ to $\Lambda^+_i$ in $\R \times Y$ and assume that there are augmentations $\varepsilon^-_i$ of $\mathcal{A}(\Lambda^-_i)$ for $i=0,1$. If either $d_{0-}=0$ or $d_{-0}=0$ then $d_{00}^2=0$ and, denoting by $HF(\Sigma_0, \Sigma_1)$ its homology, we have the following exact sequences.
  \begin{itemize}
  \item 
If $d_{0-}=0$, then there exists a long exact sequence
    \begin{equation}
      \label{eq:17}
\xymatrixrowsep{0.15in}
\xymatrixcolsep{0.15in}
\xymatrix{
       & \cdots\ar[r]& LCH^{k-1}_{\varepsilon^+_0,\varepsilon_1^+}(\Lambda^+_0,\Lambda^+_1) \ar[dl] & \\
& HF^{k}(\Sigma_0,\Sigma_1) \ar[r]_-{d_{-0}} & LCH^{k}_{\varepsilon^-_0,\varepsilon^-_1}(\Lambda^-_0,\Lambda_1^-) \ar[r] & LCH^{k}_{\varepsilon^+_0,\varepsilon_1^+}(\Lambda^+_0,\Lambda^+_1)\ar[dl]\\
&  & \cdots &}
    \end{equation}
\item 
If $d_{-0}=0$, then there exists a long exact sequence
    \begin{equation}
      \label{eq:18}
\xymatrixrowsep{0.15in}
\xymatrixcolsep{0.15in}
       \xymatrix{
 & \cdots\ar[r] & HF^{k+1}(\Sigma_0,\Sigma_1) \ar[dl] & \\
& LCH^{k}_{\varepsilon^+_0,\varepsilon_1^+}(\Lambda^+_0,\Lambda^+_1) \ar[r] &  LCH^{k}_{\varepsilon^-_0,\varepsilon^-_1}(\Lambda^-_0,\Lambda_1^-)\ar[r]^-{d_{0-}} & HF^{k+2}(\Sigma_0,\Sigma_1) \ar[dl] & \\
& & \cdots &
}
\end{equation}
\end{itemize}
\end{Lem}
Note that $HF(\Sigma_0, \Sigma_1)$ depends on the augmentations $\varepsilon_0^-$ and
$\varepsilon_1^-$.
\begin{proof}
Since $d_{00}^2+d_{0-}d_{-0}=0$, the vanishing of either $d_{0-}$ or $d_{-0}$ implies that
$d_{00}$ is a differential. Moreover $d_{- \infty}$ becomes a triangular matrix, and therefore
$CF_{- \infty}(\Sigma_0, \Sigma_1)$ is a cone. Thus the exact sequences of the lemma follow from the exact sequence of the cone and the isomorphism \eqref{eq:19}.
\end{proof}

There are action conditions that imply $d_{0-}=0$ or $d_{-0}=0$ without the need of counting pseudoholomorphic maps.
\begin{Lem} \label{lem: when things work}
If $\mathfrak{a}(p) \le \mathfrak{\gamma^-}$ for every $p \in \Sigma_0 \cap \Sigma_1$ and every Reeb chord $\gamma^- \in {\mathcal R}(\Lambda_1^-, \Lambda_0^-)$, then $d_{0-}=0$. If
$\mathfrak{a}(p) \ge - \mathfrak{a}(\gamma^-)$ for every $p \in \Sigma_0 \cap \Sigma_1$ and every Reeb chord $\gamma^- \in {\mathcal R}(\Lambda_0^-, \Lambda_1^-)$, then $d_{-0}=0$.
\end{Lem}
\begin{proof}
If $d_{0-} \ne 0$, then there is a non-empty moduli space of the form ${\mathcal M}_{\Sigma_0, \Sigma_1}(p; \boldsymbol{\delta^-}, \gamma^-, \boldsymbol{\zeta}^-) \ne \emptyset$. Then  Equation~\eqref{eq:9} implies $\mathfrak{a}(p)- \mathfrak{a}(\gamma^-) >0$.
If $d_{-0} \ne 0$, there is a moduli space ${\mathcal M}_{\Sigma_1, \Sigma_0}(p; \boldsymbol{\delta^-}, \gamma^-, \boldsymbol{\zeta}^-) \ne \emptyset$, with  $\gamma^- \in {\mathcal R}(\Lambda_0^-, \Lambda_1^-)$, forming the neck of a nessie. Then  Equation~\eqref{eq:9} implies $- \mathfrak{a}(p)- \mathfrak{a}(\gamma^-) >0$.
\end{proof}

We introduce the following terminology.
A point $p\in \Sigma_0\cap \Sigma_1$ is \textit{positive} (resp. \textit{negative}) if $\mathfrak{a}(p)=f_1(p)-f_0(p)$ is positive (resp. negative). Assuming the (generic) condition that $\mathfrak{a}(p) \ne 0$ for every intersection point, this leads to a decomposition $CF(\Sigma_0,\Sigma_1)=CF_+(\Sigma_0,\Sigma_1)\oplus CF_-(\Sigma_0,\Sigma_1)$ where $CF_+(\Sigma_0,\Sigma_1)$ and $CF_-(\Sigma_0,\Sigma_1)$
are generated by the positive and negative intersection points respectively. This motivates the following definition for a pair $(\Sigma_0,\Sigma_1)$ of cobordisms.
\begin{defn}
\label{def:directed}
We say that $(\Sigma_0,\Sigma_1)$ is \textit{directed} if $CF_+(\Sigma_0,\Sigma_1)=0$, and \textit{$V$-shaped} if $CF_-(\Sigma_0,\Sigma_1)=0$.
\end{defn}
When $(\Sigma_0, \Sigma_1)$ is directed $d_{0-}=0$ and when $(\Sigma_0, \Sigma_1)$ is
$V$-shaped $d_{-0}=0$. This is a particular case of Lemma \ref{lem: when things work}.
\begin{Rem}
Both $CF_-(\Sigma_0,\Sigma_1)$ and $CF_+(\Sigma_0,\Sigma_1)$ can endowed with a differential induced by $d_{00}$: the first as a subspace, and the second as a quotient. Their homologies are denoted by $HF_\pm (\Sigma_0, \Sigma_1)$. Note that they depend on choices of augmentations and have extremely weak invariance properties under Hamiltonian isotopies.
\end{Rem}

\subsection{Cylindrical Hamiltonians}
\label{sec:small-pert-lagr}
In the following, we assume that $\Sigma \subset \R \times Y$  is an exact Lagrangian cobordism  from $\Lambda_-$ to $\Lambda_+$ inside the symplectisation of a contactisation. We furthermore assume that $\Sigma$ is cylindrical outside of the set $[A,B] \times P \times \R$ for some $A<B$. We shall write
\begin{gather*}
\overline{\Sigma}:= \Sigma \cap \{ t \in [A,B] \}, \\
\partial_- \overline{\Sigma} := \Sigma \cap \{ t = A\},\\
\partial_+ \overline{\Sigma} := \Sigma \cap \{ t = B \},
\end{gather*}
so that clearly $\partial \overline{\Sigma}=\partial_- \overline{\Sigma} \cup \partial_+\overline{\Sigma}$.

We will consider different push-offs constructed via suitable perturbations of an autonomous Hamiltonians $\tilde{h} \colon \R \times P \times \R \to \R$ induced by a function $\tilde{h}(t)$ only depending on the symplectisation coordinate. We write $\phi_{\tilde{h}}^s$ for the Hamiltonian flow of $\tilde{h}$
and observe that it takes the particularly simple form
\[ \phi^s_{\tilde{h}}(t,p,z)=(t,p,s e^{-t}\tilde{h}'(t)+z).\]
In particular, $\phi^s_{\tilde{h}}(\Sigma)$ is an exact Lagrangian cobordism with cylindrical ends if and only if $e^{-t}\tilde{h}'(t)$ is constant. We will assume that $e^{-t}\tilde{h}'(t) = \pm 1$ for $t \not \in [A, B]$, but not necessarily with the same sign in the two components.
The ends of $\phi^s_{\tilde{h}}(\Sigma)$ are modelled on  positive or negative Reeb posh-off of $\Lambda^\pm$, depending on the sign of $\tilde{h}$.

In general $\Sigma$ and $\phi^s_{\tilde{h}}(\Sigma)$ will not intersect transversely and their Legendrian links at infinity will not be chord generic. For this reason we will
replace $\tilde{h}(t)$ with $h(t,p,z)= \tilde{h}(t)+ e^tg(t, p)$, where $g \colon \R \times P \to \R$ is a $C^2$-small function and $\partial_t g=0$ for $t \not \in [A, B]$.
For $\epsilon >0$ we will denote
\[ \Sigma_{\epsilon h}:=\phi^\epsilon_h(\Sigma); \]
if the ends of $\Sigma$ are modelled on the Legendrian submanifolds $\Lambda^{\pm}$, then  the ends of
$\Sigma_{\epsilon h}$ are modelled on Legendrian submanifolds $\Lambda_{\epsilon h}^\pm$.
We will assume, from now on, that $\Sigma$ and $\Sigma_{\epsilon h}$ intersect transversely and $(\Lambda^\pm, \Lambda_{\epsilon h}^\pm)$ are chord generic for every $\epsilon >0$ sufficiently small. In the rest of this section, we will call such a Hamiltonian a {\em generic cylindrical Hamiltonian}.

Let us denote $h'= \partial_th$. A straightforward computation yields
$$L_{X_h} (e^s \alpha)= dh' - dh.$$
So, if $f \colon \Sigma \to \R$ is a primitive of $e^t\alpha|_{\Sigma}$, then the primitive $f_{\epsilon h} \colon \Sigma_{\epsilon h} \to \R$  of $e^t\alpha|_{\Sigma_{\epsilon h}}$ is given by
\begin{equation}
f_{\epsilon h} \circ \phi^\epsilon_h =f  + \int_0^\epsilon ((h' -h)|_{\Sigma} \circ \phi_h^s )ds. \label{eq:14}
\end{equation}
We assume, without loss of generality, that $f$ vanishes on the negative end of $\Sigma$. Since $h'=h$  at the negative end, $f_{\epsilon h}$ also vanishes there.

\subsection{The Cthulhu complex of a Hamiltonian push-off}
\label{sec:computingpushoff}

Let
$$\phi_{\pm}^s \colon P \times \R \to P \times \R$$
 be the contact flow generated by the contact Hamiltonians $g \pm 1$, where $g \colon P \to \R$ is a generic $C^2$-small function.
Given a Legendrian submanifold $\Lambda \subset P \times \R$, we denote $\Lambda_{\pm \epsilon}:= \phi^\epsilon_{\pm}(\Lambda)$. Note that $\Lambda^\pm_{\epsilon h} = \Lambda^\pm_{\pm \epsilon}$, where the sign before $\epsilon$ is the sign of $\partial_th$ in the corresponding end.
\begin{Prop}
\label{prop:twocopy}
Let $\Lambda \subset P \times \R$ be a close Legendrian submanifold.
For the cylindrical lift $\widetilde{J}_P$ of a regular almost complex structure $J_P$ on $P$, we have a canonical isomorphism of complexes
\[LCC^\bullet_{\varepsilon_0,\varepsilon_1}(\Lambda, \Lambda_{+ \epsilon})  \cong LCC^\bullet_{\varepsilon_0,\varepsilon_1}(\Lambda)\]
for any sufficiently small $\epsilon>0$.

Under the additional assumption that $\Lambda$ is horizontally displaceable, we moreover have a quasi-isomorphism
\[LCC^\bullet_{\varepsilon_0,\varepsilon_1}(\Lambda, \Lambda_{-\varepsilon}) \simeq LCC_{n-1-\bullet}^{\varepsilon_0,\varepsilon_1}(\Lambda),\]
where $n$ is the dimension of $\Lambda$.
\end{Prop}
\begin{proof}
The statement follows from \cite[Proposition 2.7]{Floer_Conc} and \cite[Proposition 4.1]{Duality_EkholmetAl}.
\end{proof}

\begin{Rem} \label{aspettando il caldaista} The identifications of augmentations in Proposition~\ref{prop:twocopy}, despite the fact that the DGAs are associated to geometrically \emph{different} Legendrians, can be justified as follows. For $\Lambda,\Lambda'' \subset P \times \R$ being sufficiently $C^1$-close together with a fixed choice of compatible almost complex structure $J_P$ on $P$, the invariance theorem in \cite{LCHgeneral} gives a canonical isomorphism between the Chekanov-Eliashberg algebras $({\mathcal A}(\Lambda), \partial_{\Lambda})$ and $({\mathcal A}(\Lambda''), \partial_{\Lambda''})$ induced by the canonical bijection identifying the Reeb chords on $\Lambda$ with the Reeb chords on $\Lambda''$.
\end{Rem}

In the following theorem all  maps will be defined, as usual, using a generic path $J_\bullet$ of admissible almost complex structures. 
\begin{Thm}
\label{thm:twocopyends}
 Let $\Sigma$ be an exact Lagrangian cobordisms from $\Lambda^-$ to $\Lambda^+$ and $h$ a generic cylindrical Hamiltonian.
\begin{enumerate}
\item If $\Phi_{\Sigma} \colon {\mathcal A}(\Lambda^+) \to {\mathcal A}(\Lambda^-)$ and  $\Phi_{\Sigma_{\epsilon h}}  \colon {\mathcal A}(\Lambda^+_{\epsilon h}) \to {\mathcal A}(\Lambda^-_{\epsilon h})$ are the DGA morphisms  induced by $\Sigma$ and $\Sigma_{\epsilon h}$ respectively, then, for $\epsilon >0$ sufficiently small,
$$\Phi_{\Sigma}= \Phi_{\Sigma_{\epsilon h}}$$
after identifying ${\mathcal A}(\Lambda^\pm)$ with ${\mathcal A}(\Lambda^\pm_{\epsilon h})$ as in Remark \ref{aspettando il caldaista}.

\item  Let $\varepsilon_i$, $i=0,1$, be augmentations of ${\mathcal A}(\Lambda^-)$. If $\partial_th >0$ for $t \gg0$ and $\epsilon >0$ is small enough,
there is a commutative diagram
$$\xymatrix{
LCC^\bullet_{\varepsilon_0, \varepsilon_1}(\Lambda^-, \Lambda^-_{\epsilon h}) \ar[r]^-{d_{+-}} \ar[d]_{\cong} & LCC^\bullet_{\Phi_{\Sigma} \circ \varepsilon_0, \Phi_{\Sigma_{\epsilon h}} \circ \varepsilon_1}(\Lambda^+, \Lambda^+_{\epsilon h}) \ar[d]^\cong \\
LCC^\bullet_{\varepsilon_0, \varepsilon_1}(\Lambda^-) \ar[r]^{\Phi_{\Sigma}^{\varepsilon_0,\varepsilon_1}} & LCC^\bullet_{\Phi_{\Sigma} \circ \varepsilon_0, \Phi_{\Sigma} \circ \varepsilon_1}(\Lambda^+)
}$$
\end{enumerate}
where the vertical maps are the isomorphisms coming from Proposition \ref{prop:twocopy}
and (1) and $\Phi_{\Sigma}^{\varepsilon_0,\varepsilon_1}$ is the adjoint of the bilinearised map induced by the DGA morphism $\Phi_{\Sigma}$, as described in \cite{Ekhoka}.
\end{Thm}

\begin{proof}

Both results follow from Proposition \ref{prop:twocopy} and \cite[Theorem 2.15]{Floer_Conc}, where the latter provides the necessary identifications of pseudoholomorphic strips with boundary on the cobordism $\Sigma$ with those with boundary on $\Sigma_{\epsilon h}$ and on its two-copy $\Sigma\cup \Sigma_{\epsilon h}$. To that end, an admissible almost complex structure which coincides with cylindrical lifts in the prescribed subsets must be used.
\end{proof}
\begin{Lem}\label{lemma: fastidio}
Let $\Sigma$ be an exact Lagrangian cobordism from $\Lambda^-$ to $\Lambda^+$ and $h$ a cylindrical Hamiltonian function. For $\epsilon>0$ small enough:
\begin{enumerate}
\item If $\Lambda^-_{\epsilon h} = \Lambda^-_{+ \epsilon}$, then  $d_{0-}=0$;
\item If $\Lambda^-_{\epsilon h} = \Lambda^-_{- \epsilon}$, then $d_{-0}=0$.
\end{enumerate}
\end{Lem}
\begin{proof}
By Equation \eqref{eq:14}, the action of an intersection point $q \in \Sigma \cap
\Sigma_{\epsilon h}$ is
\begin{equation} \label{ne ho abbastanza}
\mathfrak{a}(q)= \epsilon (h'(q) - h(q)).
\end{equation}
On the other hand, if $\Lambda^-_{\epsilon h} = \Lambda^-_{+ \epsilon}$ there is lower bound independent of $\epsilon$ on the action of the Reeb chords from $\Lambda^-_{\epsilon h}$ to $\Lambda^-$. Therefore, for $\epsilon>0$ small enough, $\mathfrak{a}(p)< \mathfrak{a}(\gamma^-)$ for every intersection point $p \in \Sigma_0 \cap \Sigma_1$ and Reeb chord $\gamma^- \in {\mathcal R}(\Lambda^-_{\epsilon h}, \Lambda^-)$. Then Lemma
 \ref{lem: when things work} implies that $d_{0-}=0$. When $\Lambda^-_{\epsilon h} = \Lambda^-_{- \epsilon}$  there is lower bound independent of $\epsilon$ on the action of the Reeb chords from $\Lambda^-$ to $\Lambda^-_{\epsilon h}$ and again we can apply Lemma \ref{lem: when things work} to conclude that $d_{-0}=0$.
\end{proof}
\begin{Thm}
\label{thm:twocopy}
Let $\Sigma$ be an exact Lagrangian cobordism, $\varepsilon_0^-$ and $\varepsilon_1^-$ augmentation of $\Lambda^-$, and $h$ a generic cylindrical Hamiltonian
 such that $f:=h|_{\Sigma} \colon \Sigma \to \R$ is a Morse function. Assume we defined the differential $\mathfrak{d}_{\varepsilon_0^-, \varepsilon_1^-}$ on $Cth_\bullet(\Sigma, \Sigma_{\epsilon h})$ using an almost complex structure which induces a Riemannian metric $g$ on $\Sigma$ making $(f,g)$ into a Morse-Smale pair.
Then, for an appropriately chosen Maslov potential on $\Sigma$ and $\epsilon >0$ sufficiently small, there is an isomorphism of complexes
\[ CF_\bullet (\Sigma, \Sigma_{\epsilon h}) = C^{\OP{Morse}}_{n+1-\bullet}(f).\]
 Moreover, if $\Sigma$ is relatively Pin, then the above identification holds with coefficients in $\Z$ and if $\Sigma$ is not graded, it holds on the level of ungraded complexes.
\end{Thm}

\begin{proof}
First, we observe that $d_{00}^2=0$. In fact $\mathfrak{d}_{\varepsilon_0^-, \varepsilon_1^-}^2=0$ implies that $d_{00}^2= d_{0-}d_{-0}$, and either $d_{0-}=0$ or $d_{-0}=0$ for $\epsilon >0$ sufficiently small by Lemma \ref{lemma: fastidio}. Next, we argue that $d_{00}$ counts only Floer strips without boundary punctures asymptotic to pure Reeb chords. This is a consequence of Equation \eqref{eq:9} because the action of the pure chords of $\Lambda^-$ and $\Lambda^-_{\epsilon h}$ is bounded from below, while the action of the intersection points can be made arbitrarily small by Equation \eqref{ne ho abbastanza}.

At this point, since we have shown that we can work in a compact region, the isomorphism of Morse and Floer complexes is standard, going back to the original work of Floer \cite{FloerHFlag}. For the current setting, the analogous computation made in \cite[Theorem 6.2]{LiftingPseudoholomorphic} is also relevant.
\end{proof}
\begin{Rem}\label{rem:morsefloer}
Observe that the function $f=h|_{\Sigma}$ has no critical points on $\Sigma \setminus int( \overline{\Sigma})$. We denote by $\partial_h \overline{\Sigma}$ the portion of $\partial \overline{\Sigma}$ on which $\nabla f$ points inward; then
\[H(C^{\OP{Morse}}_{n+1-\bullet}(f))=H_{n+1-\bullet}(\overline{\Sigma},\partial_h \overline{\Sigma}).\]
\end{Rem}

\subsection{Proof of the exact sequences} \label{sec: proof of exact sequences}

\begin{proof}[Proof of Theorem \ref{thm:lespair}]
Let $h$ be a generic cylindrical Hamiltonian such that $\partial_th=e^t$ for $t \ne [A, B]$. In this case $\Lambda^\pm_{\epsilon h} = \Lambda^{\pm}_{+ \epsilon}$ and therefore,
by Lemma \ref{lemma: fastidio}, $d_{0-}=0$.

Then we use Proposition \ref{prop:twocopy}, Theorem \ref{thm:twocopyends}, Theorem \ref{thm:twocopy} and Remark \ref{rem:morsefloer} to identify the exact sequence \eqref{eq:17} of Lemma \ref{lem:posnegLES} with the exact sequence of Theorem \ref{thm:lespair}.
\end{proof}

\begin{proof}[Proof of Theorem \ref{thm:lesduality}]
Let $h$ be a generic cylindrical Hamiltonian such that $\partial_th=- e^t$ for
$t \le A$ and $\partial_th= e^t$ for $t \ge B$. In this case $\Lambda^-_{\epsilon h} = \Lambda^-_{- \epsilon}$ and $\Lambda^+_{\epsilon h} = \Lambda^+_{+ \epsilon}$, so, by Lemma \ref{lemma: fastidio}, $d_{-0}=0$. Then we use Proposition \ref{prop:twocopy}, Theorem \ref{thm:twocopyends},  Theorem \ref{thm:twocopy} and Remark \ref{rem:morsefloer} to identify the exact sequence \eqref{eq:18} of Lemma \ref{lem:posnegLES} with the exact sequence of Theorem \ref{thm:lesduality}.
\end{proof}

\begin{proof}[Proof of Theorem \ref{thm:lesmayer-vietoris}]
\label{sec:push-inducing-mayer}
We consider a generic cylindrical Hamiltonian $h$ obtained by perturbing the function
$\widetilde{h}$ as in Figure \ref{fig:pushoff-mixed}. Recall that the perturbation $h$ of $\widetilde{h}$ is realised by adding a term $e^tg(t,p)$ for a $C^2$-small function $g \colon \R \times P \to \R$ as done earlier in this section; the function $g(t,p)$ for $t \le A$ can in fact can be constructed near $\Pi_{\OP{Lag}}(\Lambda^-)$ by a suitable Morse function on $\Lambda^-$ which can be extended to the immersed Weinstein neighbourhood of the Lagrangian immersion (care has to be taken near the double points). For simplicity, we assume that the Morse function on $\Lambda^-$ has a unique local minimum and maximum.

The pair $(\Sigma, \Sigma_{\epsilon h})$ is V-shaped, and therefore all intersection points have positive action and $d_{-0}=0$. We decompose
$$C^{\bullet}(\Lambda^-,\Lambda_{\epsilon h}^-)=C_+^{\bullet}(\Lambda^-,\Lambda_{\epsilon h}^-) \oplus C_0^{\bullet}(\Lambda^-, \Lambda_{\epsilon h}^-)$$
 so that the ``plus'' summand is generated by those Reeb chords in ${\mathcal R}(\Lambda_{\epsilon h}^-, \Lambda^-)$ corresponding to Reeb chords of $\Lambda^-$, so that they have length bounded from below in terms of the minimal length of a Reeb chord on $\Lambda^-$, while the ``zero'' summand is generated by those Reeb chords in ${\mathcal R}(\Lambda_{\epsilon h}^-, \Lambda^-)$ corresponding to the critical points of a Morse function on $\Lambda^-$, which then have length roughly equal to $\epsilon>0$. See \cite[Section 3.1]{Duality_EkholmetAl}. Thus the intersection points have action of order $\epsilon$, the ``plus'' chords have action of order $e^{T}$ and the ``zero chords'' have action of order $\epsilon e^{-T}$ for some fixed $T \gg0$; see Section \ref{sec:action-energy}.

By action reasons there is a short exact sequence of complexes
\begin{equation}\label{ses}
0 \to C_0^{\bullet}(\Lambda^-, \Lambda_{\epsilon h}^-) \to CF_{- \infty}(\Sigma, \Sigma_{\epsilon h}) \to  C_+^{\bullet}(\Lambda^-,\Lambda_{\epsilon h}^-) \oplus CF^{\bullet}(\Sigma,\Sigma_{\epsilon h}) \to 0.
\end{equation}
If we take $\epsilon$ small enough we have $\mathfrak{a}(p)< \mathfrak{a}(\gamma^-)$ for every intersection point $p$ and ``plus'' chord $\gamma^-$, and therefore, by Equation \eqref{eq:9} there can be no holomorphic map between an intersection point and a ``plus'' chord. This shows that the direct sum  $C_+^{\bullet}(\Lambda^-,\Lambda_{\epsilon h}^-) \oplus CF^{\bullet}(\Sigma,\Sigma_{\epsilon h})$, a priori only of modules, is in fact a direct sum
of chain complexes. For this reason
$$H_i(C_+^{\bullet}(\Lambda^-,\Lambda_{\epsilon h}^-) \oplus CF^{\bullet}(\Sigma,\Sigma_{\epsilon h})) = H_i(C_+^{\bullet}(\Lambda^-,\Lambda_{\epsilon h}^-)) \oplus H_i(CF^{\bullet}(\Sigma,\Sigma_{\epsilon h})).$$

We rely on \cite[Theorem 3.6]{Duality_EkholmetAl} together with Proposition \ref{prop:twocopy} in order to make identifications
\begin{gather}
H^i(C_+^\bullet(\Lambda^-,\Lambda_{\epsilon h}^-)) \cong LCH^i_{\varepsilon^-_0,\varepsilon^-_1}(\Lambda^-), \label{paperoga} \\
H^i(C_0^\bullet(\Lambda^-,\Lambda_{\epsilon h}^-))  \cong H_{n-1-i}(\Lambda^-) \label{picodepaperis}
\end{gather}
and therefore, using also Equation \eqref{eq:19} and Theorem \ref{thm:twocopy} together with Remark \ref{rem:morsefloer}, we can identify the long exact sequence induced by the short exact sequence \eqref{ses} with the long exact sequence \eqref{eq:mayer-vietoris}.

The technique in the proof of \cite[Theorem 6.2(ii)]{LiftingPseudoholomorphic} shows that
the restriction of $d_{0-}$ to $C_0^\bullet(\Lambda^-,\Lambda_{\epsilon h}^-)$ fits into a commutative diagram

$$\xymatrix{
H^i(C_0^\bullet(\Lambda^-,\Lambda_{\epsilon h}^-)) \ar[r]^{d_{0-}} \ar[d]_{\cong} & HF^{i+2}(\Sigma, \Sigma_{\epsilon h}) \ar[d]^{\cong} \\
H_{n-1-i}(\Lambda_-) \ar[r]^{i_*} & H_{n-1-i}(\Sigma)
}$$
where $i_*$ is the map induced by the inclusion $i \colon \partial_- \overline{\Sigma} \to \overline{\Sigma}$ and the vertical maps come from \eqref{picodepaperis} and Theorem \ref{thm:twocopy} together with Remark \ref{rem:morsefloer} respectively.

To see the commutativity of this diagram, note that the map
$$i_* \colon H_\bullet(\Lambda^-) \to H_\bullet(\Sigma)$$
can be realised as the inclusion of the subcomplex of the Morse  complex of $f=h|_{\Sigma}$ generated by the critical points near $\{ t = a\}$ into the full Morse complex of $f$. See Figure \ref{fig:pushoff-mixed}.

It then suffices to show that $d_{0-}$ identifies $C_0(\Lambda^-, \Lambda^-_{\epsilon h})$ to the said subcomplex, as in the proof of Proposition \ref{prp:wrap} (in particular, see the computation at the end of the proof). The reason why this technique applies is that, in the subset $\{ t \le a+\epsilon\},$ the push-off $\Sigma_{\epsilon h}$ is obtained by slightly `wrapping' $\Sigma_{\epsilon e^t}$ using the negative Reeb flow.  See also Section \ref{sec:carefulpushoff} for a similar analysis.

Finally, the statements concerning the fundamental classes is a consequence of \cite[Theorem 5.5]{Duality_EkholmetAl}, a result which is only valid if we are working with $\varepsilon^-_0=\varepsilon^-_1=\varepsilon$. Namely, the latter result shows that
the shortest chord in $C_0(\Lambda^-, \Lambda_{\epsilon h}^-)$ defines a nontrivial class in $LCH^{-1}_{\varepsilon,\varepsilon}(\Lambda^-,\Lambda^-_{\epsilon h})$. Similarly, under the additional assumption that $\Lambda^-$ is \emph{horizontally displaceable},
the differential of the longest chord in $C_0(\Lambda^-, \Lambda_{\epsilon h}^-)$ is an element of $C^\bullet_+(\Lambda^-, \Lambda^-_{\epsilon h})$ which defines a non-zero class in
$$H^n(C^\bullet_+(\Lambda^-,\Lambda_{\epsilon h}^-)) \cong  LCH^{n}_{\varepsilon,\varepsilon}(\Lambda^-)$$
called the \emph{fundamental class in Legendrian contact cohomology} (also see Section \ref{sec:fund-class-twist}).

With the above considerations, we have  proved Theorem \ref{thm:lesmayer-vietoris}.
\end{proof}
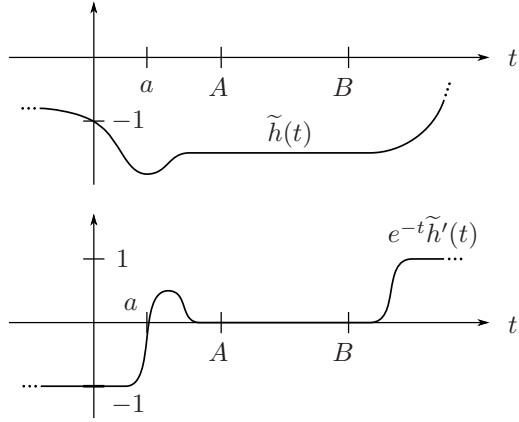
\begin{figure}
\centering
\labellist
\pinlabel $\widetilde{h}(t)$ at 106 108
\pinlabel $a$ at 52 125
\pinlabel $A$ at 79 125
\pinlabel $B$ at 126 125
\pinlabel $1$ at 43 60
\pinlabel $-1$ at 45 5
\pinlabel $-1$ at 45 112
\pinlabel $t$ at 190 136
\pinlabel $t$ at 190 35
\pinlabel $e^{-t}\widetilde{h}'(t)$ at 160 70
\pinlabel $a$ at 46 43
\pinlabel $A$ at 79 25
\pinlabel $B$ at 126 25
\endlabellist
\includegraphics{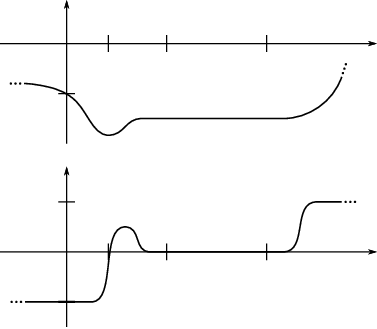}
\caption{The function $\widetilde{h}$. The corresponding Hamiltonian vector field is  $e^{-t}\widetilde{h}'(t)\partial_z$.}
\label{fig:pushoff-mixed}
\end{figure}
\subsection{Seidel's isomorphism}
\label{sec:seidel}
We end this section by reinterpreting the definition of Seidel's isomorphism, originally introduced in~\cite{Ekholm_FloerlagCOnt}, in the context of our theory. This isomorphism will be important when later considering the fundamental class in Legendrian contact homology in relation to a filling; see Section \ref{sec:fundamentalclass}.

Let $\Sigma$ be an exact Lagrangian cobordism from $\Lambda^-$ to $\Lambda^+$ and let $h \colon \R \times Y \to \R$ be a generic cylindrical Hamiltonian such that $\Lambda^+_{\epsilon h} = \Lambda^+_{+ \epsilon}$ and $\Lambda^-_{\epsilon h} = \Lambda_{-\epsilon}^-$.  Combining Theorem \ref{thm:twocopy} and Proposition \ref{prop:twocopy}, for every pair of augmentations $\epsilon_0^-$ and $\epsilon_1^-$ of $\Lambda^-$ we obtain a morphism
\[G^{\varepsilon^-_0,\varepsilon^-_1}_\Sigma \colon C^{\OP{Morse}}_{n+1-\bullet}(f) \to LCC^{\bullet-1}_{\varepsilon^+_0,\varepsilon^+_1}(\Lambda^+;R),\]
induced by the term $d_{+0}$ in the differential of $(\Cth(\Sigma,\Sigma_{\epsilon h}),\mathfrak{d}_{\varepsilon^-_0,\varepsilon^-_1})$. In this case, by Remark \ref{rem:morsefloer}, $H(C^{\OP{Morse}}_{n+1-\bullet}(f)) \cong H_{n+1-\bullet}(\Sigma)$. Observe that $d_{-0}=0$ by Lemma \ref{lemma: fastidio} and therefore $d_{+0}$ is a chain map. Then we have proved the following lemma.
\begin{Lem}
The map $G^{\varepsilon^-_0,\varepsilon^-_1}_\Sigma$ is a chain map which, in the case when the negative end of $\Sigma$ is empty, is a quasi-isomorphism.
\end{Lem}

\section{Twisted coefficients, $L^2$-completions, and applications}
\label{sec:l2-legendr-cont}

In order to deduce information about the fundamental group of an exact Lagrangian cobordism it is necessary to introduce a version of the Cthulhu complex with coefficients twisted by the fundamental group, analogous to that defined for Lagrangian Floer homology in \cite{KFloer} by Sullivan and in \cite{Damian_Lifted} by Damian. Since it is not possible, in general, to make sense of the rank of a module over a group ring, it will also be necessary to introduce an $L^2$-completion of this complex. So-called $L^2$-coefficients were first considered by Atiyah in \cite{AtiyahL2}. We start by describing the version of the complex with twisted coefficients, and we also introduce  the fundamental class in this setting. The fundamental class will be crucial for the proof of Theorems \ref{thm:pi_1carclass} (see Section \ref{sec:proofpi_1carclass}). We then continue by defining the $L^2$-completion of this complex, for which we recall some basic properties. The proof of Theorem \ref{thm:l2rigidity} will use this theory (see Section \ref{sec:proofl2rigidity}).

\subsection{The Chekanov-Eliashberg algebra with twisted
coefficients}
\label{sec:lch-with-homological}
Legendrian contact homology with twisted coefficients has previously been considered in \cite{EkholmSmith}, and a detailed account is currently under development in \cite{Albin}; see also \cite{NonCommAug}.
Here we consider a version of the Chekanov-Eliashberg algebra for a connected Legendrian submanifold $\Lambda \subset P\times \R$ with coefficients the group ring $R[\pi_1(\Lambda)]$, where $R$ is a commutative ring.

Fix a base point $* \in \Lambda$ and write $\pi_1(\Lambda):=\pi_1(\Lambda,*)$ for short. Let $A$ be a unital, not necessarily commutative, ring for which:
\begin{itemize}
\item There is a ring homomorphism $i  \colon R[\pi_1(\Lambda)]\to A$. This induces an $R[\pi_1(\Lambda)]-R[\pi_1(\Lambda)]$-bimodule structure on $A$;
\item There is an augmentation homomorphism $\mathcal{E}  \colon A\rightarrow R$ such that $\Pi:=\mathcal{E}\circ i$ is the standard augmentation $\Pi  \colon R[\pi_1(\Lambda)]\to R$.
  \end{itemize}
By abuse of notation we will see any element $a$ in $R[\pi_1(\Lambda)]$ as an element of $A$ by identifying it with its image under the ring homomorphism $i$.

As an example, we can take $A=R[G]$ the group ring of $G$ over $R.$ Any group homomorphism $\pi_1(\Lambda)\rightarrow G$ then induces a ring homomorphism $R[\pi_1(\Lambda)]\rightarrow R[G]=A$ and the augmentation corresponds to the canonical ring homomorphism $A=R[G]\rightarrow R$. In case when $G=\{1\}$, the construction we describe below will recover the standard Chekanov-Eliashberg DGA.

We denote by  $\mathcal{A}_A(\Lambda)$ (or $\mathcal{A}_G(\Lambda)$ if $A=R[G]$) the tensor algebra over $A$ of the free $A-A$-bimodule generated by the Reeb chords of $\Lambda$. If $A=R$, we denote $\mathcal{A}(\Lambda) = \mathcal{A}_R(\Lambda)$; this is the usual Chekanov-Eliashberg algebra. Recall that $\mathcal{A}_A(\Lambda)$ is generated, as an $R$-module, by elements of the form $$a_1\gamma_1\otimes a_2\gamma_2\otimes\cdots\otimes a_j\gamma_ja_{j+1},$$ where
$\gamma_1, \ldots, \gamma_j$ are Reeb chords and  $a_1,\dots,a_{j+1}\in A$. (The case $j=0$, i.e.\ no chord) is also permitted.

For any Reeb chord $\gamma$ of $\Lambda$, we fix \textit{connecting paths} $\ell_\gamma^e$ and $\ell_\gamma^s$ on $\Lambda$  from the end point and the starting point of $\gamma$, respectively, to the base point $*$. (Such paths exist because we assume that $\Lambda$ is connected.)  A punctured pseudoholomorphic disc $u\in \mathcal{M}(\gamma^+;\gamma_1^-,\gamma_2^-,\ldots,\gamma_k^-)$ determines an element $\mathbf{c}_u$ of ${\mathcal A}_A(\Lambda)$ via the following procedure. Let  $S$ be the domain of $u$ and denote by $\partial_0 S,\ldots, \partial_k S$ the connected components of $\partial S$ ordered counterclockwise starting from the puncture corresponding to $\gamma^+$. We denote by $p$ the canonical projection $\R \times Y \to Y$.
\begin{itemize}
\item For $j\in \{1,\ldots, k-1\}$, we denote by $a_j$ the based loop $(\ell_{\gamma^-_{j}}^s)^{-1}*(p\circ u|_{\partial_j})*\ell_{\gamma^-_{j+1}}^e$;
\item For $j=0$, we denote by $a_j$ the based loop $(\ell_{\gamma^+}^e)^{-1}*(p\circ u|_{\partial_0})*\ell_{\gamma^-_{1}}^e$; and
\item For $j=k$, we denote by $a_j$ the based loop $(\ell_{\gamma^-_k}^s)^{-1}*(p\circ u|_{\partial_k})*\ell_{\gamma^+}^s$.
\end{itemize}
From now on we assume $A=R[\pi_1(\Lambda)]$; the general case is obtained by the change of coefficients induced by $i \colon R[\pi_1(\Lambda)] \to A$. The element $\mathbf{c}_u$ is then given by
$$\mathbf{c}_u=a_0\gamma_1^- a_1\otimes \gamma_2^- a_2\otimes\cdots\otimes \gamma_k^- a_k.$$
The Chekanov-Eliashberg differential on ${\mathcal A}_A(\Lambda)$ is defined on generators by the formula

$$\partial(\gamma^+)=\sum_{\gamma_1,\ldots,\gamma_k}\sum_{u\in\mathcal{M}(\gamma^+;\gamma_1^-,\gamma_2^-,\ldots,\gamma_k^-)} \mathbf{c}_u,$$
where the sum is taken over the rigid components of the moduli spaces. The differential is then extended to the whole algebra using the Leibniz rule.

By an \emph{augmentation} of the Chekanov-Eliashberg DGA we mean a homomorphism of $R[\pi_1(\Lambda)]-R[\pi_1(\Lambda)]$-bimodules  $\varepsilon \colon {\mathcal A}_A(\Lambda) \rightarrow A$ which is a unital ring homomorphism and  $\varepsilon\circ\partial=0$.
\begin{Rem}
\begin{enumerate}
\item An augmentation in this setting is still determined by its values on the Reeb chord generators.
\item Any homomorphism $G \to H$ of groups induces a unital DGA morphism $r \colon \mathcal{A}_G(\Lambda) \to \mathcal{A}_H(\Lambda)$. In particular, when $H$ is the trivial group, we get a canonical DGA homomorphism $r \colon \mathcal{A}_G(\Lambda) \to \mathcal{A}(\Lambda)$. The pre-composition $\widetilde{\varepsilon}:=\varepsilon \circ r$ of an augmentation $\varepsilon$ of $\mathcal{A}(\Lambda)$ is clearly an augmentation of $\mathcal{A}_G(\Lambda)$; this augmentation will be called the \emph{lift of $\varepsilon$}.
\item An exact Lagrangian cobordism $\Sigma$ from $\Lambda^-$ to $\Lambda^+$  induces a unital DGA homomorphism $\widetilde{\Phi}_\Sigma \colon \mathcal{A}_{\pi_1(\Sigma)}(\Lambda^+) \to \mathcal{A}_{\pi_1(\Sigma)}(\Lambda^-)$ with twisted coefficients. In particular, an exact Lagrangian filling induces an augmentation in the group ring of its fundamental group.
\end{enumerate}
\end{Rem}

For any pair of augmentations, the linearisation procedure gives rise to a differential $d_{\varepsilon_0,\varepsilon_1}: LCC_\bullet^{\varepsilon_0,\varepsilon_1}(\Lambda;A) \to LCC_\bullet^{\varepsilon_0,\varepsilon_1}(\Lambda;A)$ on the free $A-A$-bimodule generated by the Reeb chords in the usual way. The map $d_{\varepsilon_0,\varepsilon_1}$ is in this situation a bimodule homomorphism. We again denote the resulting homology by $LCH_\bullet^{\varepsilon_0,\varepsilon_1}(\Lambda;A)$ and call it the \textit{bilinearised Legendrian contact homology with twisted coefficients}.

Moreover, when $\Lambda=\Lambda_0 \sqcup \Lambda_1$ and $\varepsilon_i$ is an augmentation of $\Lambda_i$ for $i=0,1$, we can again define the complex $LCC^\bullet_{\varepsilon_0,\varepsilon_1}(\Lambda_0,\Lambda_1;A)$ which, as a module, is the free right $A$-module generated by the Reeb chords starting on $\Lambda_1$ and ending on $\Lambda_0$, as well as the corresponding cohomology groups $LCH^\bullet_{\varepsilon_0,\varepsilon_1}(\Lambda_0,\Lambda_1;A)$. The result \cite[Proposition 2.7]{Floer_Conc} carries over immediately to this setting, and thus the identification
\[ LCH^\bullet_{\varepsilon_0,\varepsilon_1}(\Lambda,\Lambda';A) = LCH^\bullet_{\varepsilon_0,\varepsilon_1}(\Lambda;A)\]
holds on the level of homology (again, for a suitable small push-off $\Lambda'$ of $\Lambda$, together with a suitable lifted almost complex structure).

 \begin{Rem}
If $\varepsilon_1$ takes values in $R$, then $LCC^\bullet_{\varepsilon_0,\varepsilon_1}(\Lambda,\Lambda';A)$ is a complex of free \emph{right} $A$-modules. In particular, the differential is a right $A$-module morphism (it is defined by multiplication on the left).
This is relevant for the next section, where we will twist coefficients using $R[\pi_1(\Sigma_0)]$ in a way so that the augmentation $\varepsilon_1$ still will take values in $R$.
\end{Rem}

\subsection{The Cthulhu complex with twisted coefficients}
\label{sec: twisted cthulhu}
Let $R$ be a unital commutative ring. We are now ready to define the Cthulhu complex with twisted coefficients for a pair of exact Lagrangian cobordisms $(\Sigma_0,\Sigma_1)$. 
For Lagrangian intersection Floer homology, such a construction has previously been carried out in \cite{KFloer} and \cite{Damian_Lifted}. This was subsequently generalised to
Wrapped Floer homology in \cite[Section 4.2]{Floer_Conc}. 

Let $\Sigma_i \subset \R \times Y$ be exact Lagrangian cobordisms from $\Lambda_i^-$ to $\Lambda^+_i$, $i=0,1$, where both $\Sigma_0$ and $\Lambda^-_0$ are connected.
The ring homomorphisms $R[\pi_1(\Lambda_0^\pm)] \to R[\pi_1(\Sigma_0)]$ coming from the inclusion maps $\{\pm T\}\times\Lambda_0^\pm \rightarrow \Sigma_0$ induce non-free $R[\pi_1(\Lambda_0^+)]-R[\pi_1(\Lambda_0^+)]$ and $R[\pi_1(\Lambda_0^-)]-R[\pi_1(\Lambda_0^-)]$-bimodule structures on $R[\pi_1(\Sigma_0)]$. (Of course all mixed structures are also possible.)

Let $u\in\mathcal{M}(x;\boldsymbol{\zeta},y,\boldsymbol{\delta})$ be a pseudoholomorphic strip involved in the Cthulhu differential $\mathfrak{d}_{\varepsilon_0^-,\varepsilon_1^-}$ as defined in Section \ref{subsec: mixed moduli spaces}.  If $S$ is the domain of $u$ and $\partial_0S,\ldots,\partial_kS$ are the connected components of $\partial S$ which are mapped to $\Sigma_0$, we
 associate to each arcs $\partial_jS$  an element $a_j \in \pi_1(\Sigma_0)$ in the same manner as in the definition of the differential of the Chekanov-Eliashberg algebra with twisted coefficients described above. These paths, together with augmentations $\varepsilon_0^- \colon \mathcal{A}_{\pi_1(\Sigma_0)}(\Lambda^-_0) \to R[\pi_1(\Sigma_0)]$ and $\varepsilon_1 \colon \mathcal{A}(\Lambda^-_1) \to R$, now determine an element
$$c^{\varepsilon_0^-,\varepsilon_1^-}_u=a_1\varepsilon_0^-(\zeta_1)a_2\varepsilon_0^-(\zeta_2)\cdots a_{k-1}\varepsilon_0^-(\zeta_{k-1})a_k\varepsilon_1^-(\boldsymbol{\delta}) \in R[\pi_1(\Sigma_0)].$$
 This construction allows us to define the Cthulhu differential $\mathfrak{d}_{\varepsilon_0^-,\varepsilon_1^-}$ on the non-free $R[\pi_1(\Sigma_0)]$--$R[\pi_1(\Sigma_0)]$-bimodule
$$\Cth(\Sigma_0,\Sigma_1;R[\pi_1(\Sigma_0)]):=\Cth(\Sigma_0,\Sigma_1)\otimes_{R}R[\pi_1(\Sigma_0)]$$
that is induced by the standard (non-free) bimodule structure on the group ring $R[\pi_1(\Sigma_0)]$ by ring multiplication from left and right. First, when $y$ is either a intersection point or a Reeb chord from $\Lambda^-_0$ to $\Lambda^-_1$, we define
$$\mathfrak{d}_{\varepsilon_0^-,\varepsilon_1^-}(y)=\sum_x\sum_{u \in \mathcal{M}(x;\boldsymbol{\zeta},y,\boldsymbol{\delta})} c^{\varepsilon_0^-,\varepsilon_1^-}_ux,$$
where the sum is taken over the rigid components of the moduli space. The formula when $y$ is a Reeb chord from $\Lambda^+_1$ to $\Lambda^+_0$ is similar, but involves the pull-backs $\varepsilon_0^- \circ \widetilde{\Phi}_{\Sigma_0}$ and $\varepsilon_1^- \circ \Phi_{\Sigma_1}$ of the augmentations under the DGA homomorphism induced by the cobordisms with and without twisted coefficients, respectively. The differential is then extended to all of $\Cth(\Sigma_0,\Sigma_1;R[\pi_1(\Sigma_0)])$ as a \emph{right} $R[\pi_1(\Sigma_0)]$-module homomorphism.

The techniques in Section \ref{sec:acyclicity} can be used to prove the following theorem.
\begin{Thm} \label{thm:d2lifted}
The map
$$\mathfrak{d}_{\varepsilon_0^-,\varepsilon_1^-}:\Cth(\Sigma_0,\Sigma_1;R[\pi_1(\Sigma_0)])\rightarrow \Cth(\Sigma_0,\Sigma_1;R[\pi_1(\Sigma_0)])$$
satisfies $\mathfrak{d}_{\varepsilon_0^-,\varepsilon_1^-}^2=0$, % 
and moreover 
$$H(\Cth(\Sigma_0,\Sigma_1;R[\pi_1(\Sigma_0)]),\mathfrak{d}_{\varepsilon_0^-,\varepsilon_1^-})=0.$$
\end{Thm}

\begin{proof}
The proof is similar to the  ones in Sections \ref{sec:Cthulu-differential} and \ref{sec:proof-acyclicity}. We only need to observe the following: if $u$ and $v$ are holomorphic strips with matching asymptotics, so that they can be glued to a strip $u*v$, then  $c^{\varepsilon_0^-,\varepsilon_1^-}_{u*v}=c^{\varepsilon_0^-,\varepsilon_1^-}_u \cdot c^{\varepsilon_0^-,\varepsilon_1^-}_v$ is satisfied in the group ring.
\end{proof}

In view of the above theorem, the computations in Section \ref{sec:long-exact-sequence} can be carried over immediately to the case of twisted coefficients. We proceed to explicitly describe the long exact sequence analogous to \eqref{leqtr} in Theorem \ref{thm:lespair}. Let $\Sigma$ be an exact Lagrangian cobordism from $\Lambda^-$ to $\Lambda^+$. Let $\varepsilon^-_{0}$ and $\varepsilon^-_1$ be two augmentations of $\mathcal{A}_{\pi(\Sigma)}(\Lambda^-)$ and $\mathcal{A}(\Lambda^-)$ into $R[\pi(\Sigma)]$ and $R$, respectively. Further, we consider the
pull-back $\varepsilon^+_0:=\varepsilon^-_0 \circ \widetilde{\Phi}_\Sigma$ and $\varepsilon^+_1:=\varepsilon^-_1 \circ \Phi_\Sigma$ of these augmentations.
\begin{Rem}
It is important to note that $\varepsilon_0^+$ need \emph{not} be the lift of an augmentation into $R$ in general, even in the case when $\varepsilon_0^-$ is.
\end{Rem}

Writing $\widetilde{\Sigma}$ for the universal cover of  the compactification $\overline{\Sigma}$ of $\Sigma$,  there is a long exact sequence:
\begin{equation}
    \label{leqtrlifted}
\xymatrixrowsep{0.15in}
\xymatrixcolsep{0.15in}
\xymatrix{
       \cdots\ar[r]& LCH^{k-1}_{\varepsilon^+_0,\varepsilon_1^+}(\Lambda^+;R[\pi_1(\Sigma)]) \ar[d] & & & \\
& H_{n+1-k}({\widetilde{\Sigma}},\partial_- {\widetilde{\Sigma}};R) \ar[r] & LCH^{k}_{\varepsilon^-_0,\varepsilon^-_1}(\Lambda^-; R[\pi_1(\Sigma)]) \ar[d] & \\
& & LCH^{k}_{\varepsilon^+_0,\varepsilon_1^+}(\Lambda^+;R[\pi_1(\Sigma)])\ar[r] &\cdots} .
\end{equation}
The identification of the topological term $H_{n+1-k}({\widetilde{\Sigma}},\partial_- {\widetilde{\Sigma}};R)$ is proven in the same manner as before (see Theorem \ref{thm:twocopy}), while making the observation that the Morse homology of a manifold with coefficients twisted by its fundamental group computes the homology of its universal cover.

Finally, we point out that
$$ LCH^{k}_{\varepsilon^-_0,\varepsilon^-_1}(\Lambda^-; R[\pi_1(\Sigma)])=LCH^{k}_{\varepsilon^-_0,\varepsilon^-_1}(\Lambda^-) \otimes R[\pi_1(\Sigma)]$$
is satisfied when $\Lambda^-$ is simply connected.

\subsection{Augmentations in finite-dimensional non-commutative algebras.}
\label{sec:augm-finite-dimens}
In this subsection we describe how augmentations into non-commutative unital algebras, as used by the second and fourth authors in \cite{EstimNumbrReebChordLinReprCharAlg}, fit into the framework of this article. Since we will be interested in computing the ranks of linearised complexes, we will restrict ourselves to the case when the augmentations take values in algebras which are finite-dimensional as vector spaces over a field $\F$.

A finite-dimensional augmentation of the Chekanov-Eliashberg algebra is a unital DGA homomorphism
\[ \varepsilon \colon \mathcal{A}(\Lambda) \to A, \]
where $A$ is a \emph{not necessarily commutative} unital algebra over the ground field $\F$, seen as a DGA with trivial differential and concentrated in degree $0$. Here $\F$ denotes the field that was used as coefficient ring for $\mathcal{A}(\Lambda)$. Recall that the existence of such a (graded) augmentation is equivalent to the existence of a finite-dimensional representation of the so-called \emph{characteristic algebra}, which is defined as the quotient algebra $\mathcal{A}(\Lambda)/\langle \partial(\mathcal{A})\rangle$ by the two-sided ideal generated by the boundaries (see \cite{Ngcomputable}).

Given two such augmentations
\[\varepsilon_i \colon \mathcal{A}(\Lambda^-_i) \to A_i, \:\:i=0,1,\]
we can define the linearised and Floer complexes as free left $A_0 \otimes_\F (A_1)^{\OP{op}}$-modules or,  put differently, as free $A_0-A_1$ bimodules.

To construct the differentials in this setting one proceeds as in \cite{EstimNumbrReebChordLinReprCharAlg}.
The differentials are of the form
\[d((a_0\otimes a_1)x)=\sum_{y}\sum_{\boldsymbol{\delta}^-,\boldsymbol{\zeta}^-}\#\mathcal{M}(y;\boldsymbol{\delta}^-,x,\boldsymbol{\zeta}^-)\cdot\varepsilon_0(\boldsymbol{\delta}^-)a_0 \otimes a_1\varepsilon_1(\boldsymbol{\zeta}^-)\cdot y,\]
where $x$ and $y$ denote either intersection points or Reeb chords, and where $a_i \in A_i$, $i=0,1$.
\begin{Rem}
This convention tells us that the differential is defined by multiplication of  $A_0 \otimes_\F (A_1)^{\OP{op}}$  from the \emph{left}, and is hence a morphism of \emph{right}  $A_0 \otimes_\F (A_1)^{\OP{op}}$-modules.
\end{Rem}

The long exact sequences in homology for these bimodules now follow verbatim from the proofs in the case when the augmentation is taken into $\F$. It is important to notice that all the complexes above are  finite dimensional as vector spaces over $\F$. More precisely,
\[ \dim_\F LCC^\bullet(\Lambda^\pm_0,\Lambda^\pm_1)=|\mathcal{R}(\Lambda_1^\pm,\Lambda^\pm_0)| \cdot \dim_\F A_0 \cdot \dim_\F A_1,\]
while
\[ \dim_\F H_\bullet(X,Y;A_0 \otimes_\F (A_1)^{\OP{op}})=\dim_\F H_\bullet(X,Y;\F)\dim_\F(A_0)\dim_\F(A_1)\]
holds by the universal coefficients theorem.

\begin{Ex}
There are examples of Legendrian submanifolds which admit augmentations into finite-dimensional non-commutative unital algebras, but which do not admit any augmentation into any commutative unital algebra. We refer to Part (1) of Example \ref{rem:examplesthatanswernonsymqofcsandt} below for  Legendrian torus knots, found by Sivek in \cite{TheContHomofLegKNotswithMAXTBI}, which admit augmentations into the matrix algebra $\mathrm M_{2}(\Z_2)$. The second and the fourth author later used these examples in order to construct plenty of Legendrian submanifolds inside contact spaces $(\R^{2n+1},\xi_{\OP{std}})$ for arbitrary $n\in\mathbb{N}$ whose Chekanov-Eliashberg algebras admit augmentations into $\mathrm M_{2}(\Z_2)$, but not into any commutative algebra.
\end{Ex}

\subsection{The fundamental class and twisted coefficients}
\label{sec:fund-class-twist}
In this section we will introduce the fundamental class in the setting of twisted coefficients. We will prove that this class coincides with a twisted coefficient version of the fundamental class introduced in \cite{Duality_EkholmetAl}. We will also prove that it is functorial under exact Lagrangian cobordisms. In Section \ref{sec:proofpi_1carclass} we will use this functoriality to prove Theorem \ref{thm:pi_1carclass}. In the following we let $\Sigma$ be a connected exact Lagrangian cobordism from $\Lambda^-$ to $\Lambda^+$, where the latter Legendrian submanifolds are connected as well.

\subsubsection{The definition of the fundamental class}
\label{sec:fundamentalclass}
Recall that in Section \ref{sec:seidel} we defined the map
\[G^{\varepsilon^-_0,\varepsilon^-_1}_\Sigma \colon H_\bullet(\Sigma;R) \to LCH^{n-\bullet}_{\varepsilon^+_0,\varepsilon^+_1}(\Lambda^+;R).\]
The underlying chain map of $G^{\varepsilon^-_0,\varepsilon^-_1}_\Sigma$ lifts to the corresponding complexes with twisted coefficients. Namely, we define the chain map by the same count of pseudoholomorphic strips, but where the count takes the homotopy class of the boundary of the strips into account in the manner described in Section \ref{sec: twisted cthulhu}. The lifted map on homology will be denoted by
\[\widetilde{G}^{\varepsilon^-_0,\varepsilon^-_1}_\Sigma \colon H_\bullet(\Sigma;R[\pi_1(\Sigma)]) \to LCH^{n-\bullet}_{\varepsilon^+_0,\varepsilon^+_1}(\Lambda^+;R[\pi_1(\Sigma)]).\]

Observe that this map is $R[\pi_1(\Sigma)]$-linear from the left, and hence can be interpreted as being $\pi_1(\Sigma)$-equivariant in the appropriate sense. Also, we recall that $H_\bullet(\Sigma;R[\pi_1(\Sigma)])$ is isomorphic to $H_\bullet(\widetilde{\Sigma},R)$ of the universal cover $\widetilde{\Sigma} \to \overline{\Sigma}$  (observe that $\Sigma$ and $\overline{\Sigma}$ are homotopy equivalent).

Later we will be particularly interested in the case when $\Lambda^+$ is simply connected and when $\varepsilon^+_i$, $i=0,1$ both are induced from augmentations in the ground ring $R$. In this situation the universal coefficients theorem gives us an identification
\begin{equation}
\label{eq:free}
LCH^\bullet_{\varepsilon^+_0,\varepsilon^+_1}(\Lambda^+;R[\pi_1(\Sigma)])=LCH^\bullet_{\varepsilon^+_0,\varepsilon^+_1}(\Lambda^+;R) \otimes_R R[\pi_1(\Sigma)].
\end{equation}

Choosing a generator $m \in H_0(\Sigma;R[\pi_1(\Sigma)])$, the fundamental class induced by $\Sigma$ is defined to be the image
\[ \widetilde{c}^{\varepsilon^-_0,\varepsilon^-_1}_{\Sigma,m}:=\widetilde{G}^{\varepsilon^-_0,\varepsilon^-_1}_\Sigma(m) \in  LCH^n_{\varepsilon^+_0,\varepsilon^+_1}(\Lambda^+;R[\pi_1(\Sigma)]).\]

Let $\Lambda^+_{+ \epsilon}$ be a $C^1$-small perturbations of positive positive push-off of $\Lambda^+$ by the Reeb flow as defined in Section \ref{sec:computingpushoff}.
If we identify a standard contact neighbourhood of $\Lambda^+$ with the jet space $J^1 \Lambda^+$, we can identify  $\Lambda^+_{+ \epsilon}$ with the jet
$j^1g^+ \subset J^1 \Lambda^+$ of a Morse function $g^+ \colon \Lambda^+ \to \R$. We assume that $g^+$ has  a unique local minimum $m^+ \in \Lambda^+$. We will denote by $m^+$ also the Reeb chord from $\Lambda^+$ to $\Lambda^+_{+ \epsilon}$ induced by it.

There is a twisted ``banana'' chain map
\[ \widetilde{b} \colon LCC^{\epsilon_1^+, \epsilon_0^+}_\bullet(\Lambda^+_{+ \epsilon}, \Lambda^+; R[\pi_1(\Sigma)]) \to  LCC_{\epsilon_0^+, \epsilon_1^+}^{n-1-\bullet}(\Lambda^+, \Lambda^+_{+ \epsilon}; R[\pi_1(\Sigma)]) \]
defined as $\widetilde{b}(\gamma_{0,1}) = \sum \limits _u  c^{\varepsilon^+_0,\varepsilon^+_1}_u\gamma_{1,0}$, where
$$u\in \widetilde{\mathcal{M}}_{\R \times \Lambda^+, \R \times \Lambda^+_{+ \epsilon}} (\gamma_{1,0}, \boldsymbol{\delta}, \gamma_{0,1}, \boldsymbol{\zeta}).$$

The twisted coefficient version of the fundamental class of \cite{Duality_EkholmetAl} is the class
\begin{equation}
\widetilde{c}^{\varepsilon^+_0,\varepsilon^+_1}_{\Lambda^+,m^+} \in LCH_{\varepsilon_0^+,
\varepsilon_1^+}^n(\Lambda^+; R[\pi(\Sigma)])
\end{equation}
defined as the image of $\widetilde{b}(m^+)$ by the identification
\begin{equation}
\label{eq:pushid} LCH^\bullet_{\varepsilon^+_0,\varepsilon^+_1}(\Lambda^+,\Lambda^+_{+ \epsilon}; R[\pi_1(\Sigma)])  \cong LCH^\bullet_{\varepsilon^+_0,\varepsilon^+_1}(\Lambda^+;R[\pi_1(\Sigma)]),
\end{equation}
which holds by \cite[Proposition 2.7]{Floer_Conc}. The following proposition shows that the two definitions of the fundamental class given above in fact coincide. Its proof is postponed until Section \ref{sec:proof-prop-refprp:f}.
\begin{Prop}
\label{prp:fundclasstwisted}Assume that the natural map
$$H_0(\Lambda^+;R[\pi_1(\Sigma)]) \to H_0(\Sigma;R[\pi_1(\Sigma)])$$
sends $m^+ \in H_0(\Lambda^+,R[\pi_1(\Sigma)])$ to $m \in H_0(\Sigma;R[\pi_1(\Sigma)])$. For appropriate choices of almost complex structures and Hamiltonian perturbations in the constructions, there is an identification
\[\widetilde{c}^{\varepsilon^+_0,\varepsilon^+_1}_{\Lambda^+,m^+} = \widetilde{c}^{\varepsilon^-_0,\varepsilon^-_1}_{\Sigma,m} \in LCH^n_{\varepsilon^+_0,\varepsilon^+_1}(\Lambda^+;R[\pi_1(\Sigma)])\]
of fundamental classes.
\end{Prop}

In Section \ref{sec:proofpi_1carclass} we will see that Theorem \ref{thm:pi_1carclass} is a direct consequence of the above proposition.

\subsubsection{Slightly wrapping the end}
\label{sec:carefulpushoff}

The push-off $\Sigma_{\epsilon h}$ of $\Sigma$ as used in the proof of Theorem \ref{thm:lesduality} (c.f.~Section \ref{sec:seidel}) will here be constructed with additional care. Since $\Sigma_{\epsilon h}$ is $C^1$-close to $\Sigma$, it can be seen as the graph of $dF$ inside a Weinstein neighbourhood of  $\Sigma$ for a $C^2$-small Morse function $F \colon \Sigma \to \R$ with appropriate behaviour outside of a compact subset. The goal in this subsection is to choose this Morse function to be of a particular form, so that the technique from the end of the proof of Proposition \ref{prp:wrap} can be used to compute the relevant part of the  component $d_{+0}$ of the Cthulhu differential. Recall that the latter computation concerns the differential of a pair of trivial cylinders, where one of them has been wrapped by applying the Reeb flow. This techniques will be applicable here as well since we will achieve making the (relevant part of the) push-off, which is the graph of $dF,$ to be given by a small `wrapping' of the end.

\begin{Rem} Recall that a similar argument was used in Section \ref{sec: proof of exact sequences}. The difference here is that the wrapping takes place near the positive end instead of near the negative end, and that we want some additional control of the Morse function $F$ describing the push-off.
\end{Rem}

Here we suppose that $\Sigma$ is cylindrical outside of the subset $[-T,T] \times P \times \R$.
We will take a Hamiltonian $h(t, p, z) = \widetilde{h}(t) + e^t g(t, p)$ as in Section \ref{sec:small-pert-lagr} such that $\widetilde{h}$ is a function as shown in Figure \ref{fig:pushoff-lambda} and $g$ is a $C^2$-small function satisfying $\partial_tg=0$ for $|t| \ge T$. The complex $\Cth(\Sigma,\Sigma_{\epsilon h})$ is of the type considered in the proof of Theorem \ref{thm:lesduality}  because $\Lambda^-_{\epsilon h}= \Lambda^-_{- \epsilon}$ and $\Lambda^+_{\epsilon h}= \Lambda^+_{+ \epsilon}$ , and can thus be used in order to define the Seidel isomorphism.

\begin{figure}[htp]
\centering
 \labellist
 \pinlabel  $\widetilde{h}(t)$ at 90 108
 \pinlabel  $-T$ at 60 122
 \pinlabel  $T$ at 120 122
 \pinlabel  $T+a$ at 143 121
 \pinlabel $1$ at 45 60
 \pinlabel $-1$ at 45 5
 \pinlabel $-1$ at 45 112
 \pinlabel $t$ at 190 136
 \pinlabel $t$ at 190 35
 \pinlabel $e^{-t}\widetilde{h}'(t)$ at 160 70
 \pinlabel $-T$ at 60 22
 \pinlabel $T$ at 120 22
 \pinlabel $T+a$ at 158 21
 \endlabellist
\includegraphics{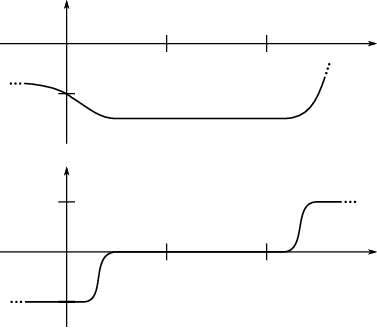}
\caption{The function $\widetilde{h}$.  The corresponding Hamiltonian vector field is $e^{-t}\widetilde{h}'(t)\partial_z$.}
\label{fig:pushoff-lambda}
\end{figure}

\begin{Lem}
\label{lem:morse}
For a suitable function $h$ as above and $\epsilon >0$ sufficiently small, $\Sigma_{\epsilon h}$
is given as a section $dF$ in a Weinstein neighbourhood of $\Sigma$, where $F \colon \Sigma \to \R$ is a Morse function. Furthermore, we may assume that
\begin{enumerate}
\item $\Lambda^+_{\epsilon h}$ is given as a section $j^1{g^+} \subset J^1\Lambda^+$ for a standard Legendrian neighbourhood of the positive end $\Lambda^+$ of $\Sigma$ with $g^+ \colon \Lambda^+ \to (0,\epsilon]$ a small Morse function with a unique minimum $m^+,$
\item $\partial_tF>0$ and $\partial_t F<0$ holds for all $t \gg 0$ and $t \ll 0,$ respectively, and
\item there is a unique local minimum $m$ of $F$ which appears near the slice $\{t = T+a\},$ and which corresponds to the unique minimum $m^+$ of the Morse function $g^+$.
\end{enumerate}
\end{Lem}

The Morse functions $F \colon \Sigma \to \R$ and $g^+ \colon \Lambda^+ \to \R$ both have unique local minima, denoted by $m$ and $m^+,$ respectively. Now we will see that it is also is possible to assume that there is a unique Floer strip connecting the minimum of $F$ with the chord at $+\infty$ corresponding to the minimum of $g^+$. In fact, this Floer strip can be identified with a gradient flow line from the critical point $m^+,$ seen as a saddle point at $+\infty,$ to the global minimum $m$ of $F$.

\begin{Lem}
\label{lem:m}
Let $\Sigma_{\epsilon h}$ be the Hamiltonian push-off of $\Sigma$ described in Lemma \ref{lem:morse}. There exists almost complex structures for which the canonical projection
$$[T, + \infty) \times P \times \R \to P$$
is pseudoholomorphic (in particular, $J$ is a cylindrical lift in the convex end $t \gg 0$) and there is a \emph{unique} and transversely cut out $J$-holomorphic disc having boundary on $\Sigma \cup \Sigma_{\epsilon h}$ and a single positive puncture asymptotic to $\gamma_{m^+}$. This disc is moreover a rigid strip having precisely two punctures, where the second puncture is mapped to the intersection point $m \in \Sigma \cap \Sigma_{\epsilon h}$ corresponding to the unique minimum of the Morse function $F.$
\end{Lem}
\begin{proof}
This follows by the same technique used in the end of the proof of Proposition \ref{prp:wrap}. Also, see the explicit calculation made in \cite[Lemma 8.2]{LiftingPseudoholomorphic} as well as in the proof of \cite[Theorem 2.15]{Floer_Conc}. The key observation is that the boundary of $\Sigma_{\epsilon h} \cap \{ t \le T+ a-\delta\}$ is a submanifold $\Lambda'$ which is Legendrian in $\{ T+ a-\delta \} \times P \times \R$ (for some small $\delta>0$), and therefore $\Sigma_{\epsilon h}$  can be obtained from $\Sigma_{\epsilon h} \cap \{ t \le T+ a- \delta \}$ by concatenation with a small ``wrapping'' of $\R \times \Lambda'$ using a suitable $t$-dependent positive rescaling of the Reeb flow. This is possible because $\partial_tg=0$ if $t \ge 0$.
\end{proof}

\subsubsection{The proof of Proposition \ref{prp:fundclasstwisted} and functoriality}
\label{sec:proof-prop-refprp:f}
\begin{proof}[Proof of Proposition \ref{prp:fundclasstwisted}]
  Let $l_\gamma$ denote the coefficient of the Reeb chord
  $\gamma$ in the fundamental class
  $c^{\varepsilon_0^+,\varepsilon_1^+}_{\Lambda^+,m^+}$. Recall that $l_\gamma$
is given by the count of rigid punctured holomorphic bananas in the moduli spaces of the form
  $\widetilde{\mathcal{M}}_{\R \times \Lambda^+, \R \times  \Lambda^+_{\epsilon h}}(\gamma; \boldsymbol{\delta}, m^+, \boldsymbol{\zeta})$,
  where each strip is counted with the weight
  $\varepsilon^+_0(\boldsymbol{\delta})\varepsilon^+_1(\boldsymbol{\zeta})$.

 Now we consider the moduli spaces ${\mathcal M}_{\Sigma, \Sigma_{\epsilon h}}(\gamma;\boldsymbol{\delta},m^+,\boldsymbol{\zeta})$ of punctured holomorphic bananas with boundary on $\Sigma \cup \Sigma_{\epsilon h}$ having precisely two positive punctures, one of which is asymptotic to the Reeb chord $m^+$ from $\Lambda^+$ to
  $\Lambda^+_{\epsilon h}$ corresponding to the minimum of the Morse function
  $g^+$, and the other one is asymptotic to the Reeb chord $\gamma$ from
  $\Lambda^+_{\epsilon h}$ to $\Lambda^+$. By combining the properties of the non-negativity of the Fredholm index for a generic path of admissible almost complex structures, together with the positivity of the energy, we can conclude the following: the compactification of the index one moduli spaces ${\mathcal M}_{\Sigma, \Sigma_{\epsilon h}}(\gamma; \boldsymbol{\delta}, m^+, \boldsymbol{\zeta}; J_\bullet)$ a priori consists of pseudoholomorphic buildings of the following different kinds when the path of almost complex structures is generically chosen:
  \begin{enumerate}
  \item Pseudoholomorphic buildings with:
    \begin{itemize}
    \item a level consisting of a punctured banana of index one with
      boundary on $\R \times \Lambda^+ \cup \R \times \Lambda^+_{\epsilon h}$
      (which hence is rigid up to translations), and
    \item a level consisting of punctured discs of index
      zero having boundary on either $\Sigma$ or    $\Sigma_{\epsilon h}$.
    \end{itemize}
  \item Pseudoholomorphic buildings with:
    \begin{itemize}
    \item a level consisting of a punctured strip of index one
      having boundary on $\R \times \Lambda^+ \cup \R \times
      \Lambda^+_{\epsilon h}$ (which hence is rigid up to translation) together
      with a trivial strip over a Reeb chord, and
    \item a level consisting of a single punctured banana of
      index zero having boundary on $\Sigma \cup \Sigma_{\epsilon h}$.
    \end{itemize}
  \item Pseudoholomorphic buildings with:
    \begin{itemize}
    \item a level consisting of two punctured strips of index
      zero having boundary on $\Sigma \cup \Sigma_{\epsilon h}$, and
    \item a level consisting of a punctured banana with
      boundary on $\R \times \Lambda^- \cup \R \times \Lambda^-_{\epsilon h}$
      which is of index one (and hence rigid up to translation).
    \end{itemize}
  \item Pseudoholomorphic buildings with:
    \begin{itemize}
    \item a level consisting of a punctured banana of
      index zero having boundary on $\Sigma \cup \Sigma_{\epsilon h}$, and
    \item a  level consisting of a  punctured disc
      of index one with boundary on either $\R \times \Lambda^-$ or $\R \times \Lambda^-_{\epsilon h}$
      (which hence are rigid up to translation), together with
      additional trivial strips over Reeb chords.
    \end{itemize}
  \item A broken punctured strip having boundary on $\Sigma \cup
    \Sigma_{\epsilon h}$ with matching ends asymptotic to an intersection point between $\Sigma$ and $\Sigma_{\epsilon h}$.
  \end{enumerate}
  See Figure \ref{fig:fundclassbreaking} for a schematic picture of
  the above pseudoholomorphic buildings.

A gluing argument implies that the configurations in (1) are in
  bijection with the configurations contributing to the coefficient of $\gamma$ in the fundamental class $\widetilde{c}_{\Lambda^+, m^+}^{\varepsilon_0^+, \varepsilon_0^+}$. Similarly, the count of the configurations in (5) gives the coefficient of $\gamma$ in
  $\widetilde{G}^{\varepsilon^-_0,\varepsilon^-_1}_\Sigma(m)$ by Lemma
  \ref{lem:m}. We proceed to infer that the count of all
  buildings of type (2)-(4) is equal to the coefficient of $\gamma$
  in the expression $d_{\varepsilon^+_0,\varepsilon^+_1}\circ
  {\widetilde{b}}^{\Sigma_1,\Sigma}(m^+)$, from which the sought equality on the
  level of homology now follows. (As usual, all counts above are
  weighted by the augmentations $\varepsilon^-_i$, $i=0,1$.)

  (2): There are two cases: either the non-trivial strip in the top
  level has a positive puncture asymptotic to $m^+$, or it has positive puncture asymptotic to $\gamma$. The former case can
  be excluded by actions reasons, while the count of the latter
  configurations corresponds exactly to the coefficient in front of
  $\gamma$ of the boundary $d_{\varepsilon^+_0,\varepsilon^+_1}\circ
  {\widetilde{b}}^{\Sigma_1,\Sigma}(m_+)$.

  (3): There are no buildings of this type. Namely, by Lemma \ref{lem:m},
  we may assume there are no punctured pseudoholomorphic strips with boundary on
  $\Sigma \cup \Sigma_{\epsilon h}$ having positive asymptotic to the minimum $m^+$ and a negative asymptotic to a Reeb chord from
  $\Lambda^-_0$ to $\Lambda^-_{\epsilon h}$.

  (4): The sum of these contributions vanishes, as follows from the
  fact that $\varepsilon^-_i$, $i=0,1$, vanishes on any boundary of
  the Chekanov-Eliashberg algebras of $\Lambda^-$ and $\Lambda^-_{\epsilon h}$ (see \ref{II} in Section \ref{sec:Cthulhu-complex}). 

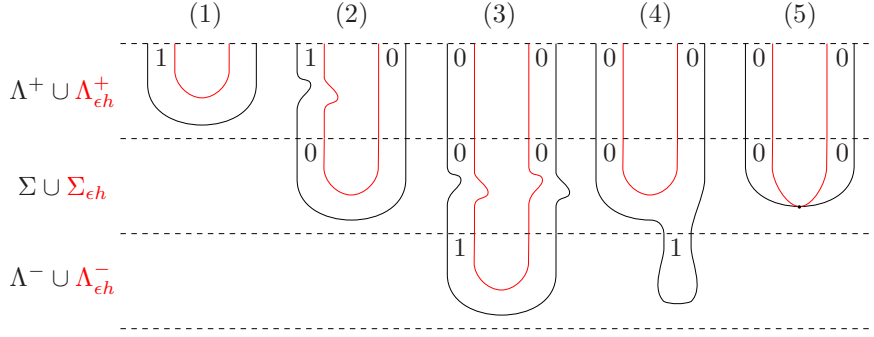
\begin{figure}[ht!]
  \centering \vspace{0.5cm}
  \labellist
  \pinlabel (1) at 50 185 \pinlabel (2) at 136 185 \pinlabel (3) at
  223 185 \pinlabel (4) at 315 185 \pinlabel (5) at 400 185 \pinlabel
  $1$ at 24 160 \pinlabel $1$ at 112 160 \pinlabel $0$ at 112 104
  \pinlabel $0$ at 160 160 \pinlabel $0$ at 200 160 \pinlabel $0$ at
  248 160 \pinlabel $0$ at 200 104 \pinlabel $0$ at 248 104 \pinlabel
  $1$ at 200 48 \pinlabel $0$ at 288 160 \pinlabel $0$ at 288 104
  \pinlabel $0$ at 337 160 \pinlabel $1$ at 327 48 \pinlabel $0$ at
  376 160 \pinlabel $0$ at 376 104 \pinlabel $0$ at 425 160 \pinlabel
  $0$ at 425 104 \pinlabel $\Lambda^+\cup\textcolor{red}{\Lambda^+_{\epsilon h}}$
  at -34 140 \pinlabel $\Sigma\cup\textcolor{red}{\Sigma_{\epsilon h}}$ at -34 83
  \pinlabel $\Lambda^-\cup\textcolor{red}{\Lambda^-_{\epsilon h}}$ at -34 30
  \endlabellist
  \includegraphics[scale=0.65]{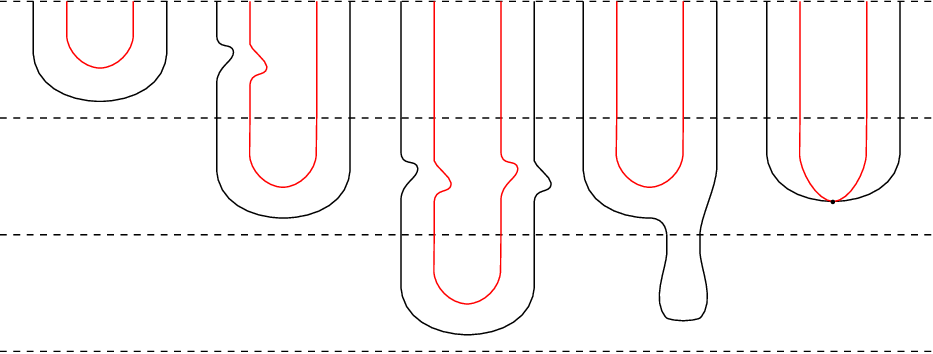}
  \caption{The pseudoholomorphic buildings (1)-(5) described in the
    proof of Proposition \ref{prp:fundclasstwisted}. The number on
    each component denotes its Fredholm index.}
  \label{fig:fundclassbreaking}
\end{figure}
\end{proof}

The argument above allows us to prove the fact (already pointed out in \cite{Ekhoka}) that the fundamental class is functorial with respect to exact Lagrangian cobordisms. Indeed, stretching the neck in the slice $\{t=-T\}$ decomposes the map $G^{\varepsilon^-_0,\varepsilon^-_1}_\Sigma$ into  $\Phi_\Sigma^{\varepsilon_0^-,\varepsilon_1^-}\circ G^{\varepsilon^-_0,\varepsilon^-_1}_{\mathbb{R}\times\Lambda^-}$, where $\Phi_\Sigma$ is the DGA morphism induced by the cobordism. Alternatively, one can also use the long exact sequence produced by Theorem \ref{thm:lesmayer-vietoris} together with Proposition \ref{prp:fundclasstwisted} in order to deduce this. In either case, we have:
\begin{Thm}\cite[Theorem 7.7]{Ekhoka}
\label{prp:fundclassfunct}
Let $\Sigma$ be a connected exact Lagrangian cobordism from $\Lambda^-$ to $\Lambda^+$, and let $\varepsilon^-_i$, $i=0,1$, be augmentations of the Chekanov-Eliashberg algebra of $\Lambda^-$ which pull back to augmentations $\varepsilon^+_i$ under the DGA morphism $\Phi_\Sigma$ induced by $\Sigma$. It follows that
\[\Phi^{\varepsilon^-_0,\varepsilon^-_1}_\Sigma(c^{\varepsilon^-_0,\varepsilon^-_1}_{\Lambda^-,m^-})=c^{\varepsilon^+_0,\varepsilon^+_1}_{\Lambda^+,m^+},\]
i.e.~the fundamental class is preserved under the bilinearised dual of the DGA morphism induced by $\Sigma$, under the additional assumption that the images of $m^\pm$ under the natural maps $H_0(\Lambda^\pm,\F) \to H_0(\Sigma)$ agree.
\end{Thm}

\begin{Rem} Recall that the fundamental class $c^{\varepsilon_0,\varepsilon_1}_{\Lambda,m}$ is guaranteed to be nonvanishing in homology only when $\varepsilon_0=\varepsilon_1$; we again refer to \cite[Theorem 5.5]{Duality_EkholmetAl}.
\end{Rem}

\subsection{A brief introduction to homology with $L^2$-coefficients.}
\label{sec:l2}
We use the technology of $L^2$-Betti numbers, introduced by Atiyah in \cite{AtiyahL2}, as a tool to study rank properties of Legendrian contact cohomology when the coefficient ring is a (countable) group ring
$\mathbb{C}[\pi]$ for a group $\pi$. This technique  will be used later
in the proof of Theorem \ref{thm:l2rigidity}, see Section \ref{sec:proofl2rigidity}.
The main idea is to replace
$\mathbb{C}[\pi]$, which is not a priori a Noetherian ring, with a more manageable module. Namely, we consider its $L^2$-completion $\ell^2(\pi)$ defined by the set of functions $f\co
\pi\rightarrow\mathbb{C}$ satisfying $\sum_{g\in\pi}|f(g)|^2<\infty$, endowed with its natural structure of a Hilbert space.

We do not intend to give a comprehensive introduction to the subject of $L^2$-homology and $L^2$-dimension,  but we still try to give an understandable overview of the techniques we use. For more details we refer the reader to the book of L{\"u}ck~\cite{L2book} and the introductory paper of Eckmann~\cite{L2intro} as the main references for the results used.

A  \textit{Hilbert $\pi$-module} $V$ is a Hilbert space over $\C$ on which $\pi$ acts by isometries. It is said to be {\em finitely generated} if it admits a $\pi$-equivariant isometric embedding $i \colon V \to \ell^2(\pi)\otimes_{\mathbb{C}}\mathbb{C}^m$ for a some $m\in \mathbb{N}$.
When this is the case, there is also a $\pi$-equivariant orthogonal projection $p \colon \ell^2(\pi)\otimes_{\mathbb{C}}\mathbb{C}^m \to V$ such that $p \circ i = \id_V$.

Morphisms of Hilbert $\pi$-modules are $\pi$-equivariant bounded linear
maps. Given an endomorphism $f\co V\rightarrow V$ of a finitely generated Hilbert module, we define its
\textit{von Neumann trace} by
\begin{equation}
\tr_{\ell^2}(f):=\sum_{i=1}^m\langle \overline{f}(1\otimes e_i),1\otimes e_i \rangle, \label{eq:34}
\end{equation}
where  $\overline{f}:=i\circ f\circ p$ and
$\{e_i\}$ is the standard basis of $\mathbb{C}^m$. A simple computation shows that this trace only depends on $f$ and not on the particular choice of the embedding.

The \textit{von Neumann dimension} of $V$ is ${\dim_{\ell^2}}(V)=\tr_{\ell^2}{\id_V}$, which is a non-negative number bounded from above by $m$, under the assumption that $V$ can be embedded in
$\ell^2(\pi)\otimes_{\mathbb{C}}\mathbb{C}^m$. Note that the von Neumann dimension can take non-integer values.

A sequence of morphisms of Hilbert $\pi$-modules
$$U\xrightarrow{f} V\xrightarrow{g} W$$
is {\em weakly exact} (at $V$) if $\overline{\im f}=\ker g$.
The following basic properties will be crucial:
\begin{Lem}[Theorem 1.12 in \cite{L2book}]
\label{lem:l2basics}
\begin{enumerate}
\item $V=0$ holds if and only if ${\dim_{\ell^2}}(V)=0$;
\item If
$$0\rightarrow U\xrightarrow{i} V\xrightarrow{p} W\rightarrow 0$$
is a short weakly exact sequence, then
${\dim_{\ell^2}}(V)={\dim_{\ell^2}}(U)+{\dim_{\ell^2}}(W)$.
\end{enumerate}
\end{Lem}
If $(C_\bullet, \partial)$ is a complex of finitely generated free $\C[\pi]$-modules, we can consider its $\ell^2$-completion $(\overline{C}_\bullet, \overline{\partial})$; where $\overline{C}_\bullet= C_\bullet \otimes_{\C[\pi]} \ell^2(G)$ and $\overline{\partial} = \partial
\otimes_{\C[\pi]} \id_{\ell^2(\pi)}$. Then $(\overline{C}_\bullet, \overline{\partial})$ is a complex
of finitely generated Hilbert $\pi$-modules. Its {\em $L^2$-homology}, denoted by $H_\bullet^{(2)}(C_\bullet, \partial)$, is defined as
$$H_\bullet^{(2)}(C_\bullet, \partial) = \ker \overline{\partial} \left / \overline{\OP{im} \overline{\partial}} \right. ;$$
i.e. as the quotient of the subspace of cycles by the {\em closure} of the subspace of boundaries.  It follows from the definition that $H^{(2)}_\bullet(C_\bullet,\partial)$ is also a finitely generated Hilbert $\pi$-module. \color{black}

\begin{Lem}
\label{lem:l2ranks}
If $(C_\bullet, \partial)$ is a complex of finitely generated free $\C[\pi]$-modules, then
\[{\dim_{\ell^2}} H^{(2)}_i (C_\bullet,\partial) \le {\dim_{\C}
C_i \otimes_{\C[\pi]} \C.} \]
Furthermore, for a finite-dimensional complex $(C'_\bullet,\partial')$ over $\C$, we have
\[H^{(2)}_i(C'_\bullet \otimes \C[\pi] ,\partial' \otimes \id_{\C[\pi]})=H_i (C'_\bullet,\partial') \otimes_\C \ell^2(\pi),\]
and thus, in particular,
\[{\dim_{\ell^2}} H^{(2)}_i (C'_\bullet \otimes \C[\pi] ,\partial' \otimes \id_{\C[\pi]}) = \dim_\C H_i (C'_\bullet,\partial'). \]
\end{Lem}
\begin{proof}
The first statement follows from Lemma \ref{lem:l2basics} together with the Hodge decomposition in \cite[Lemma 1.18]{L2book}. The second statement follows by a direct computation.
\end{proof}

For a pair of finite CW complexes $(X,Y)$, $Y \subset X$ and a group homomorphism $\varphi \co \pi_1(X) \to \pi$, there is an induced covering $(\widetilde{X},\widetilde{Y}) \to (X,Y)$  with a natural free $\pi$-action. (We do not assume that the total space of a covering is connected.) When $(C_\bullet,\partial)$ is the $\pi$-equivariant cellular complex associated to such a covering, we will write the corresponding $L^2$-homology groups by $H^{(2)}_\bullet(X,Y;\varphi)$ or, by abuse of notation, $H^{(2)}_\bullet(X,Y;\pi)$.

$L^2$-homology for {\em finite}\footnote{The finiteness condition can be relaxed, but not completely removed.} CW-complexes satisfies a version of the exact sequence for the pair and of Mayer-Vietoris sequence.
\begin{Prop} \label{prop: L^2 exact sequences}
Let $X$ be a finite CW-complex and $\varphi \co \pi_1(X) \to \pi$ a group homomorphism.
\begin{enumerate}
\item If $Y \subset X$ is a sub-CW complex and $* \in Y,$ then there is a long weakly exact sequence
$$\to H^{(2)}_{i+1}(X, Y; \pi) \to H^{(2)}_i(Y; \pi) \to H^{(2)}_i(X; \pi) \to H^{(2)}_i(X, Y; \pi)
\to.$$
\item If $X_0, X_1\subset X$ are sub-CW-complexes and $* \in X_0 \cap X_1$, then there is a long weakly exact sequence
\begin{align*}
& \to H^{(2)}_{i+1}(X; \pi) \to \\
& \to H^{(2)}_i(X_0 \cap X_1; \pi) \to H_i^{(2)}(X_0; \pi) \oplus H_i^{(2)}(X_i; \pi) \to H^{(2)}_i(X; \pi) \to.
\end{align*}
\end{enumerate}
\end{Prop}
\begin{proof}[Sketch of proof]
Let $\widetilde{X} \to X$ be the cover induced by $\varphi$ and let $\widetilde{Y}, \widetilde{X}_0, \widetilde{X}_1$ be the preimages of $Y, X_0, X_1$ respectively.
There are $\pi$-equivariant short exact sequences of cellular homology complexes
$$0 \to C_\bullet(\widetilde{Y}) \to C_\bullet(\widetilde{X}) \to C_\bullet(\widetilde{X},
\widetilde{Y}) \to 0$$
and
$$0 \to C_\bullet(\widetilde{X}_0 \cap \widetilde{X}_1) \to C_\bullet(\widetilde{X}_0) \oplus C_\bullet(\widetilde{X}_1) \to C_\bullet(\widetilde{X}) \to 0.$$
If we complete the $\C[\pi]$-modules appearing in the above exact sequences, we obtain
short exact sequences of complexes of finitely generated Hilbert $\pi$-modules. Since complexes of finitely generated Hilbert $\pi$-modules are automatically Fredholm in the sense of \cite[Definition 1.20]{L2book}, we can apply the snake lemma for $L^2$-homology due to Cheeger and Gromov \cite{CheeGro} (see also \cite[Theorem~1.21]{L2book}), from which the long weakly exact sequences follow.
\end{proof}

Let $\widetilde{\Sigma}_0 \to \Sigma_0$ be the covering associated to a group homomorphism $\varphi \co \pi_1(\Sigma_0) \to \pi$ and let $\varepsilon_i^-$, $i=0,1$ be augmentations of the negative ends $\Lambda_i^-$ with values in $\C[\pi]$. Then we can complete the bilinearised Legendrian contact homology complexes $LCC_{\varepsilon_0^\pm, \varepsilon_1^\pm}^\bullet(\Lambda_0^\pm, \Lambda_1^\pm; \pi)$ and the Cthulhu complex $\Cth(\Sigma_0, \Sigma_1; \pi)$ of Section~\ref{sec: twisted cthulhu}. This construction yields
$L^2$ bilinearised Legendrian contact homology groups $LCH^{(2) \bullet}_{\varepsilon_0^\pm, \varepsilon_1^\pm}(\Lambda_0^\pm, \Lambda_1^\pm; \pi)$ and $LCH^{(2) \bullet}_{\varepsilon_0, \varepsilon_1}(\Lambda; \pi)$.

The proof of acyclicity goes through for the completed Cthulhu complex and, since it is a complex of finitely generated Hilbert $\pi$-modules, when either $d_{0-}=0$ or $d_{-0}=0$ (e.g.\ if Lemma \ref{lem: when things work} holds) the $L^2$  Snake Lemma of Cheeger and Gromov yields an $L^2$ version of the exact sequences \eqref{eq:17} and \eqref{eq:18}. Finally, the same arguments of Section
\ref{sec:computingpushoff} go through in the $L^2$ setting, and therefore we have the following proposition.

\begin{Prop}\label{L^2 Cthulhu sequence}
Let $\Sigma$ be a relative Pin exact Lagrangian cobordism with ends $\Lambda^+$ and $\Lambda^-$. If $\varphi \colon \pi_1(\Sigma) \to \pi$ is a group homomorphism, $\varepsilon_0^-$, $\varepsilon_1^-$ are augmentations of $\Lambda_-$ with values in $\C[\pi]$ and $\C$ respectively, and $\varepsilon_0^+ := \widetilde{\Phi}_\Sigma \circ \varepsilon_0^-$, $\varepsilon_1^+ := \Phi_{\Sigma} \circ \varepsilon_1^-$ are their pull-back, then there is a weakly exact sequence
\[
  \xymatrix{
    LCH^{(2) \bullet}_{\varepsilon_0^-, \varepsilon_1^-}(\Lambda^-;\pi) \ar[rr] & &  LCH^{(2) \bullet}_{\varepsilon_0^+, \varepsilon_1^+}(\Lambda^+;\pi)\ar[dl] \\
    & H^{(2)}_\bullet(\overline{\Sigma}, \partial_-\overline{\Sigma};\pi) \ar[ul]
    & }
  \]
\end{Prop}
\color{black}

\section{Applications and examples}
\label{sec:examples}
In this section we deduce all applications mentioned in the introduction of the paper. In addition, we provide explicit examples of Lagrangian cobordisms: both examples to which our results apply, but also examples showing the importance of the different hypotheses used.
We will use $H(X)$ to denote the total homology of $X$, and a similar convention will be used for the Legendrian contact homology.
\subsection{The homology of an endocobordism}
\label{sec:topol-endoc}

The following proofs of Theorems~\ref{homrigidityold} and
\ref{thm:homologycylinder} are similar to the proofs given in
\cite{Rigidityofendo}.

\begin{proof}[Proof of Theorem \ref{homrigidityold}]
We begin by showing the result in the case when $\F=\mathbb{Z}_2$.
(i): First, recall the elementary fact from algebraic topology that
\begin{align}\label{ineqgoeintheformnat}
\dim_{\F} H(\Sigma;\mathbb F) \geq \dim_{\F} H(\Lambda;\mathbb F)
\end{align}
is satisfied, which follows by studying the long exact sequence of the pair $(\overline{\Sigma},\partial\overline{\Sigma})$ together with Poincar\'{e} duality (see \cite[Lemma 2.1]{Rigidityofendo}).

We proceed to prove the opposite inequality
$$\dim_{\F} H(\Sigma;\mathbb F)\leq\dim_{\F}(\Lambda;\mathbb F).$$
Linearised Legendrian contact cohomology satisfies the bound
\[ \dim_\F LCH_{\varepsilon'}(\Lambda) \le |\mathcal{R}(\Lambda)|\]
for any $\varepsilon'$. Thus we can fix an augmentation $\varepsilon$ of $\mathcal{A}(\Lambda;\F)$ satisfying
\begin{align}\label{maxdeffff}
\dim_{\F} LCH_{\varepsilon}(\Lambda;\mathbb F) = \max_{\varepsilon'}\{\dim_{\F} LCH_{\varepsilon'}(\Lambda;\mathbb F)\}.
\end{align}
The exact triangle in Theorem~\ref{thm:lesmayer-vietoris} gives us
\begin{eqnarray*}
\lefteqn{\dim_{\F} LCH_{\varepsilon_+}(\Lambda;\mathbb F) \ge} \\
& \ge & \dim_{\F} LCH_{\varepsilon}(\Lambda;\mathbb F)+\dim_\F H(\Sigma;\mathbb{F})-\dim_\F H(\Lambda;\mathbb{F})
\end{eqnarray*}
where $\varepsilon_+$ is the augmentation of $\mathcal{A}(\Lambda;\F)$ obtained as the pull-back
$\varepsilon_+ :=\varepsilon\circ \Phi_\Sigma$. Formula \eqref{maxdeffff} implies that $\dim_\F H(\Sigma;\mathbb{F})-\dim_\F H(\Lambda;\mathbb{F})\leq 0$. Together with inequality \eqref{ineqgoeintheformnat}, we obtain
\begin{align}\label{equalityofdim}
\dim_{\F} H(\Sigma;\mathbb F)=\dim_{\F} H(\Lambda;\mathbb F).
\end{align}

In order to show that $\dim_{\F} H_i(\Sigma;\mathbb F)=\dim_{\F} H_i(\Lambda;\mathbb F)$ for all $i$, we argue by contradiction, assuming that
\[d_{i_0}(\Sigma):=\dim_{\F} H_{i_0}(\Sigma;\F)-\dim_{\F} H_{i_0}(\Lambda;\F)>0\]
for some $i_0$. By the Mayer-Vietoris sequence we conclude that the inequality
\[\dim_\F H_{i_0}(\Sigma \odot \Sigma;\F) \ge 2\dim_\F H_{i_0}(\Sigma;\mathbb F)-\dim_\F H_{i_0}(\Lambda;\F)\]
holds. In particular,
\[d_{i_0}(\Sigma \odot \Sigma) := \dim_\F H_{i_0}(\Sigma \odot \Sigma;\F) -\dim_{\F} H_{i_0}(\Lambda;\F) \ge 2d_{i_0}(\Sigma),\]
which by induction then leads to a contradiction with equality \eqref{equalityofdim}.  Indeed, after the $k$-th iteration of this argument, we obtain the inequality
$$ \dim_\F H_{i_0}(\Sigma^{\odot 2k};\F) \ge 2^kd_{i_0}(\Sigma),$$
where the right hand side is positive by assumption.

(ii): The argument is the same as the one in the proof of \cite[Theorem 1.6 (ii)]{Rigidityofendo}, and follows form Part (i) applied to the concatenation $\Sigma \odot \Sigma$. Namely the Mayer-Vietoris sequence for the concatenation $\Sigma \odot \Sigma$ seen as two copies of $\overline{\Sigma}$ glued along the boundary component $\Lambda$ shows that
\[ \dim_\F H(\Sigma \odot \Sigma;\F) \ge 2 \dim_\F H(\Sigma;\F) - \dim_\F \im(i^-_*,i^+_*) \]
and by the above result, we conclude that
\[\dim_\F \im(i^-_*,i^+_*) =\dim_\F H(\Sigma;\F)=\dim_\F H(\Lambda;\F),\]
from which the claim follows.

(iii): By contradiction, we assume that $i^+_* \oplus i^-_* \colon H(\Lambda \sqcup \Lambda) \to H(\Sigma)$ is not a surjection. Considering a representative $V \subset H(\Sigma)$ of the cokernel of this map, which hence is of dimension $\dim_\F V > 0$, the Mayer-Vietoris long exact sequence implies that the image of $V \oplus V$ under the map
\[ H(\Sigma) \oplus H(\Sigma) \to H(\Sigma \odot \Sigma) \]
has image of dimension $2 \dim_\F V>0$. Moreover, $V \oplus V$ is not contained in the image of
\[i^+_* \oplus -i^-_* \colon H(\Lambda \sqcup \Lambda) \to H(\Sigma \odot \Sigma).\]
Namely, the above inclusion factorises through the canonical maps as
\[i^+_* \oplus -i^-_* \colon H(\Lambda \sqcup \Lambda) \to H(\Sigma \sqcup \Sigma) \to H(\Sigma \odot \Sigma),\]
where the latter morphism is the one from the above Mayer-Vietoris long exact sequence. In conclusion, the cokernel of
\[i^+_* \oplus i^-_* \colon H(\Lambda \sqcup \Lambda) \to H(\Sigma \odot \Sigma)\]
is of dimension at least $2 \dim_\F V$. Arguing by induction, now we arrive at a contradiction with Part (i).

The proof is now complete for $\Z_2$.
Under the additional assumption that $\Lambda$ is Pin, and admitting an augmentation in an arbitrary field $\F$, Corollary \ref{cor:spin-orient} of Theorem \ref{thm:w_ivanish} will imply that any endocobordism of $\Lambda$ is Pin as well. This allows us to repeat the previous argument with coefficients in the field $\F$. Note that Theorem \ref{thm:w_ivanish} relies on Theorem \ref{homrigidityold} in the particular case $\F=\Z_2$, which can be established without orienting the moduli spaces, and therefore is independent of any assumption on the Stiefel-Whitney classes of $\Sigma$.
\end{proof}

We now prove the following theorem, of which Theorem \ref{thm:homologycylinder} is an immediate corollary. Observe that it can be proved also by using Theorem \ref{homrigidityold}.
\begin{Thm}\label{homrigidity}
Let $\Lambda$ be an $n$-dimensional Legendrian homology sphere inside a contactisation, $\Sigma$ an exact Lagrangian cobordism from $\Lambda$ to itself inside the symplectisation, and $\F$ a field. If $\mathcal{A}(\Lambda;\F)$ admits an augmentation, then $H_\bullet(\Sigma,\Lambda;\F)=0$, i.e.~$\Sigma$ is a $\F$-homology cylinder.
\end{Thm}

\begin{proof}
Let $\Sigma^{\odot k}$, $k \ge 1$, be the $k$-fold concatenation of
$\Sigma$ with itself, which again is an exact Lagrangian cobordism
from $\Lambda$ to $\Lambda$. Since $\Lambda$ is a homology sphere it is Pin and, hence, $\Sigma^{\odot k}$ is Pin for all $k \ge 1$ by Corollary \ref{cor:spin-orient}.

We fix an augmentation $\varepsilon$ of
$\mathcal{A}(\Lambda;\F)$ and let $\varepsilon_k$ be the
augmentation of $\mathcal{A}(\Lambda;\F)$ obtained by the pull-back
of $\varepsilon$ under the unital DGA morphism induced by
$\Sigma^{\odot k}$.

The (ungraded version of the) long exact sequence in Theorem \ref{thm:lespair} becomes
\begin{equation}
\label{eq:13}
\xymatrix{
{LCH}_{\varepsilon}(\Lambda) \ar[rr] & &  {LCH}_{\varepsilon_k}(\Lambda)\ar[dl] \\
& {H}(\overline{\Sigma}^{\odot k},\partial_{-}\overline{\Sigma}^{\odot k};\F)  \ar[ul] & }
\end{equation}
Observe that
\[\dim H_i(\overline{\Sigma}^{\odot k},\partial_{-}\overline{\Sigma}^{\odot k};\F)=\begin{cases}
0, & i=0,n+1, \\
k\dim H_i(\overline{\Sigma},\partial_{-}\overline{\Sigma};\F), & 0<i<n+1,
\end{cases}
\]
as follows from the Mayer-Vietoris long exact sequence together with the assumption that $\Lambda$ is a $\F$-homology sphere.

Since the linearised contact cohomology satisfies the bound
\[ \dim_\F LCH_{\varepsilon'}(\Lambda) \le |\mathcal{R}(\Lambda)|\]
for any augmentation $\varepsilon'$, we get the inequality

\[ k\dim H_i(\overline{\Sigma},\partial_{-}\overline{\Sigma};\F) = \dim H_i(\overline{\Sigma}^{\odot k},\partial_{-}\overline{\Sigma}^{\odot k};\F) \le 2 |\mathcal{R}(\Lambda)|, \:\: 0<i<n+1,\]
for each $k$, where the exactness of the above triangle has been
used to show the last inequality. In conclusion, we have established
\[ \dim H_i(\overline{\Sigma},\partial_{-}\overline{\Sigma};\F) = 0, \:\: 0<i<n+1,\]
which finishes the proof.
\end{proof}

\begin{proof}[Proof of Theorem \ref{thm:homologycylinder}]
  Since $\Lambda$ is assumed to have an augmentation over $\mathbb{Z}$ it admits an augmentation over $\mathbb{Q}$ as well. And thus it follows from Theorem \ref{homrigidity} that $H_\bullet(\Sigma,\Lambda;\mathbb{Q})=0$ and thus that $H_\bullet(\Sigma,\Lambda;\mathbb{Z})$ is torsion. The augmentation over $\mathbb{Z}$ also induces an augmentation over any finite field, and thus Theorem \ref{homrigidity} implies that $H_\bullet(\Sigma,\Lambda;\mathbb{Z})$ has no $p$-torsion for any prime $p$. Thus $H_\bullet(\Sigma,\Lambda;\mathbb{Z})=0$.
\end{proof}

\begin{Rem}
Following the discussion in Section \ref{sec:augm-finite-dimens} we get that Theorem \ref{homrigidityold} holds under the weaker assumption that the Chekanov-Eliashberg algebra admits a non-commutative augmentation in a finite-dimensional $\F$-algebra. (The proof is a verbatim reproduction of the precedent.) In the same manner, Theorem \ref{thm:homologycylinder} thus holds under the weaker assumption that the Chekanov-Eliashberg algebra admits a non-commutative augmentation in a finite-rank $\Z$-algebra of characteristic zero.
\end{Rem}

\subsection{Characteristic classes of endocobordisms}
\label{sec:nonorientable}
\begin{proof}[Proof of Theorem~\ref{thm:w_ivanish}]
Recall from Section \ref{sec:remarks-about-grad} that Theorem \ref{homrigidityold} still applies when the cobordism $\Sigma$ is not orientable and has Maslov number one i.e. when the Cthulhu complexes involving $\Sigma$ necessarily are \emph{ungraded}. In this case, however, we obtain exact triangles instead of long exact sequences.

The dual statement of Part (iii) of Theorem \ref{homrigidityold} reads as follows: let $\Sigma$ be an exact Lagrangian endocobordism of $\Lambda$, the map $(i_+^*,i_-^*):H^*(\Sigma,\mathbb{Z}_2)\rightarrow H^*(\Lambda\sqcup \Lambda,\mathbb{Z}_2)$ is injective. Theorem \ref{thm:w_ivanish}  is then an immediate corollary of this and of the naturality of characteristic classes. Theorem \ref{thm:w_ivanish} for the Maslov class and the Pontryagin classes follows similarly, assuming that $\Lambda$ is relatively Pin.
\end{proof}
Theorem \ref{thm:w_ivanish} implies the following.
\begin{Cor}\label{cor:spin-orient}
If $\Lambda$ is orientable (respectively, Pin) and admits an augmentation into a finite-dimensional $\Z_2$-algebra, then any exact endocobordism $\Sigma$ of $\Lambda$ is orientable (respectively, Pin) as well.
\end{Cor}
This result can be seen as a generalisation of the result of Capovilla-Searle and Traynor, see \cite[Theorem 1.2]{NonorLagcobbetLegknots}.

\begin{Ex}
Recall that a Legendrian knot in the standard contact $\R^3$ for which the Kauffman bound on $\mathtt{tb}$ is not sharp does not admit an augmentation in a commutative ring \cite{KauffmanBound}.
\begin{enumerate}\label{rem:examplesthatanswernonsymqofcsandt}
\item
Consider the family of the Legendrian representatives of torus $(p,-q)$-knots $\Lambda_{(p,-q)}\subset \R^3$ with $q>p\geq 3$ and $p$ odd described by Sivek in \cite[Figure 3]{TheContHomofLegKNotswithMAXTBI}.
Following Sivek \cite{TheContHomofLegKNotswithMAXTBI}, we observe that $\mathtt{tb}(\Lambda_{(p,-q)})=-pq$ and, hence, from the classification result of Etnyre and Honda \cite{Etnyre_&_Knots_Contact1} it follows that $\Lambda_{(p,-q)}$ is $\mathtt{tb}$-maximising. Recall that Sivek \cite{TheContHomofLegKNotswithMAXTBI} proved that the Chekanov-Eliashberg algebra of $\Lambda_{(p,-q)}$ admits a $2$-dimensional representation over $\Z_2$, but for which the Kauffman bound on $\mathtt{tb}$ is not sharp. Therefore, these Legendrian knots do not admit non-orientable exact Lagrangian endocobordisms.
\item
Consider $\Lambda_{(p,-q)}\#\Lambda$, where $p$ is odd, $q>p\geq 3$, and let $\Lambda$ be a $\mathtt{tb}$-maximising Legendrian knot of $\R^3$ whose Chekanov-Eliashberg algebra admits an augmentation (or, more generally, $m$-dimensional linear representation) over $\Z_2$. Then, following the discussion in \cite[Lemma 4.3]{EstimNumbrReebChordLinReprCharAlg}, we see that the Kauffman bound for $\Lambda_{(p,-q)}\#\Lambda$ is not sharp and that the Chekanov-Eliashberg algebra of $\Lambda_{(p,-q)}\#\Lambda$ admits a finite-dimensional linear representation over $\Z_2$. In addition, from the fact that $\Lambda_{(p,-q)}$ and $\Lambda$ are $\mathtt{tb}$-maximising, together with \cite[Corollary 3.5]{Etnyre_&_Connected_Sums} (or \cite[Theorem 1.1]{AdditivityofTBofLegknots}), it follows that $\Lambda_{(p,-q)}\#\Lambda$ also is $\mathtt{tb}$-maximising. This leads us to many other examples, besides $\Lambda_{(p,-q)}$, which do not admit non-orientable exact Lagrangian endocobordisms.
\end{enumerate}
\end{Ex}

\begin{Rem}
The above examples provide a negative answer to a question of Capovilla-Searle and Traynor, see \cite[Question 6.1]{NonorLagcobbetLegknots}: it is not necessarily the case that a Legendrian knot admits a non-orientable endocobordism in the case when its Kauffman bound on $\mathtt{tb}$ is not sharp.
\end{Rem}

There is also an example due to Sivek, see \cite[Sections 2.2 and 3]{TheContHomofLegKNotswithMAXTBI}, of a $\mathtt{tb}$-maximising knot with non-sharp Kauffman bound on $\mathtt{tb}$, whose Chekanov-Eliashberg algebra does not admit a finite-dimensional linear representation over $\Z_2,$ but which does admit a representation in a infinite-dimensional algebra. Unfortunately, in this case our methods do not provide an obstruction to the existence of a topologically non-trivial endocobordism.

\subsection{Restrictions on the fundamental group of an endocobordism between simply connected Legendrians}
\label{sec:fond-group-some}

We now prove the results concerning the fundamental groups of endocobordisms between simply connected Legendrian submanifolds.

\subsubsection{Proof of Theorem \ref{thm:pi_1carclass}}
\label{sec:proofpi_1carclass}
\begin{proof}[Proof of Theorem \ref{thm:pi_1carclass}]
  Recall the construction of the fundamental class in the setting of
  twisted coefficients carried out in Section
  \ref{sec:fund-class-twist}. The proof will be a straightforward
  consequence of Proposition \ref{prp:fundclasstwisted} therein.

 From the assumptions of the theorem, the Legendrian submanifold
  $\Lambda^+$ has a unique augmentation $\varepsilon^+$. It follows that \cite[Theorem
  5.5]{Duality_EkholmetAl} can be applied, and hence the fundamental
  class
  $\widetilde{c}^{\varepsilon^+,\varepsilon^+}_{\Lambda^+,m^+}$ is
  non-vanishing. By Proposition \ref{prp:fundclasstwisted} we,
  moreover, conclude that this fundamental class is the image of a
  generator $m$ of $H_0(\Sigma;R[\pi_1(\Sigma)])$ under the map
  $\widetilde{G}^{\varepsilon^-,\varepsilon^-}_\Sigma.$ Since
  $\Lambda^+$ is simply connected by assumption, it follows from
  \eqref{eq:free} above that this image is not torsion. In particular,
$$g\cdot \widetilde{c}^{\varepsilon^+,\varepsilon^+}_{\Lambda^+,m^+}\not=\widetilde{c}^{\varepsilon^+,\varepsilon^+}_{\Lambda^+,m^+},\:\: \forall g\in\pi_1(\Sigma).$$
    Thus, $m$ is not torsion either, and since it generates
  $H_0(\Sigma;R[\pi_1(\Sigma)])$ we conclude that
    $H_0(\Sigma;R[\pi_1(\Sigma)])=R[\pi_1(\Sigma)]$.
    However, since $\widetilde{\Sigma}$ is connected, we know that
    $H_0(\Sigma;R[\pi_1(\Sigma)])=H_0(\widetilde{\Sigma})=R$. In
  other words, $\pi_1(\Sigma)$ is the trivial group, as sought.
\end{proof}
\subsubsection{Proof of Theorem \ref{thm:l2rigidity}}
\label{sec:proofl2rigidity}
\begin{proof}[Proof of Theorem \ref{thm:l2rigidity}]
  Here it will be crucial to use the machinery of $L^2$-coefficients
  as described in Section \ref{sec:l2}. We will denote $\pi=\pi_1(\Sigma)$.
Since $\Lambda$ is Pin by assumption, it follows from Corollary \ref{cor:spin-orient} that the $k$-fold concatenated cobordisms $\Sigma^{\odot k}$ are Pin for all $k \ge 1$.

Let $\overline{\Sigma}^{\odot k}$ be the quotient of $\sqcup_{i=1}^k \overline{\Sigma}_i$, $\overline{\Sigma}_i \cong \overline{\Sigma}$, which identifies $\partial_+(\overline{\Sigma}) \subset \overline{\Sigma}_i$ with $\partial_-(\overline{\Sigma}) \subset \overline{\Sigma}_{i+1}$. We will write $\partial \overline{\Sigma}^{\odot k}=\partial_-\overline{\Sigma}^{\odot k} \sqcup \partial_+\overline{\Sigma}^{\odot k}$, where $\partial_-\overline{\Sigma}^{\odot k} = \partial_- \overline{\Sigma}_1$, $\partial_+\overline{\Sigma}^{\odot k} = \partial_+ \overline{\Sigma}_k$.

Further, consider the covering space $\widetilde{\Sigma}^{\odot k} \to \overline{\Sigma}^{\odot k}$ obtained by gluing the boundary of the universal cover $\sqcup_{i=1}^k \widetilde{\Sigma}_i \to \sqcup_{i=1}^k \overline{\Sigma}_i$ in a $\pi$-equivariant way via the identification of the induced cover
\[ \widetilde{\Sigma}_i \supset \bigsqcup_{g \in \pi} \partial_+(\overline{\Sigma})\to \partial_+(\overline{\Sigma}) \subset \overline{\Sigma}_i\]
with the induced cover
\[  \widetilde{\Sigma}_{i+1} \supset \bigsqcup_{g \in \pi} \partial_-(\overline{\Sigma})\to \partial_-(\overline{\Sigma}) \subset \overline{\Sigma}_{i+1}.\]
Observe that the covering $\widetilde{\Sigma}^{\odot k} \to \overline{\Sigma}^{\odot k}$ obtained is induced by a group epimorphism
\[ \pi_1(\overline{\Sigma}^{\odot k}) {\cong} \underbrace{\pi * \hdots * \pi}_{k} \to \pi = \pi_1(\overline{\Sigma}) \]
given by multiplying all elements in a word.

  First,  we will consider the case
  $|\pi_1(\Sigma)|<\infty$. Under
  this assumption, the version of the long exact sequence in Theorem
  \ref{thm:lespair} applied to the system of local coefficients
  induced by the covering $\widetilde{\Sigma}^{\odot k} \to \overline{\Sigma}^{\odot k}$ (see Section \ref{sec: twisted cthulhu}) becomes
\begin{equation}
    \label{eq:13pi1}
    \xymatrix{
      LCH^\bullet_{\varepsilon, \varepsilon}(\Lambda;\C[\pi]) \ar[rr] & &  LCH^\bullet_{{\widetilde{\varepsilon}_k,} \varepsilon_k}(\Lambda;\C[\pi])\ar[dl] \\
      & H_\bullet(\widetilde{\Sigma}^{\odot k},\partial_-\widetilde{\Sigma}^{\odot k};\C).  \ar[ul] & }
  \end{equation}
  Here the augmentation $\varepsilon_k$ is the pull-back of the
  augmentation $\varepsilon$ under the unital DGA morphism induced by
  $\Sigma^{\odot k}$ and $\widetilde{\varepsilon}_k$ is the twisted
pull-back of $\varepsilon$.

  For $k \gg 0$ sufficiently large, unless $|\pi_1(\Sigma)|=1$ the
  equality
  \[ \dim_\C(H_1(\widetilde{\Sigma}^{\odot
    k},\partial_-{\widetilde{\Sigma}^{\odot k}};\C)) =
  (|\pi_1(\Sigma)|-1)k, \:\:k \ge 1,\] proved using the Mayer-Vietoris exact sequence,
  together with the universal bound
  \[\dim_\C
  LCH^\bullet_{\varepsilon_0,\varepsilon_1}(\Lambda;\C[\pi_1(\Sigma)])
  \le |\pi_1(\Sigma)||\mathcal{R}(\Lambda)|\] gives a contradiction.

  It remains to show that $|\pi_1(\Sigma)|$ is finite. Assuming the
  contrary, we apply the long weakly exact sequence from Proposition
\ref{L^2 Cthulhu sequence} to the cobordisms $\Sigma^{\odot k}$ to obtain the following weakly exact triangle:
  \[
  \xymatrix{
    LCH^{(2) \bullet}_{{\varepsilon}, \varepsilon}(\Lambda; \pi) \ar[rr] & &  LCH^{(2) \bullet}_{{\widetilde{\varepsilon}_k,} \varepsilon_k}(\Lambda; \pi)\ar[dl] \\
    & H^{(2)}_\bullet(\overline{\Sigma}^{\odot
      k},\partial_-\overline{\Sigma}^{\odot k};\pi). \ar[ul]
    & }
  \]
  The inequality
  \[ {\dim_{\ell^2}}(H^{(2)}_1(\overline{\Sigma}^{\odot
    k},\partial_-{\overline{\Sigma}^{\odot k}};\pi_1(\Sigma))) \ge
  k, \] shown in Lemma \ref{lem:rankgrowth} below, together with the
  universal bound
  \[{\dim_{\ell^2}} LCH^{(2) \bullet}_{\varepsilon'}(\Lambda) \le
  |\mathcal{R}(\Lambda)|, \] which follows by Lemma \ref{lem:l2ranks},
  finally gives  a contradiction. It thus follows that $\pi_1(\Sigma)$ is
  finite, and therefore trivial by the previous argument.
\end{proof}
\subsubsection{Estimating the first $L^2$-Betti number of a tower}
We finish the proof of Theorem \ref{thm:l2rigidity}.
\begin{Lem}
\label{lem:rankgrowth}
If the fundamental group $\pi$ is infinite, then the $L^2$-homology group $H^{(2)}_1(\overline{\Sigma}^{\odot k},\partial_-{\overline{\Sigma}^{\odot k}};\pi)$
satisfies
\[ {\dim_{\ell^2}}(H^{(2)}_1(\overline{\Sigma}^{\odot k},\partial_-{\overline{\Sigma}^{\odot k}};\pi)) \ge k.\]
\end{Lem}
\begin{proof}
 Lemma \ref{lem:l2ranks} implies that
$$ H^{(2)}_0(\partial_\pm \overline{\Sigma};\pi)=\ell^2(\pi) \:\: \text{and} \:\: H^{(2)}_1(\partial_\pm \overline{\Sigma};\pi)=0,$$
since $\partial_\pm \overline{\Sigma}$ are  connected and  simply connected. Observe that we also have
\[H^{(2)}_0(\overline{\Sigma}^{\odot k};\pi)=0, \:\:k \geq 1,\]
as follows from \cite[Theorem 1.35(8)]{L2book}, using the fact that $\pi$ is infinite. The long weakly exact sequence of a pair (Proposition~\ref{prop: L^2 exact sequences}(1)) immediately implies the base case
\[{\dim_{\ell^2}}H^{(2)}_1(\overline{\Sigma}^{\odot 1},\partial_-{\overline{\Sigma}^{\odot 1}};\pi) ={ \dim_{\ell^2}}H^{(2)}_1(\overline{\Sigma},\partial_-{\overline{\Sigma}};\pi)\ge 1\]
as well as the vanishing
\[H^{(2)}_0(\overline{\Sigma}^{\odot k},\partial_-{\overline{\Sigma}^{\odot k}};\pi)=0.\]
The Mayer-Vietoris long weakly exact sequence  (Proposition~\ref{prop: L^2 exact sequences}(2))
\begin{eqnarray*}
\lefteqn{\hdots \to H^{(2)}_1(\partial_-{\overline{\Sigma}};\pi) \to } \\
& \to & H^{(2)}_1(\overline{\Sigma};\pi) \oplus H^{(2)}_1(\overline{\Sigma}^{\odot (k-1)},\partial_-{\overline{\Sigma}^{\odot (k-1)}};\pi) \to\\
& \to & H^{(2)}_1(\overline{\Sigma}^{\odot k},\partial_-{\overline{\Sigma}^{\odot k}};\pi) \\
& \to & H^{(2)}_0(\partial_-{\overline{\Sigma}};\pi) \to 0 \to \hdots,
\end{eqnarray*}
together with $H^{(2)}_1(\partial_-{\overline{\Sigma}};\ell^2(\pi))=0$ and \cite[Theorem 1.12(2)]{L2book} gives that
\begin{eqnarray*}
\lefteqn{{\dim_{\ell^2}} H^{(2)}_1(\overline{\Sigma}^{\odot k},\partial_-{\overline{\Sigma}^{\odot k}};\pi) \ge } \\
& \ge & {\dim_{\ell^2}} H^{(2)}_1(\overline{\Sigma}^{\odot (k-1)},\partial_-{\overline{\Sigma}^{\odot (k-1)}};\pi)+{\dim_{\ell^2}} H^{(2)}_0(\partial_-{\overline{\Sigma}};\pi).
\end{eqnarray*}
Since ${\dim_{\ell^2}} H^{(2)}_0(\partial_-{\overline{\Sigma}};\pi)=1$, the claim now follows by induction.
\end{proof}
\subsection{Explicit examples of Lagrangian cobordisms}
\label{sec:some-expl-lagr}
In this subsection we discuss some examples to illustrate the applications of this section. We start by recalling a few general constructions of Legendrian submanifolds and exact Lagrangian cobordisms.
\subsubsection{A Legendrian ambient surgery on the front-spin}
\label{sec:frontspin}
The front $S^m$-spinning construction described in \cite{NoteSpin} by the fourth author constructs a Legendrian embedding $\Sigma_{S^m}\Lambda \subset (\R^{2(m+n)+1},\xi_{\OP{std}})$ of $S^m \times \Lambda$, given a Legendrian embedding $\Lambda \subset (\R^{2n+1},\xi_{\OP{std}})$. In the same article, it was also shown that the same construction can be applied to an exact Lagrangian cobordism $\Sigma \subset \R \times \R^{2n+1}$ from $\Lambda^-$ to $\Lambda^+$ inside the symplectisation, producing an exact Lagrangian cobordism $\Sigma_{S^m}\Sigma \subset \R \times \R^{2(n+m)+1}$ from $\Sigma_{S^m}\Lambda^-$ to $\Sigma_{S^m}\Lambda^+$ that is diffeomorphic to $S^m \times \Sigma$.

Consider a Legendrian knot $\Lambda \subset (\R^3,\xi_{\OP{std}})$. Its left-most cusp edge in the front projection for a generic representative corresponds to a cusp edge diffeomorphic to $S^m$ in the front projection of the front spin $\Sigma_{S^m}\Lambda \subset (\R^{2m+3},\xi_{\OP{std}})$. Moreover, this cusp edge bounds an obvious embedding of an isotropic $(m+1)$-disc $D \subset (\R^{2(m+n)+1},\xi_{\OP{std}})$ whose interior is disjoint from $\Sigma_{S^m}\Lambda$, while its boundary coincides with this cusp edge; see Figure \ref{fig:beforesurgery}.

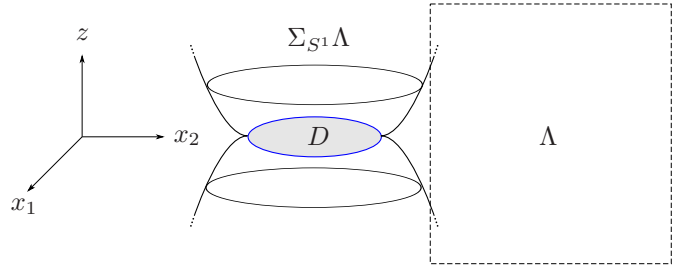
\begin{figure}[ht!]
  \centering
  \labellist
  \pinlabel $z$ at 33 138
  \pinlabel $x_1$ at -2 35
  \pinlabel $x_2$ at 96 76
  \pinlabel $D$ at 175 77
  \pinlabel $\Sigma_{S^1}\Lambda$ at 175 135
  \pinlabel $\Lambda$ at 313 77
  \endlabellist
  \includegraphics[height=3.5cm]{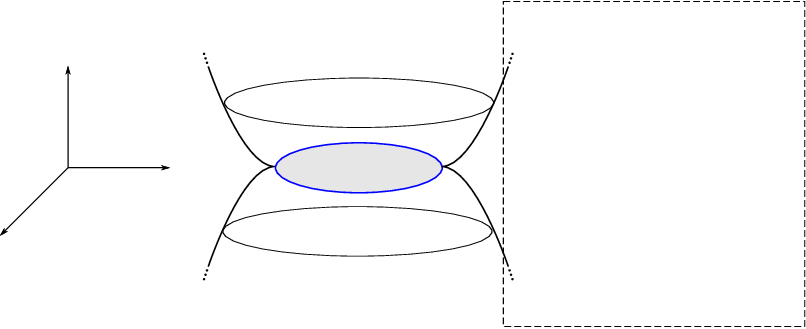}
  \caption{The front projection of the front spin $\Sigma_{S^1}\Lambda \subset (\R^5,\xi_{\OP{std}})$ near the left-most cusp of $\Lambda \subset(\R^3,\xi_{\OP{std}}) $. The corresponding cusp-edge for the front projection of the front spin bounds an obvious embedded Legendrian disc $D$ intersecting $\Sigma_{S^1}\Lambda$ cleanly along this cusp edge.}
  \label{fig:beforesurgery}
\end{figure}

A Legendrian ambient $m$-surgery, described in \cite{LegendrianAmbient} by the second author, can be performed on the sphere $S^m \hookrightarrow \Sigma_{S^m}\Lambda$ corresponding to the cusp edge $\partial D$, utilising the bounding Legendrian disc $D$. The Legendrian submanifold $\Lambda^+ \subset (\R^{2(m+n)+1},\xi_{\OP{std}})$ resulting from the surgery has the front projection shown in Figure \ref{fig:aftersurgery} in the case of $m=1=\dim \Lambda$. Recall that there also is a corresponding elementary Lagrangian $(m+1)$-handle attachment, which is an exact Lagrangian cobordism from $\Sigma_{S^m}\Lambda$ to the Legendrian submanifold $\Lambda^+$ obtained after the surgery. Topologically, this cobordism is simply the handle attachment corresponding to the surgery.

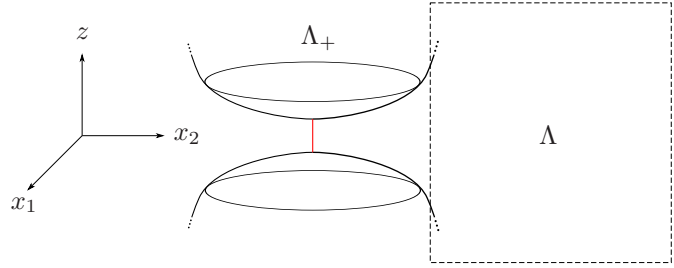
\begin{figure}[ht!]
  \centering
  \labellist
  \pinlabel $z$ at 33 138
  \pinlabel $x_1$ at -2 35
  \pinlabel $x_2$ at 96 76
  \pinlabel $\Lambda_+$ at 175 135
  \pinlabel $\Lambda$ at 313 77
  \endlabellist
  \includegraphics[height=3.5cm]{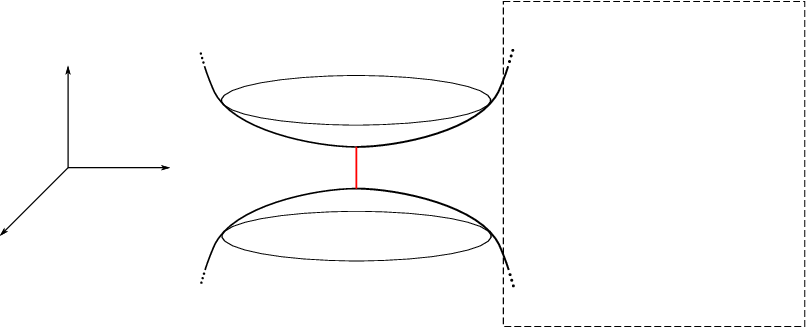}
  \caption{The front projection of the Legendrian submanifold $\Lambda^+ \subset (\R^5,\xi_{\OP{std}})$ obtained after a Legendrian ambient surgery on the front spin $\Sigma_{S^1}\Lambda \subset (\R^5,\xi_{\OP{std}})$, utilising the Legendrian disc $D$ as shown in Figure \ref{fig:beforesurgery}.}
  \label{fig:aftersurgery}
\end{figure}

\subsubsection{Non-simply connected exact Lagrangian fillings of Legendrian spheres (the proof of Proposition \ref{prop:example})}
Using the constructions in Section \ref{sec:frontspin}, the sought examples will not be difficult to produce. We start with a Legendrian knot $\Lambda\subset(\mathbb{R}^3,\xi_{\OP{std}})$ which admits a non-simply connected Lagrangian filling $\Sigma$. For instance, we can take the Legendrian right handed trefoil knot and one of its exact Lagrangian filling diffeomorphic to a punctured torus; see \cite{Ekhoka}. It follows that $\Sigma_{S^m}\Lambda \subset (\R^{2(m+1)+1},\xi_{\OP{std}})$ is a Legendrian $S^m \times S^1$ which admits an exact Lagrangian filling $\Sigma_{S^m}\Sigma$ diffeomorphic to $S^m \times \Sigma$; this filling is of course also not simply connected.

The Legendrian ambient surgery along a cusp-edge in the class $S^m \times \{p\}$ for $p \in \Lambda$ corresponding to the left-most cusp edge of $\Lambda\subset(\mathbb{R}^3,\xi_{\OP{std}})$ as described above produces a Legendrian sphere $\Lambda'$, and concatenating $\Sigma_{S^m}\Sigma$ with the corresponding elementary Lagrangian $(m+1)$-handle provides a non-simply connected filling $\Sigma'$ of $\Lambda'$.
\begin{Rem} Theorem \ref{thm:l2rigidity} rules out non-simply connected endocobordisms of $\Lambda'$. However, the existence of the non-simply connected filling implies that Theorem \ref{thm:pi_1carclass} does not apply to $\Sigma'$. By using Remark \ref{rem:uniqueaug}, one thus sees that this Legendrian sphere necessarily has Reeb chords in degree zero, and that it moreover admits at least two distinct augmentations.
\end{Rem}

\subsubsection{Non-invertible Lagrangian concordances}
\label{sec:noninv}
Here we will prove the statement that
\begin{Prop}
\label{prop:noninvertible}
In all contact spaces $(\R^{2n+1},\xi_{\OP{std}})$ with $n \ge 1$ there exists a Legendrian $n$-sphere $\Lambda$ of $\mathtt{tb}=(-1)^{n(n-1)/2+1}$ which is fillable by a Lagrangian disc, but which admits no Lagrangian concordance to the standard Legendrian sphere $\Lambda_0$ of $\mathtt{tb}=(-1)^{n(n-1)/2+1}$. (Recall that the filling induces a Lagrangian concordance from $\Lambda_0$ to $\Lambda$.)
\end{Prop}
In \cite{Chantraine_Non_symmetry} the first author proved that the relation of Lagrangian concordance is not symmetric by establishing the above proposition in the case $n=1$. More precisely, it was shown there that the Legendrian representative $\Lambda_{9_{46}} \subset (\R^3,\xi_{\OP{std}})$ of the knot $9_{46}$ as depicted in Figure \ref{fig:946} (satisfying $\mathtt{tb}=-1$; this is maximal for this smooth knot class), which is fillable by a Lagrangian disc, is not concordant to the standard Legendrian unknot $\Lambda_0$ of $\mathtt{tb}=-1$.

Recall that an exact Lagrangian filling by a disc can be used to construct a concordance $C$ from $\Lambda_0$ to $\Lambda_{9_{46}}$, which was explicitly described in the same article. One such concordance is described in Figure \ref{fig:concordance946} below. Note that along the entire concordance the leftmost cusp-edge $p$ is fixed, and so we can assume that the cylinder $C$ coincides with the trivial cylinder $\mathbb{R}\times l$ for a small arc $p \in l \subset \Lambda_{9_{46}}$ inside a neighbourhood of this cusp. This fact will be important below.

Using the results in the current article, the non-existence of a concordance from $\Lambda_{9_{46}}$ to $\Lambda_0$ can be reproved by applying Corollary \ref{cor:concobstruction} together with the calculations in \cite{Chantraine_Non_symmetry}. Namely, in the latter article it is shown that, for an appropriate pair $\varepsilon_0,\varepsilon_1$ of augmentations of the Chekanov-Eliashberg algebra of $\Lambda_0$, we have
\[ LCH_{-1}^{\varepsilon_0,\varepsilon_1}(\Lambda_{9_{46}}) \neq 0,\]
and no concordance going the other way can thus exist by Corollary \ref{cor:concobstruction}.

The front spinning construction produces exact Lagrangian concordances $\Sigma_{S^m}C \subset \R \times \R^{3+2m}$, obtained as the front spin of $C$, from $\Sigma_{S^m}\Lambda_0 \subset (\R^{3+2m},\xi_{\OP{std}})$ to $\Sigma_{S^m}\Lambda_{9_{46}} \subset (\R^{3+2m},\xi_{\OP{std}})$. Here, the latter Legendrian submanifolds are the front spins of $\Lambda_0$ and $\Lambda_{9_{46}}$, respectively. In \cite[Section 5]{Floer_Conc} the authors proved using the K\"{u}nneth formula in Floer homology that again
\[ LCH_{-1}^{\widetilde{\varepsilon}_0,\widetilde{\varepsilon}_1}(\Sigma_{S^m}\Lambda_{9_{46}}) \neq 0\]
holds for a suitable pair of augmentations, which together with Corollary \ref{cor:concobstruction} implies that there is no Lagrangian concordance from $\Sigma_{S^m}\Lambda_{9_{46}}$ to $\Sigma_{S^m}\Lambda_0$.

\begin{figure}[ht!]
  \centering
  \includegraphics[height=2.2cm]{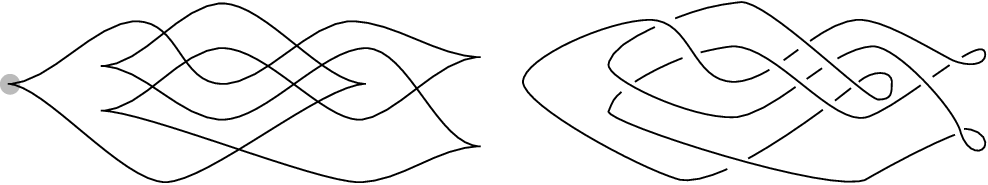}
  \caption{Front (left) and Lagrangian (right) projections of the maximal TB $m(9_{46})$ knot.}
  \label{fig:946}
\end{figure}

\begin{figure}[ht!]
  \centering
  \includegraphics[height=7cm]{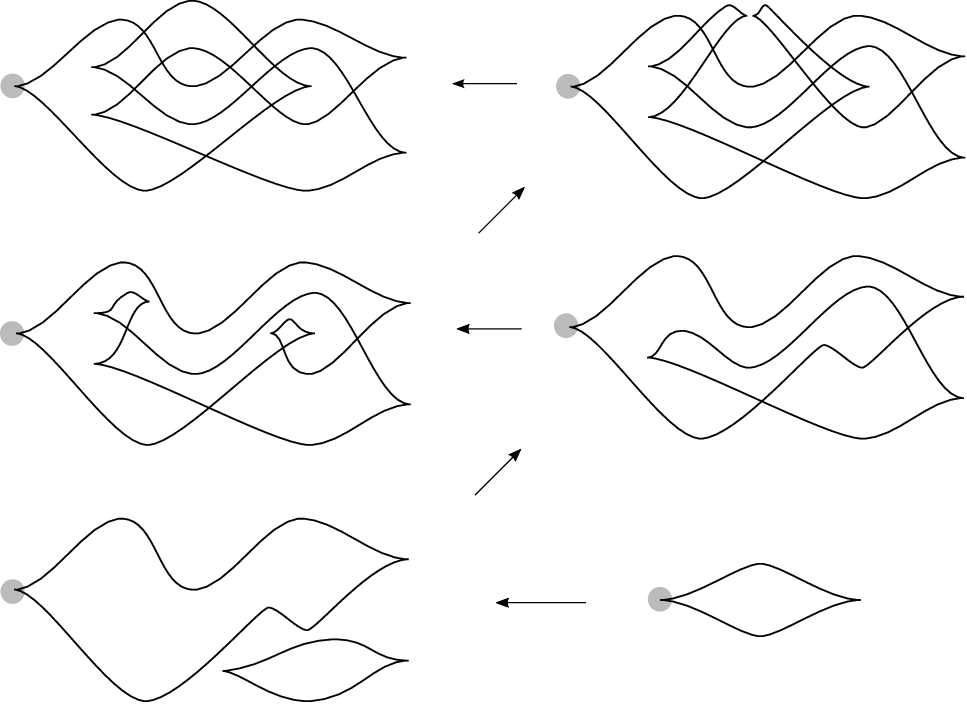}
  \caption{A Lagrangian concordance from $\Lambda_0$ to $\Lambda_{9_{46}}$.}
  \label{fig:concordance946}
\end{figure}

Recall that $\Sigma_{S^m}\Lambda_0 {\cong} \Sigma_{S^m}\Lambda_{9_{46}} {\cong} S^m \times S^1$, while $\Sigma_{S^m}C {\cong} \R \times S^m \times S^1$. We will now perform an explicit modification of the above example to produce an example of Legendrian \emph{spheres} in all dimensions which admit a concordance \emph{from} the standard sphere, but which do not admit a concordance \emph{to} the standard sphere; this establishes Proposition \ref{prop:noninvertible}.

\begin{proof}[Proof of Proposition \ref{prop:noninvertible}]
The Legendrian ambient surgery can be performed to the cusp-edge of the front projection of $\Sigma_{S^m}\Lambda_{9_{46}}$ corresponding to the left-most cusp edge $p \in \Lambda_{9_{46}}$. In this way, a Legendrian sphere $\Lambda^+ \subset (\R^{2(m+1)+1},\xi_{\OP{std}})$ is produced. Since the concordance $C$ moreover may be assumed to be a trivial cylinder over a neighbourhood of $p \in \Lambda$ and, hence, so is $\Sigma_{S^m}C$, we obtain a Lagrangian concordance from $\Lambda^-$ to $\Lambda^+$, where $\Lambda^-$ is the Legendrian sphere obtained by performing the corresponding Legendrian ambient surgery on $\Sigma_{S^m}\Lambda_0$. In fact, the latter sphere is the standard Legendrian $(m+1)$-sphere of $\mathtt{tb}=(-1)^{(m+1)m/2+1}$.

Recall that the Legendrian ambient surgery also produces an exact Lagrangian handle attachment cobordism from $\Sigma_{S^m}\Lambda_{9_{46}}$ to $\Lambda^+$. Inspecting the long exact sequence induced by Theorem \ref{thm:lespair}, we immediately conclude that there are augmentations $\varepsilon^+_i$, $i=0,1$ for the Chekanov-Eliashberg algebra of the Legendrian sphere $\Lambda^+$ satisfying
\[ LCH_{-1}^{\varepsilon^+_0,\varepsilon^+_1}(\Lambda^+) {\cong} LCH_{-1}^{\widetilde{\varepsilon}_0,\widetilde{\varepsilon}_1}(\Sigma_{S^m}\Lambda_{9_{46}}) \neq 0.\]
Once again, Corollary \ref{cor:concobstruction} shows that there exists no concordance from $\Lambda^+$ to $\Lambda^-$.

\end{proof}

%\bibliographystyle{plain}
%\bibliography{Bibliographie_en}

\end{document}